\documentclass[reqno, twoside, a4paper]{amsart}

\usepackage{amsmath}
\usepackage{amssymb}
\usepackage{amsthm}
\usepackage{graphicx}
\usepackage{subfig}
\usepackage{enumerate}
\usepackage[all]{xy}
\usepackage{hyperref}

\usepackage{comment}
\usepackage{xspace}
\usepackage{mathtools}

\usepackage{tikz}

\usetikzlibrary{patterns,decorations.pathreplacing}

\tikzset{bullet/.style={
shape = circle,fill = black, inner sep = 0pt, outer sep = 0pt, minimum size = 0.35em, line width = 0pt, draw=black!100}}

\tikzset{circle/.style={
shape = circle,fill = none, inner sep = 0pt, outer sep = 0pt, minimum size = 0.35em, line width = 1pt, draw=black!100}}

\tikzset{rectangle/.style={
shape = rectangle,fill = white, inner sep = 0pt, outer sep = 0pt, minimum size = 0.35em, line width = 0pt, draw=black!100}}

\tikzset{empty/.style={
shape = circle,fill = white, inner sep = 0pt, outer sep = 0pt, minimum size = 0.35em, line width = 0pt, draw=white!100}}

\tikzset{xmark/.style={
shape = x,fill = white, inner sep = 0pt, outer sep = 0pt, minimum size = 0em, line width = 0pt, draw=white!100}}

\tikzset{longrectangle/.style={
inner sep = 1em,
rectangle,
minimum size=1em,
very thick,
draw=black!100, 
}}

\tikzset{label distance=-0.15em}

\tikzset{font=\scriptsize}

\usepackage{tikz-cd}
\usepackage{tkz-graph}

\newtheorem{theorem}{Theorem}[section]
\newtheorem{lemma}[theorem]{Lemma}
\newtheorem{proposition}[theorem]{Proposition}
\newtheorem{corollary}[theorem]{Corollary}
\newtheorem{maintheorem}{Main Theorem}

\theoremstyle{definition}

\newtheorem{definition}[theorem]{Definition}
\newtheorem{remark}[theorem]{Remark}

\newtheorem{example}[theorem]{Example}

\newtheorem*{convention}{Convention}

\numberwithin{equation}{section}

\def\sheaf#1{\ensuremath \mathcal#1}

\newcommand{\abs}[1]{\ensuremath \left\lvert #1 \right\rvert}
\newcommand{\SpinC}{\ensuremath \mathrm{Spin}^{c}}

\DeclarePairedDelimiter{\ceil}{\lceil}{\rceil}

\DeclareMathOperator{\Spec}{Spec}
\DeclareMathOperator{\Proj}{Proj}

\DeclareMathOperator{\Def}{Def}
\DeclareMathOperator{\DefQG}{Def^{\mathrm{QG}}}
\DeclareMathOperator{\adm}{adm}

\begin{document}

\title[Milnor fibers and symplectic fillings]{Milnor fibers and symplectic fillings of quotient surface singularities}

\author[H. Park]{Heesang Park}

\address{Department of Mathematics, Konkuk University, Seoul 05029, Korea}

\email{HeesangPark@konkuk.ac.kr}

\author[J. Park]{Jongil Park}

\address{Department of Mathematical Sciences, Seoul National University, Seoul 08826, Korea \& Korea Institute for Advanced Study, Seoul 130-722, Korea}

\email{jipark@snu.ac.kr}

\author[D. Shin]{Dongsoo Shin}

\address{Department of Mathematics, Chungnam National University, Daejeon 34134, Korea}

\email{dsshin@cnu.ac.kr}

\author[G. Urz\'ua]{Giancarlo Urz\'ua}

\address{Facultad de Matem\'aticas, Pontificia Universidad Cat\'olica de Chile, Santiago, Chile}

\email{urzua@mat.puc.cl}

\subjclass[2010]{14B07, 53D35}

\keywords{Milnor fiber, quotient surface singularity, symplectic filling}

\date{July 24, 2015; revised at May 16, 2016, at Feb 12, 2018}

\begin{abstract}
We determine a one-to-one correspondence between Milnor fibers and minimal symplectic fillings of a quotient surface singularity (up to diffeomorphism type) by giving an explicit algorithm to compare them mainly via techniques from the minimal model program for 3-folds and Pinkham's negative weight smoothing. As by-products, we show that:

-- Milnor fibers associated to irreducible components of the reduced versal deformation space of a quotient surface singularity are not diffeomorphic to each other with a few obvious exceptions. For this, we classify minimal symplectic fillings of a quotient surface singularity up to diffeomorphism.

-- Any symplectic filling of a quotient surface singularity is obtained by a sequence of rational blow-downs from a special resolution (so-called the maximal resolution) of the singularity, which is an analogue of the one-to-one correspondence between the irreducible components of the reduced versal deformation space and the so-called $P$-resolutions of a quotient surface singularity.
\end{abstract}

\maketitle

\tableofcontents

\section{Introduction}

We show that two ``smooth'' objects --- Milnor fibers and (minimal) symplectic fillings --- associated to a quotient surface singularity $(X,0)$ coming from algebraic geometry and symplectic topology, respectively, are essentially the same (up to diffeomorphism type). For this, we provide an explicit algorithm to identify Milnor fibers of a quotient surface singularity with minimal symplectic fillings of its link using special partial resolutions (called $P$-resolutions) and the minimal model program for 3-folds. Minimal symplectic fillings are realized as complements of certain divisors embedded in rational surfaces. So the algorithm determines how such divisors are embedded in rational surfaces for given Milnor fibers; See Section~\ref{section:identifying-Milnor-fibers}. Conversely, we prove that every minimal symplectic filling of a quotient surface singularity is diffeomorphic to the Milnor fiber of a smoothing of the singularity by constructing a smoothing whose Milnor fiber is diffeomorphic to the given minimal symplectic filling; Section~\ref{section:minimal-symplectic-fillings-are-Milnor-fibers}. For this, we apply Pinkham's theory of smoothings of negative weight.

As an explicit one-to-one correspondence (up to diffeomorphism type) is established between Milnor fibers and minimal symplectic fillings, one may apply results regarding Milnor fibers to minimal symplectic fillings, and vice versa.

For instance, we classify minimal symplectic fillings of a non-cyclic quotient surface singularity given in Bhupal--Ono~\cite{Bhupal-Ono-2015} up to diffeomorphism in Section~\ref{section:Diffeomorphism-type}, from which we conclude that Milnor fibers associated to irreducible components of the reduced versal deformation space of a non-cyclic quotient surface singularity are non-diffeomorphic to each other except for obviously diffeomorphic pairs because of the symmetry of the minimal resolutions of the corresponding singularities (cf. Proposition~\ref{proposition:reduced-BO-list}); Theorem~\ref{theorem:components-Def(X)}. On the other hand, it has been known that every Milnor fiber of a quotient surface singularity is given as a smoothing of a certain partial resolution (so-called \emph{$P$-resolution}) and every $P$-resolution is dominated by a special resolution (so-called the \emph{maximal resolution}) of the singularity; cf.~KSB~\cite{Kollar-Shepherd-Barron-1988}. Then we show that every minimal symplectic filling of a quotient surface singularity can be constructed from the maximal resolution via rational blow-down surgery; Theorem~\ref{theorem:minimal-symplectic-filling-maximal-resolution}.

\subsection{Milnor fibers and symplectic fillings}

We recall briefly some relevant notions. Let $(X,0)=(\mathbb{C}^2/G,0)$ be a germ of a quotient surface singularity, where $G$ is a small finite subgroup of $GL(2, \mathbb{C})$. A \emph{smoothing} of $(X,0)$ is a proper flat map $\pi \colon \mathcal{X} \to \Delta$, where $\Delta=\{ t \in \mathbb{C} : \abs{t} < \epsilon\}$, from a threefold isolated singularity $(\mathcal{X}, 0)$ such that $(\pi^{-1}(0), 0) \cong (X, 0)$ and $\pi^{-1}(t)$ is smooth for every $t \neq 0$. The \emph{Milnor fiber} $M$ of a smoothing $\pi$ of $(X,0)$ is a general fiber $\pi^{-1}(t)$ ($0 < t \ll \epsilon$).

Assume that $(X, 0) \subset (\mathbb{C}^N, 0)$, which is always possible. If $B \subset \mathbb{C}^N$ is a small ball centered at the origin, then the small neighborhood $X \cap B$ of the singularity is contractible and homeomorphic to the cone over its boundary $L := X \cap \partial B$. The smooth compact 3-manifold $L$ is called the \emph{link} of the singularity. The topology of the germ $(X, 0)$ is completely determined by its link $L$. The link $L$ admits a natural contact structure $\xi_{\textrm{st}}$, so-called \emph{Milnor fillable contact structure}, where a \emph{contact structure} on a 3-manifold is a two-dimensional distribution $\xi$ given as the kernel of a one-form $\alpha$ such that $\alpha \wedge d\alpha$ is a volume form. The Milnor fillable contact structure $\xi_{\textrm{st}}$ on $L$ is defined by complex tangency of the complex structure $J$ along $L$, that is, those tangent planes to $L$ that are complex with respect to the complex structure near $L$; i.e., $\xi_{\textrm{st}}=\ker{\alpha_{\textrm{st}}}=TL \cap JTL$. A \emph{(strong) symplectic filling} of $(X, 0)$ is a  symplectic 4-manifold $(W, \omega)$ with the boundary $\partial W=L$ satisfying the compatibility condition $\omega=d\alpha_{\textrm{st}}$. One may also define a so-called \emph{weak} symplectic filling. But it is known that two notions of symplectic fillings coincide in our case because the link $L$ is a rational homology sphere. So we simply call them \emph{symplectic fillings}. Finally a \emph{Stein filling} of $(X,0)$ is a Stein manifold $W$ with $L$ as its strictly pseudoconvex boundary and $\xi_{\textrm{st}}$ is the set of complex tangencies to $L$. It is clear that Stein fillings are minimal symplectic fillings.

Minimal symplectic fillings are classified by Lisca~\cite{Lisca-2008} for cyclic quotient singularities and by Bhupal--Ono~\cite{Bhupal-Ono-2012} for non-cyclic quotient singularities. According their results, any minimal symplectic fillings of quotient surface singularities are given as complements $Z - E_{\infty}$ of the so-called \emph{compactifying divisor} $E_{\infty}$ embedded in a smooth rational $4$-manifold $Z$, where $E_{\infty}$ is a collection of symplectic 2-spheres depending only on the singularity $X$ itself, not on symplectic fillings; See Definitions~\ref{definition:natural-compactification-cyclic}, \ref{definition:natural-compactification-non-cyclic}.

On the other hand, according to the general theory of Milnor fibrations (see Looijenga~\cite{Looijenga-1984} for example), the Milnor fiber $M$ is a compact 4-manifold with the link $L$ as its boundary and the diffeomorphism type of $M$ depends only on the smoothing $\pi$, and, indeed, only on the irreducible component of $\Def(X,0)$ that contains the smoothing $\pi$. Furthermore $M$ has a natural Stein (and hence symplectic) structure, and so it provides a natural example of a Stein (hence minimal symplectic) filling of $L$.

Therefore it would be an intriguing problem to compare Milnor fibers and minimal symplectic fillings. In particular there are two basic questions: How can one identify Milnor fibers as minimal symplectic fillings, that is, as complements $Z - E_{\infty}$ in the lists of Lisca~\cite{Lisca-2008} and Bhupal--Ono~\cite{Bhupal-Ono-2012}? And is every minimal symplectic filling obtained from a Milnor fiber? that is, for any given minimal symplectic filling, is there a Milnor fiber which is diffeomorphic to the minimal symplectic filling?

\subsection{Classification of symplectic fillings}

First of all we classify symplectic fillings up to diffeomorphism, which is also one of the fundamental problems in contact and symplectic geometry.

For a \emph{cyclic quotient surface singularity} (that is, $G$ is a finite cyclic group), McDuff~\cite{McDuff-1990} classifies symplectic deformation classes of minimal symplectic fillings of cyclic quotient surface singularities of type $\frac{1}{n}(1,1)$. Ohta--Ono~\cite{Ohta-Ono-2005} investigates symplectic fillings of $A_n$-singularities. Then Lisca~\cite{Lisca-2008} presents a complete classification of symplectic fillings of any cyclic quotient surface singularities up to orientation-preserving diffeomorphism. Indeed his classification is up to orientation-preserving homeomorphism.

We briefly review Lisca's classification. For details, see Section~\ref{section:minimal-symplectic-fillings}. Let $(X,0)$ be a cyclic quotient surface singularity of type $\frac{1}{n}(1,a)$ with $(n,a)=1$. The link of $X$ is the Lens space $L(n,a)$. Lisca~\cite{Lisca-2008} parametrizes minimal symplectic fillings of $(X,0)$ by a set $K_e(n/n-a)$ of certain sequences of integers $\underline{k}=(k_1, \dotsc, k_e) \in \mathbb{N}^e$ (see Definition~\ref{definition:K_e(n/n-a)}). That is, by surgery diagrams, he constructs compact oriented 4-manifolds $W_{n,a}(\underline{k})$ with boundary $L(n,a)$ which are parametrized by $\underline{k} \in K_e(n/n-a)$ and shows that $W_{n,a}(\underline{k})$ is a Stein filling of $L(n,a)$. Finally he proves that any symplectic filling of $L(n,a)$ is orientation-preserving diffeomorphic to a manifold obtained from one of the $W_{n,a}(\underline{k})$ by a composition of blow-ups; so every minimal symplectic filling is diffeomorphic to a $W_{n,a}(\underline{k})$.

For  \emph{non-cyclic quotient surface singularities}, Ohta--Ono~\cite{Ohta-Ono-2005} classifies symplectic fillings of non-cyclic ADE singularities, that is, $D_n, E_6, E_7, E_8$ singularities. They show that there is only one diffeomorphism type of symplectic fillings for each non-cyclic ADE singularities. Then Bhupal--Ono~\cite{Bhupal-Ono-2012} provides a list of all possible minimal symplectic fillings for all non-cyclic quotient surface singularities.

As mentioned above, according to Bhupal--Ono~\cite{Bhupal-Ono-2012}, any minimal symplectic filling $W$ of a non-cyclic quotient surface singularity $X$ is orientation-preserving diffeomorphic to the complement $Z - \nu(E_{\infty})$ of a regular neighborhood $\nu(E_{\infty})$ of the compactifying divisor $E_{\infty}$ of $X$ embedded in a certain rational symplectic 4-manifold $Z$. The rational symplectic 4-manifold $Z$ is called the \emph{compactification} of $W$. Bhupal--Ono~\cite{Bhupal-Ono-2012} presents a complete list of ways of constructing $(Z \supset E_{\infty})$ up to symplectic deformation equivalence from $\mathbb{CP}^2$ or $\mathbb{CP}^1 \times \mathbb{CP}^1$ by successive blow-ups.

On the other hand, there are duplicate entries in the list of Bhupal--Ono~\cite{Bhupal-Ono-2012} coming from the same minimal symplectic fillings, we first remove the duplications in Proposition~\ref{proposition:reduced-BO-list}. Then, in Section~\ref{section:Diffeomorphism-type}, we classify minimal symplectic fillings of non-cyclic quotient surface singularities up to orientation-preserving diffeomorphism.

\begin{maintheorem}[Theorem~\ref{theorem:diffeomorphism-type}, Corollary~\ref{corollary:no-exotic-fillings}]
Any two minimal symplectic fillings of a non-cyclic quotient surface singularity in the reduced list of Bhupal-Ono (cf. Proposition~\ref{proposition:reduced-BO-list}) are not orientation-preserving diffeomorphic to each other. Furthermore there are no exotic symplectic fillings for a non-cyclic quotient surface singularity.
\end{maintheorem}

Combined with the classification of Lisca~\cite{Lisca-2008} for symplectic fillings of cyclic quotient surface singularities, one can conclude that there are no exotic symplectic fillings for any quotient surface singularity.

We use a similar strategy of Lisca~\cite{Lisca-2008}. Let $W_1$ and $W_2$ be two minimal symplectic fillings of a non-cyclic quotient surface singularity $X$. Suppose that $W_i$ is orientation-preserving diffeomorphic to $Z_i - \nu(E_{\infty})$, where $Z_i$ is a rational 4-manifold. We show that any diffeomorphism $\phi \colon W_1 \to W_2$ (if any) can be extended to a diffeomorphism $\overline{\phi} \colon Z_1 \to Z_2$ such that $\overline{\phi}$ preserves the compactifying divisor $E_{\infty}$. Since $W_i = Z_i - E_{\infty}$ is assumed to be minimal, every $(-1)$-curve in $Z_i$ should intersect with $E_{\infty}$ in $Z_i$. Since $\overline{\phi}$ also preserves $(-1)$-curves, we show that the diffeomorphism type of a minimal symplectic filling $W$ of a non-cyclic quotient surface singularity is completely determined by the data of the intersections of $(-1)$-curves in $Z$ with $E_{\infty}$. It is easy to check that the positions of $(-1)$-curves are all different for each minimal symplectic filling in the reduced list of Bhupal--Ono~\cite{Bhupal-Ono-2012} (for a few exceptions for which we can easily handle the diffeomorphism type problem), which proves the first part of Main Theorem~1. Furthermore, in Corollary~\ref{corollary:no-exotic-fillings}, we show that the above proof can be easily extended to the case of orientation-preserving homeomorphisms; hence there are no exotic symplectic fillings as asserted.

\subsection{Explicit correspondence: From Milnor fibers to symplectic fillings}

We provide an explicit algorithm for identifying a given Milnor fiber as a minimal symplectic filling, that is, as complements $Z - E_{\infty}$. For this, we apply some techniques from the minimal model program for 3-folds such as divisorial contractions and flips.

We first compactify a given smoothing of $(X,0)$. We briefly sketch the idea: Let $M$ be the Milnor fiber of a smoothing $\pi \colon \mathcal{X} \to \Delta$ of a quotient surface singularity $X$. Let $Y' \to X$ be the $P$-resolution corresponding to $\pi$. According to Behnke--Christophersen~\cite{Behnke-Christophersen-1994}, there is a special partial resolution, so called, \emph{$M$-resolution} $Y \to X$ dominating $Y'$ (See Definition~\ref{definition:$M$-resolution}), and a $\mathbb{Q}$-Gorenstein smoothing $\phi \colon \mathcal{Y} \to \Delta$ such that the smoothing $\phi$ blows down to $\pi$. We have a commutative diagram as described in Figure~\ref{figure:introduction-diagram}. It is easy to show that a general fiber $Y_t = \phi^{-1}(t)$ is isomorphic to $X_t = \pi^{-1}(t)$. Therefore we have
\begin{equation*}
M = Y_t.
\end{equation*}

\begin{figure}
\centering

\subfloat [$M$-resolutions]{\includegraphics{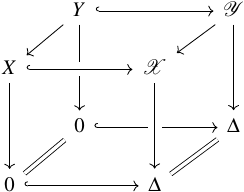}}
\qquad \qquad
\subfloat [{compactifications}]{\includegraphics{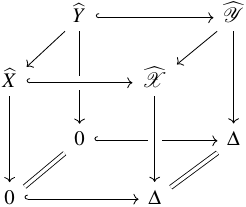}}

\caption{Deformations of partial resolutions and natural compactifications}

\label{figure:introduction-diagram}
\end{figure}

We then compactify $X$ and $Y$ to compact complex surfaces, following Lisca~\cite{Lisca-2008} and Pinkham~\cite{Pinkham-1978}, so that the two smoothings $\mathcal{Y} \to \Delta$ and $\mathcal{X} \to \Delta$ can be extended to the deformations $\widehat{\smash[b]{\mathcal{Y}}} \to \Delta$ and $\widehat{\smash[b]{\mathcal{X}}} \to \Delta$ of the \emph{natural compactifications} $\widehat{X}$ and $\widehat{Y}$ of $X$ and $Y$, where $\widehat{X}$ and $\widehat{Y}$ are obtained, roughly speaking, by pasting a regular neighborhood $\nu(E_{\infty})$ of the compactifying divisor $E_{\infty}$ of $X$ (See Section~\ref{section:Milnor-fiber-as-complements} for the definition of the natural compactification). We again have a commutative diagram, as described in Figure~\ref{figure:introduction-diagram}.

The deformations  $\widehat{\smash[b]{\mathcal{Y}}} \to \Delta$ and $\widehat{\smash[b]{\mathcal{X}}} \to \Delta$ are locally trivial along $E_{\infty}$. So the Milnor fiber $Y_t$ is given as the complement of the compactifying divisor $E_{\infty}$ in a general fiber $\widehat{Y}_t$ (which is called a \emph{compactified Milnor fiber}; See Definition~\ref{definition:compactified-Milnor-fiber}) of the deformation $\widehat{\smash[b]{\mathcal{Y}}} \to \Delta$; See Proposition~\ref{proposition:Milnor-fiber=complement}. So we have
\begin{equation*}
M = Y_t = \widehat{Y}_t - E_{\infty}.
\end{equation*}

We need to recognize $(\widehat{Y}_t, E_{\infty})$ as $(Z, E_{\infty})$ in the lists of Lisca~\cite{Lisca-2008} and Bhupal--Ono~\cite{Bhupal-Ono-2012}. For this, we show that

\begin{maintheorem}[Theorem~\ref{theorem:smoothable-by-flips}]
By applying specific divisorial contractions and flips in a controlled and explicit manner to the deformation $\widehat{\smash[b]{\mathcal{Y}}} \to \Delta$, we obtain a new deformation $\mathcal{W} \to \Delta$ such that all of its fibers are smooth.
\end{maintheorem}

During this minimal model program (in short, \textit{MMP}) process, we can track down special $(-1)$-curves appearing in the general fiber, and at the end, since all fibers of $\mathcal{W} \to \Delta$ are smooth, we can also see $(-1)$-curves in the general fiber $W_t$ coming from $(-1)$-curves in $W_0$. Using this fact, we show that a general fiber $\widehat{Y}_t$ can be obtained from an explicit general fiber $W_t$ by blowing up several times in a specific manner. In this way, we get an explicit data of the intersections of $(-1)$-curves in $\widehat{Y}_t$ with $E_{\infty}$, and so we can identify $(\widehat{Y}_t, D)$ as $(Z,E_{\infty})$ in the lists of Lisca~\cite{Lisca-2008} and Bhupal--Ono~\cite{Bhupal-Ono-2012}.

As a byproduct, this explicit MMP gives another proof of Stevens' result \cite{Stevens-1991} on the one-to-one correspondence between the set of $P$-resolutions of a cyclic quotient surface singularity $(X,0)$ of type $\frac{1}{n}(1,a)$  and the set of zero continued fractions $K_e(n/n-a)$ described in Definition~\ref{definition:K_e(n/n-a)}; see Corollary~\ref{corollary:New-proof-of-Stevens-theorem}. More importantly, this shows a geometric way to connect directly Lisca's \cite{Lisca-2008} and Koll\'ar--Shepherd-Barron's \cite{Kollar-Shepherd-Barron-1988} one-to-one correspondences. By N\'emethi--Popescu-Pampu \cite{Nemethi-PPampu-2010}, it also connects Christophersen--Stevens' \cite{Christophersen-1991,Stevens-1991} correspondence, together with the one induced by the $M$-resolutions of Behnke--Christophersen \cite{Behnke-Christophersen-1994}, and in particular we answer the question raised in N{\'e}methi--Popescu-Pampu~\cite[\S11.2]{Nemethi-PPampu-2010}.

Furthermore, as we classified minimal symplectic fillings of quotient surface singularities up to diffeomorphism, the above explicit identification of Milnor fibers induces the classification of Milnor fibers up to diffeomorphism as follows:

\begin{maintheorem}[Theorem~\ref{theorem:components-Def(X)}]
The Milnor fibers associated to irreducible components of the reduced semi-universal deformation space of a quotient surface singularity are non-diffeomorphic to each other.
\end{maintheorem}

\subsection{Explicit correspondence: From minimal symplectic fillings to Milnor fibers}

The question is that, for a given minimal symplectic filling, there is a Milnor fiber which is diffeomorphic to the minimal symplectic filling.

For cyclic quotient surface singularities, Ohta--Ono~\cite{Ohta-Ono-2005} proves that a minimal symplectic filling of an $A_n$-singularity is diffeomorphic to its Milnor fiber. Then N{\'e}methi--Popescu-Pampu~\cite{Nemethi-PPampu-2010} shows that every minimal symplectic filling of a cyclic quotient surface singularity is diffeomorphic to a Milnor fiber of the singularity, and provides an explicit one-to-one correspondence between Milnor fibers and minimal symplectic fillings.

We briefly review related results. As mentioned above, minimal symplectic fillings of a cyclic quotient surface singularity of type $\frac{1}{n}(1,a)$ are parametrized by Lisca by the set $K_e(n/n-a)$ of certain sequences of integers $\underline{k}=(k_1, \dotsc, k_e) \in \mathbb{N}^e$. On the other hand Christophersen~\cite{Christophersen-1991}, Stevens~\cite{Stevens-1991}, and de Jong and van Straten~\cite{deJong-vanStraten-1998} parametrize by the same set $K_e(n/n-a)$ (but with different methods) the reduced irreducible components of the versal deformation space of $(X,0)$.
Since each component contains a smoothing because $(X,0)$ is a rational singularity, the diffeomorphism types of Milnor fibers are invariants of irreducible components of the versal deformation space of the singularity $(X,0)$. Hence Milnor fibers of $(X,0)$ are also parametrized by the same set $K_e(n/n-a)$. So Lisca~\cite[p.768]{Lisca-2008} raises the following conjecture: The Milnor fiber of the irreducible component of the reduced versal base space of the cyclic quotient surface singularity $(X,0)$ parametrized by $\underline{k} \in K_e(n/n-a)$ is diffeomorphic to $W_{n,a}(\underline{k})$. The conjecture was solved affirmatively by N{\'e}methi--Popescu-Pampu~\cite{Nemethi-PPampu-2010}. In order to identify Milnor fibers with Stein fillings, N{\'e}methi--Popescu-Pampu~\cite{Nemethi-PPampu-2010} uses the explicit equations of reduced versal base space of cyclic quotient surface singularities developed by Riemenschneider~\cite{Riemenschneider-1974} and Arndt~\cite{Arndt-1988}.

For non-cyclic quotient surface singularities, Ohta--Ono~\cite{Ohta-Ono-2005} shows that any minimal symplectic fillings of non-cyclic ADE singularities (i.e., $D_n, E_6, E_7, E_8$ singularities) are diffeomorphic to their Milnor fibers. However, the correspondence between Milnor fibers and minimal symplectic fillings of the other non-cyclic quotient singularities remained widely open.

All possible minimal symplectic fillings of non-cyclic quotient surface singularity were found by Bhupal--Ono~\cite{Bhupal-Ono-2012} as mentioned above. On the other hand, Milnor fibers are invariants of the irreducible components of the reduced versal deformation space of $X$. According to Koll\'ar--Shepherd-Barron~\cite{Kollar-Shepherd-Barron-1988}, there is a one-to-one correspondence between the irreducible components and \emph{$P$-resolutions} of $X$ (which are special partial resolutions of $X$ admitting only singularities of class $T$. See Definition~\ref{definition:$P$-resolution}). Therefore Milnor fibers of irreducible components are in one-to-one correspondence with $P$-resolutions. Stevens~\cite{Stevens-1993} determines all $P$-resolutions of quotient surface singularities.

We found that the number of $P$-resolutions in Stevens~\cite{Stevens-1993} and that of minimal symplectic fillings in Bhupal--Ono~\cite{Bhupal-Ono-2012} are the same (cf.~Remarks~\ref{remark:erratum-Bhypal-Ono}, \ref{remark:erratum-Stevens}). Hence it would be natural to raise the question: Is there a one-to-one correspondence between Milnor fibers and minimal symplectic fillings also for non-cyclic quotient surface singularities? Since a Milnor fiber is a minimal symplectic filling, the question is if every minimal symplectic filling can be realized as a Milnor fiber. We give an affirmative answer in Section~\ref{section:minimal-symplectic-fillings-are-Milnor-fibers}.

\begin{maintheorem}[Theorem~\ref{theorem:fillings-diffeomorphic-to-Milnors}]
Let $(X,0)$ be a quotient surface singularity, and let $(W, \omega)$ be a minimal symplectic filling of $(X,0)$. Then there is a smoothing $\pi \colon \mathcal{X} \to \Delta$ of $(X,0)$ such that $W$ is diffeomorphic to the Milnor fiber of $\pi$.
\end{maintheorem}

That is, every minimal symplectic filling of a given quotient surface singularity can be realized as a Milnor fiber. Therefore we show that there is a one-to-one correspondence between minimal symplectic fillings and Milnor fibers of irreducible components of the reduced versal deformation space for any quotient surface singularities (not only cyclic quotient surface singularities). Main Theorem~2 above is a generalization of the results of N{\'e}methi--Popescu-Pampu~\cite{Nemethi-PPampu-2010} for cyclic quotient surface singularities and Ohta--Ono~\cite{Ohta-Ono-2005} for ADE singularities.

For the proof, we apply Pinkham's negative weight smoothing technique \cite{Pinkham-1978}. By Lisca~\cite{Lisca-2008} and Bhupal--Ono~\cite{Bhupal-Ono-2012}, every minimal symplectic filling of a quotient surface singularity is symplectic deformation equivalent to the complement of the compactifying divisor $E_{\infty}$ in a certain rational complex surface $Z$. There are complete lists in Lisca~\cite{Lisca-2008} and Bhupal--Ono~\cite{Bhupal-Ono-2012} of $(Z, E_{\infty})$ corresponding to minimal symplectic fillings. We show that, for each pair $(Z, E_{\infty})$ in these lists, the divisor $E_{\infty}$ supports an ample divisor $E$ in $Z$. So there is an embedding of $Z$ into a projective space $\mathbb{CP}^n$ by the linear series $\abs{mE}$ for some $m \gg 0$. Taking the affine cone $C(Z)$ over $Z$, it follows by Proposition~\ref{proposition:supporting-ample-divisor} (essentially by Pinkham~\cite[Theorem~6.7]{Pinkham-1978}. See also SSW~\cite[Theorem~8.1]{Stipzicz-Szabo-Wahl-2008} and Fowler~\cite[Theorem 2.2.3]{Fowler-2014}) that a hyperplane section of $C(Z)$ through the origin is just the original singularity $(X,0)$ and, by moving hyperplane sections, we get a smoothing of $(X,0)$ whose Milnor fiber is diffeomorphic to the complement $Z-E_{\infty}$.

On the other hand, according to KSB~\cite{Kollar-Shepherd-Barron-1988}, every Milnor fiber of a quotient surface singularity $(X,0)$ can be obtained topologically by rationally blowing down the corresponding $P$-resolution (cf. Corollary~\ref{corollary:Milnor-fiber=Rational-blow-down}), and every $P$-resolution of a quotient surface singularity is dominated by the so-called \emph{maximal resolution} of the singularity, which can be obtained uniquely by an explicit sequence of blowing-ups from its minimal resolution (See Definition~\ref{definition:maximal-resolution}, Proposition~\ref{proposition:maximal-resolution-dominating}). Therefore Main Theorem 4 above implies that

\begin{maintheorem}[Theorem~\ref{theorem:minimal-symplectic-filling-maximal-resolution}]
Any minimal symplectic filling of a quotient surface singularity is obtained by a sequence of rational blow-downs from its unique maximal resolution.
\end{maintheorem}

Bhupal--Ozbagci~\cite{Bhupal-Ozbagci-2013} also proves a similar result for cyclic quotient surface singularities by using a Lefschetz fibration technique.

\subsection*{Organization}

The paper is organized as follows. After reviewing generalities on quotient surface singularities in Section~\ref{section:generalities}, we explain how to compactify a quotient surface singularity and its smoothing  in Section~\ref{section:Compactifying-divisor} primarily based on Pinkham~\cite{Pinkham-1977, Pinkham-1978}. We introduce the classifications of Lisca~\cite{Lisca-2008} and Bhupal--Ono~\cite{Bhupal-Ono-2012} on minimal symplectic fillings in Section~\ref{section:minimal-symplectic-fillings} and then we classify minimal symplectic fillings of non-cyclic quotient surface singularities up to diffeomorphism in Section~\ref{section:Diffeomorphism-type}.

We apply techniques of the minimal model program for 3-folds (introduced in Section~\ref{section:Semi-stable-MMP}) for identifying Milnor fibers as minimal symplectic fillings in Section~\ref{section:identifying-Milnor-fibers}. Indeed we show that every Milnor fiber is given as a complement of the compactifying divisor in a rational complex surface in Section~\ref{section:Milnor-fiber-as-complements}, where the rational surface is obtained by a general fiber of a smoothing of the corresponding $M$-resolution which is introduced in Section~\ref{section:$P$-resolution}. As applications, in Section~\ref{section:applications}, we provide another proof of Stevens' result \cite{Stevens-1991} on the one-to-one correspondence between the set of $P$-resolutions of a cyclic quotient surface singularity and the set of certain zero continued fractions. Furthermore we classify Milnor fibers up to diffeomorphism associated to the irreducible components of the reduced semi-universal deformation space of quotient surface singularities.

Finally, in Section~\ref{section:minimal-symplectic-fillings-are-Milnor-fibers} we establish an explicit correspondence from minimal symplectic fillings to Milnor fibers by proving that every minimal symplectic filling is diffeomorphic to a Milnor fiber. As a corollary, we prove that every minimal symplectic filling of a quotient surface singularity can be obtained from the maximal resolution by applying a sequence of rational blow-down surgeries.

\subsection*{Acknowledgements}

The authors would like to thank Selman Akbulut, Mohan Bhupal, Kaoru Ono, Kyungbae Park, Patrick Popescu-Pampu, and Jan Stevens for their kind discussion. HP and DS thank Korea Institute for Advanced Study when they were associate members in KIAS, and HP, DS, JP thank National Institute Mathematical Sciences for warm hospitality when they visited NIMS as Research in CAMP Program. GU is grateful to KIAS, particularly to JongHae Keum, for the invitation to visit during January-February 2015, and to DS for inviting him to lecture in the Winter school on Algebraic Surfaces. The involvement of GU in this project happened during that visit.

HP was supported by Samsung Science and Technology Foundation under Project Number SSTF-BA1402-03.
JP was supported by Leaders Research Grant funded by Seoul National University and by the National Research Foundation of Korea Grant (2010-0019516). He also holds a joint appointment at KIAS and in the Research Institute of Mathematics, SNU.
DS was supported by Basic Science Research Program through the National Research Foundation of Korea funded by the Ministry of Education (NRF-2015R1D1A1A01060476).
GU was supported by the FONDECYT regular grant 1150068 funded by the Chilean government.

\section{Generalities on quotient surface singularities}
\label{section:generalities}

To fix notions and notations we recall some basics on quotient surface singularities. Let $(X,0)=(\mathbb{C}^2/G, 0)$ be a germ of quotient surface singularity, where $G$ is a finite subgroup of $GL(2, \mathbb{C})$. One may assume that $G$ does not contain any reflections. It is known that, for two finite subgroups $G_1$ and $G_2$ of $GL(2, \mathbb{C})$ without reflections, $(\mathbb{C}^2/G_1, 0)$ is analytically isomorphic to $(\mathbb{C}^2/G_2, 0)$ if and only if $G_1$ is conjugate to $G_2$. Therefore it is enough to classify finite subgroups of $GL(2,\mathbb{C})$ without reflections up to conjugation for classifying quotient surface singularities $(\mathbb{C}^2/G,0)$. We may assume that $G \subset U(2)$ because $G$ is finite. Then the action of $G$ on $\mathbb{C}^2$ lifts to an action on the blowing-up of $\mathbb{C}^2$ at the origin. So $G$ acts on the exceptional divisor $E \cong \mathbb{CP}^1$, where the action is induced by the double covering $G \subset U(2) \to PU(2) \cong SO(3)$. The image of $G$ in $SO(3)$ is either a (finite) cyclic subgroup, a dihedral group, the tetrahedral group, the octahedral group, or the icosahedral group. Therefore quotient surface singularities are divided into five classes: cyclic quotient surface singularities, dihedral singularities, tetrahedral singularities, octahedral singularities, and icosahedral singularities.

\subsection{Hirzebruch--Jung continued fractions}
\label{subsection:HJ-continued-fraction}

We denote by $[c_1, \dotsc, c_t]$ ($c_i \ge 1$) the \emph{Hirzebruch--Jung continued fraction} defined recursively by $[c_t]=c_t$, and
\begin{equation*}
[c_i, c_{i+1}, \dotsc, c_t] = c_i-\dfrac{1}{[c_{i+1}, \dotsc, c_t]}.
\end{equation*}
A continued fraction $[c_1, \dotsc, c_t]$ of integers often represents a chain of smooth rational curves on a complex surface whose dual graph is given by
\begin{equation*}
\begin{tikzpicture}
\node[bullet] (10) at (1,0) [label=above:{$-c_1$}] {};
\node[bullet] (20) at (2,0) [label=above:{$-c_2$}] {};

\node[empty] (250) at (2.5,0) [] {};
\node[empty] (30) at (3,0) [] {};

\node[bullet] (350) at (3.5,0) [label=above:{$-c_{t-1}$}] {};
\node[bullet] (450) at (4.5,0) [label=above:{$-c_t$}] {};

\draw [-] (10)--(20);
\draw [-] (20)--(250);
\draw [dotted] (20)--(350);
\draw [-] (30)--(350);
\draw [-] (350)--(450);
\end{tikzpicture}
\end{equation*}
So we use the term \emph{blowing up} for the following operations:
\begin{align*}
[c_1, \dotsc, c_{i-1}, c_{i+1}, \dotsc, c_t]
&\mapsto [c_1, \dotsc, c_{i-1}+1, 1, c_{i+1}+1, \dotsc, c_t]\\
[c_2, \dotsc, c_t]
&\mapsto [1, c_2+1, \dotsc, c_t]\\
[c_1, \dotsc, c_t]
&\mapsto [c_1, \dotsc, c_t+1, 1]\\
\end{align*}

Let $n$ and $a$ be positive numbers with $n > a$. Then the continued fraction expansions of $n/a$ and $n/(n-a)$ are \emph{dual}. That is, one finds the one form from the other with Riemenschneider's point diagram \cite{Riemenschneider-1974}: Let
\begin{equation*}
\frac{n}{a}=[b_1, \dotsc, b_r], \quad \frac{n}{n-a} = [a_1, a_2, \dotsc, a_e].
\end{equation*}
Place in the $i$-th row $a_i-1$ dots, the first one under the last one of the $(i-1)$-st row; then, the column $j$ contains $b_j-1$ dots, and vice versa.

\begin{example}\label{example:cyclic-19/7}
For $n=19$, $a=7$, we have $\frac{n}{a}=[3,4,2]$ and $\frac{n}{n-a}=[2,3,2,3]$. Its Riemenschneider's point diagram is
\begin{equation*}
\begin{tikzpicture}
\node[bullet] at (0,4) [] {};
\node[bullet] at (0,3.5) [] {};
\node[bullet] at (0.5,3.5) [] {};
\node[bullet] at (0.5,3) [] {};
\node[bullet] at (0.5,2.5) [] {};
\node[bullet] at (1,2.5) [] {};
\end{tikzpicture}
\end{equation*}
\end{example}

Furthermore we remark that
\begin{equation*}
[b_1, \dotsc, b_r, 1, a_e, \dotsc, a_1]=0,
\end{equation*}
which implies that the linear chain of $\mathbb{CP}^1$'s
\begin{equation*}
\begin{tikzpicture}
\node[bullet] (10) at (1,0) [label=above:{$-b_1$}] {};
\node[bullet] (20) at (2,0) [label=above:{$-b_2$}] {};

\node[empty] (250) at (2.5,0) [] {};
\node[empty] (30) at (3,0) [] {};

\node[bullet] (350) at (3.5,0) [label=above:{$-b_{r-1}$}] {};
\node[bullet] (450) at (4.5,0) [label=above:{$-b_r$}] {};

\node[bullet] (550) at (5.5,0) [label=above:{$-1$}] {};

\node[bullet] (650) at (6.5,0) [label=above:{$-a_e$}] {};
\node[empty] (70) at (7,0) [] {};
\node[empty] (750) at (7.5,0) [] {};
\node[bullet] (80) at (8,0) [label=above:{$-a_2$}] {};
\node[bullet] (90) at (9,0) [label=above:{$-a_1$}] {};

\draw [-] (10)--(20);
\draw [-] (20)--(250);
\draw [dotted] (20)--(350);
\draw [-] (30)--(350);
\draw [-] (350)--(450);
\draw [-] (450)--(550);

\draw [-] (550)--(650);
\draw [-] (650)--(70);
\draw [dotted] (650)--(80);
\draw [-] (750)--(80);
\draw [-] (80)--(90);
\end{tikzpicture}
\end{equation*}
is blown down to $\begin{tikzpicture}
\node[bullet] (00) at (0,0) [label=above:{$0$}] {};
\end{tikzpicture}$.

\subsection{Minimal resolutions}

We recall the dual graph of the minimal resolution of quotient surface singularities.

\subsubsection{Cyclic quotient surface singularities}

Let $G$ be a (multiplicative) cyclic group of order $n$ generated by a $n$-th root of unity $\zeta$. A \emph{cyclic quotient surface singularity $(X,0)$ of type $\frac{1}{n}(1,a)$} with $1 \le a < n$ and $(n,a)=1$ is a quotient surface singularity where $G$ acts by
\begin{equation*}
\zeta \cdot (x,y) = (\zeta x, \zeta^a y).
\end{equation*}
The dual graph of the minimal resolution of $(X,0)$ is given by
\begin{equation*}
\begin{tikzpicture}
\node[bullet] (10) at (1,0) [label=above:{$-b_1$}] {};
\node[bullet] (20) at (2,0) [label=above:{$-b_2$}] {};

\node[empty] (250) at (2.5,0) [] {};
\node[empty] (30) at (3,0) [] {};

\node[bullet] (350) at (3.5,0) [label=above:{$-b_{r-1}$}] {};
\node[bullet] (450) at (4.5,0) [label=above:{$-b_r$}] {};

\draw [-] (10)--(20);
\draw [-] (20)--(250);
\draw [dotted] (20)--(350);
\draw [-] (30)--(350);
\draw [-] (350)--(450);
\end{tikzpicture}
\end{equation*}
where
\begin{equation*}
\frac{n}{a} = [b_1, b_2, \dotsc, b_r]
\end{equation*}
with $b_i \ge 2$ for all $i$.

\subsubsection{Dihedral singularities}

Let $(X,0)$ be a dihedral singularity of type $D_{n,a}$, where $1 < a < n$ and $(n,a)=1$. The dual graph of the minimal resolution of $(X,0)$ is given by Figure~\ref{figure:non-cyclic-minimal-resolution}(A), where
\begin{equation*}
\frac{n}{a} = [b, b_1, b_2, \dotsc, b_r]
\end{equation*}
with $b \ge 2$ and $b_i \ge 2$ for all $i$.

\subsubsection{Tetrahedral, octahedral, icosahedral singularities}

Let $(X,0)$ be a tetrahedral, octahedral, or icosahedral singularity. The minimal resolution has a central curve $C_0$ with $C_0 \cdot C_0 = -b$ ($b \ge 2$) and three arms. See Figures~\ref{figure:non-cyclic-minimal-resolution}(B) and 2(C). One of the arms is always
\begin{tikzpicture}[scale=0.5]
\node[bullet] (00) at (0,0) [label=above:{$-2$}] {};
\end{tikzpicture}.
Another arm is
\begin{tikzpicture}[scale=0.5]
\node[bullet] (00) at (0,0) [label=above:{$-2$}] {};
\node[bullet] (10) at (1,0) [label=above:{$-2$}] {};
\draw [-] (00)--(10);
\end{tikzpicture}
or
\begin{tikzpicture}[scale=0.5]
\node[bullet] (00) at (0,0) [label=above:{$-3$}] {};
\end{tikzpicture}.
Bhupal--Ono~\cite{Bhupal-Ono-2012} divides tetrahedral, octahedral, icosahedral singularities into two types: \emph{type $(3,2)$} if another arm is
\begin{tikzpicture}[scale=0.5]
\node[bullet] (00) at (0,0) [label=above:{$-2$}] {};
\node[bullet] (10) at (1,0) [label=above:{$-2$}] {};
\draw [-] (00)--(10);
\end{tikzpicture}
and \emph{type $(3,1)$} if another arm is
\begin{tikzpicture}[scale=0.5]
\node[bullet] (00) at (0,0) [label=above:{$-3$}] {};
\end{tikzpicture}.
For a complete list of dual graphs, refer Riemenschneider~\cite{Riemenschneider-1981} or Bhupal--Ono~\cite{Bhupal-Ono-2012}.
\begin{figure}[t]
\centering
\subfloat [dihedral] {%
\begin{tikzpicture}
\node[bullet] (00) at (0,0) [label=below:{$-2$}] {};
\node[bullet] (10) at (1,0) [label=below:{$-b$}] {};
\node[bullet] (20) at (2,0) [label=below:{$-b_1$}] {};

\node[empty] (250) at (2.5,0) [] {};
\node[empty] (30) at (3,0) [] {};

\node[bullet] (350) at (3.5,0) [label=below:{$-b_{r}$}] {};

\node[bullet] (11) at (1,1) [label=left:{$-2$}] {};

\draw [-] (00)--(10);
\draw [-] (10)--(20);
\draw [-] (20)--(250);
\draw [dotted] (20)--(350);
\draw [-] (30)--(350);

\draw [-] (10)--(11);
\end{tikzpicture}}
\\
\subfloat [type $(3,2)$] {%
\begin{tikzpicture}
\node[bullet] (-10) at (-1,0) [label=below:{$-2$}] {};
\node[bullet] (00) at (0,0) [label=below:{$-2$}] {};
\node[bullet] (10) at (1,0) [label=below:{$-b$}] {};
\node[bullet] (20) at (2,0) [label=below:{$-b_1$}] {};

\node[empty] (250) at (2.5,0) [] {};
\node[empty] (30) at (3,0) [] {};

\node[bullet] (350) at (3.5,0) [label=below:{$-b_r$}] {};

\node[bullet] (11) at (1,1) [label=left:{$-2$}] {};

\draw [-] (-10)--(00);
\draw [-] (00)--(10);
\draw [-] (10)--(20);
\draw [-] (20)--(250);
\draw [dotted] (20)--(350);
\draw [-] (30)--(350);

\draw [-] (10)--(11);
\end{tikzpicture}}
\qquad \qquad
\subfloat [type $(3,1)$] {%
\begin{tikzpicture}
\node[bullet] (00) at (0,0) [label=below:{$-3$}] {};
\node[bullet] (10) at (1,0) [label=below:{$-b$}] {};
\node[bullet] (20) at (2,0) [label=below:{$-b_1$}] {};

\node[empty] (250) at (2.5,0) [] {};
\node[empty] (30) at (3,0) [] {};

\node[bullet] (350) at (3.5,0) [label=below:{$-b_r$}] {};

\node[bullet] (11) at (1,1) [label=left:{$-2$}] {};

\draw [-] (00)--(10);
\draw [-] (10)--(20);
\draw [-] (20)--(250);
\draw [dotted] (20)--(350);
\draw [-] (30)--(350);

\draw [-] (10)--(11);
\end{tikzpicture}}

\caption{The dual graphs of the minimal resolutions of non-cyclic quotient  singularities}
\label{figure:non-cyclic-minimal-resolution}
\end{figure}
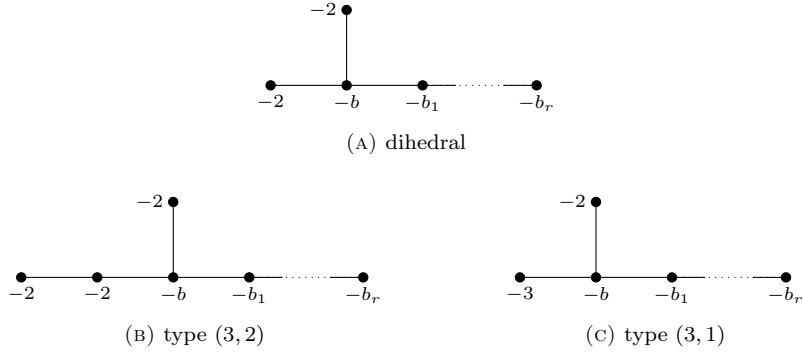

\begin{example}
\label{example:I_30(b-2)+23}
The icosahedral singularity $I_{30(b-2)+23}$ is of type $(3,1)$ and the dual graph of its minimal resolution is given by
\begin{equation*}
\begin{tikzpicture}
\node[bullet] (00) at (0,0) [label=below:{$-3$}] {};
\node[bullet] (10) at (1,0) [label=below:{$-b$}] {};
\node[bullet] (20) at (2,0) [label=below:{$-3$}] {};
\node[bullet] (30) at (3,0) [label=below:{$-2$}] {};
\node[bullet] (11) at (1,1) [label=left:{$-2$}] {};

\draw [-] (00)--(10);
\draw [-] (10)--(20);
\draw [-] (20)--(30);
\draw [-] (10)--(11);
\end{tikzpicture}
\end{equation*}
\end{example}

\section{Natural compactifications and compactifying divisors}
\label{section:Compactifying-divisor}

Let $(X,0)$ be a quotient surface singularity. We introduce a compactification of $X$ ``compatible'' with its smoothings. We refer Lisca~\cite{Lisca-2008} for cyclic quotient surface singularities, and Pinkham~\cite{Pinkham-1977, Pinkham-1978} for non-cyclic quotient surface singularities.

\subsection{Cyclic quotient surface singularities}
\label{subsection:compactifying-divisor-cyclic}

Let $(X,0)$ be a cyclic quotient surface singularity of type $\frac{1}{n}(1,a)$. Let $\mathbb{F}_1$ be a ruled surface over $\mathbb{CP}^1$. Let $S_0$ and $S_{\infty}$ be two sections with $S_0 \cdot S_0=-1$ and $S_{\infty} \cdot S_{\infty}=1$, and let $F$ be a fiber. One can blow up $\mathbb{F}_1$ appropriately at $p \in S_0 \cap F$, and at the infinitely near points over $p$, so that one obtain the following linear chain of $\mathbb{CP}^1$'s starting from the proper transform $\overline{S}_0$ of $S_0$ and ending with the proper transform $\overline{S}_{\infty}$ of $S_{\infty}$:
\begin{equation}\label{equation:blowing-up-F_1}
\begin{tikzpicture}
\node[bullet] (10) at (1,0) [label=above:{$-b_1$},label=below:{$\overline{S}_0$}] {};
\node[bullet] (20) at (2,0) [label=above:{$-b_2$}] {};

\node[empty] (250) at (2.5,0) [] {};
\node[empty] (30) at (3,0) [] {};

\node[bullet] (350) at (3.5,0) [label=above:{$-b_{r-1}$}] {};
\node[bullet] (450) at (4.5,0) [label=above:{$-b_r$}] {};

\node[bullet] (550) at (5.5,0) [label=above:{$-1$}] {};

\node[bullet] (650) at (6.5,0) [label=above:{$-a_e$},label=below:{$D_e$}] {};
\node[empty] (70) at (7,0) [] {};
\node[empty] (750) at (7.5,0) [] {};
\node[bullet] (80) at (8,0) [label=above:{$-a_2$},label=below:{$D_2$}] {};
\node[bullet] (90) at (9,0) [label=above:{$1-a_1$},label=below:{$D_1$}] {};
\node[bullet] (100) at (10,0) [label=above:{$+1$},label=below:{$\overline{S}_{\infty}$}] {};

\draw [-] (10)--(20);
\draw [-] (20)--(250);
\draw [dotted] (20)--(350);
\draw [-] (30)--(350);
\draw [-] (350)--(450);
\draw [-] (450)--(550);

\draw [-] (550)--(650);
\draw [-] (650)--(70);
\draw [dotted] (650)--(80);
\draw [-] (750)--(80);
\draw [-] (80)--(90);
\draw [-] (90)--(100);
\end{tikzpicture}
\end{equation}
where
\begin{equation*}
\frac{n}{a} = [b_1, \dotsc, b_r], \quad \frac{n}{n-a}=[a_1, \dotsc, a_e].
\end{equation*}

\begin{definition}
\label{definition:natural-compactification-cyclic}
Let $\widetilde{X}$ be the blown-up of the Hirzebruch surface $\mathbb{F}_1$ that contains the configuration in Equation~\eqref{equation:blowing-up-F_1} above. Contracting the configuration
\begin{equation}\label{equation:[b]}
\begin{tikzpicture}
\node[bullet] (10) at (1,0) [label=above:{$-b_1$}] {};
\node[bullet] (20) at (2,0) [label=above:{$-b_2$}] {};

\node[empty] (250) at (2.5,0) [] {};
\node[empty] (30) at (3,0) [] {};

\node[bullet] (350) at (3.5,0) [label=above:{$-b_{r-1}$}] {};
\node[bullet] (450) at (4.5,0) [label=above:{$-b_r$}] {};

\draw [-] (10)--(20);
\draw [-] (20)--(250);
\draw [dotted] (20)--(350);
\draw [-] (30)--(350);
\draw [-] (350)--(450);
\end{tikzpicture}
\end{equation}
in $\widetilde{X}$, we get a singular surface $\widehat{X}$ with the cyclic quotient surface singularity $(X,0)$, which is called the \emph{natural compactification} of $X$. We call the effective divisor
\begin{equation*}
E_{\infty} = \overline{S}_{\infty}+\sum_{i=1}^{e} D_i \subset \widehat{X}
\end{equation*}
the \emph{compactifying divisor} of $X$. The \emph{$\mathbb{C}^{\ast}$-compactification} of $X$, which is denoted by $\overline{X}$, is the singular surface obtained by contracting the linear chains
\begin{equation}\label{equation:[a]}
\begin{tikzpicture}
\node[bullet] (550) at (5.5,0) [label=above:{$-a_e$}] {};
\node[bullet] (650) at (6.5,0) [label=above:{$-a_{e-1}$}] {};
\node[empty] (70) at (7,0) [] {};
\node[empty] (750) at (7.5,0) [] {};
\node[bullet] (80) at (8,0) [label=above:{$-a_2$}] {};
\node[bullet] (90) at (9,0) [label=above:{$1-a_1$}] {};

\draw [-] (550)--(650);
\draw [-] (650)--(70);
\draw [dotted] (650)--(80);
\draw [-] (750)--(80);
\draw [-] (80)--(90);
\end{tikzpicture}
\end{equation}
in $\widehat{X}$. The image $\overline{E}_{\infty}$ in $\overline{X}$ of the $(+1)$-curve $\overline{S}_{\infty}$ in $\widehat{X}$ is called the \emph{singular compactifying divisor} of $X$. In summary, we have
\begin{equation*}
X \hookrightarrow (\overline{X}, \overline{E}_{\infty}) \leftarrow (\widehat{X}, E_{\infty}) \leftarrow (\widetilde{X}, E_{\infty}).
\end{equation*}
\end{definition}

\begin{remark}
If $a_i=2$ for all $i$, then there is no singularity on $\overline{E}_{\infty}$. Otherwise there is one cyclic quotient surface singularity on $\overline{E}_{\infty}$.
\end{remark}

\begin{lemma}\label{lemma:H2(-log(D))=0}
Let $\widetilde{X}$ be the resulting blowing up surface of $\mathbb{F}_1$ and let $D$ be the simple normal crossing divisor in $\widetilde{X}$ corresponding to the configuration in Equation~\eqref{equation:blowing-up-F_1}. Then
\begin{equation*}
H^2(\widetilde{X}, T_{\widetilde{X}}(-\log{D}))=0.
\end{equation*}
\end{lemma}

\begin{proof}
Since $\widetilde{X}$ is obtained from $\mathbb{F}_1$ by blowing up at the proper transforms of $S_0+S_{\infty}+F$, it follows by Flenner--Zaidenberg~\cite[Lemma~1.5]{Flenner-Zaidenberg-1994} that
\begin{equation*}
H^2(\widetilde{X}, T_{\widetilde{X}}(-\log{D})) = H^2(\mathbb{F}_1, T_{\mathbb{F}_1}(-\log(S_0 + S_\infty + F))).
\end{equation*}
From the standard exact sequence
\begin{equation*}
0 \to T_S(-\log{E}) \to T_S \to \oplus_i \sheaf{N_{E_i, S}} \to 0,
\end{equation*}
for a smooth surface $S$ and its simple normal crossing divisor $E=\sum_i E_i$, we have
\begin{equation*}
H^2(\mathbb{F}_1, T_{\mathbb{F}_1}(-\log(S_0 + S_\infty + F))) = H^2(\mathbb{F}_1, T_{\mathbb{F}_1}(-\log(S_0 + S_\infty))).
\end{equation*}
because $F \cdot F=0$. Therefore we need to prove that
\begin{equation*}
H^2(\mathbb{F}_1, T_{\mathbb{F}_1}(-\log(S_0 + S_\infty)))=0.
\end{equation*}
On the other hand, by Serre duality,
\begin{equation*}
H^2(\mathbb{F}_1, T_{\mathbb{F}_1}(-\log(S_0 + S_\infty)))=H^0(\mathbb{F}_1, \Omega_{\mathbb{F}_1}(\log(S_0+S_{\infty})) \otimes \Omega_{\mathbb{F}_1}^2)
\end{equation*}
where $\Omega_{\mathbb{F}_1}^2 = \sheaf{O_{\mathbb{F}_1}}(-2S_0 - 3F)$. Using the residue sequence
\begin{equation*}
0 \to \Omega_S \to \Omega_S(\log{E}) \to \oplus_i \sheaf{O_{E_i}} \to 0
\end{equation*}
and the fact that $S_0$ and $S_{\infty}$ are numerically independent, we have that
\begin{equation*}
H^0(\mathbb{F}_1, \Omega_{\mathbb{F}_1}(\log(S_0 + S_{\infty})))=0.
\end{equation*}
Therefore the assertion follows.
\end{proof}

\begin{proposition}\label{proposition:Extendable-cyclic}
Any smoothing of $X$ can be extended to a deformation of $\overline{X}$ which is locally trivial on $\overline{E}_{\infty}$, and hence, to a deformation of $\widehat{X}$ that preserves $E_{\infty}$.
\end{proposition}

\begin{proof}
Let $D_0$ be the exceptional divisor of the cyclic quotient surface singularity of type $\frac{1}{n}(1,a)$ given in Equation~\eqref{equation:[b]} and let $D_1$ be the divisor corresponding to the linear chain given in Equation~\eqref{equation:[a]}. In Lemma~\ref{lemma:H2(-log(D))=0}, one can delete the $(-1)$-curve in $D$ so that
\begin{equation*}
H^2(\widetilde{X}, T_{\widetilde{X}}(-\log(D_0+D_1)))=0.
\end{equation*}
Therefore, it follows from Y.~Lee--J.~Park~\cite[Theorem~2]{Lee-Park-K^2=2} that $H^2(\overline{X}, T_{\overline{X}})=0$, which implies that there is no local-to-global obstructions to deform $\overline{X}$. Therefore the assertion follows.
\end{proof}

\subsection{Non-cyclic quotient surface singularities}

Let $X=\Spec(A)$ be a germ of a non-cyclic quotient surface singularity. We refer to Pinkham~\cite{Pinkham-1977, Pinkham-1978} for the details in what follows: Since $X$ is a weighted homogeneous surface singularity, it admits a good $\mathbb{C}^\ast$-action. Hence $A$ is a graded ring.

\begin{definition}[cf.~Pinkham~\cite{Pinkham-1977}]
\label{definition:natural-compactification-non-cyclic}
The projective surface $\overline{X}=\Proj(A[t])$ is classically called the \emph{$\mathbb{C}^\ast$-compactification} of $X$, where the degree of $t$ is given by $1$ in the graded ring $A[t]$. The difference $\overline{E}_{\infty}=\overline{X}-X=\Proj(A)$ is called the \emph{singular compactifying divisor} of $X$. It is known that there are three cyclic quotient surface singularities on $\overline{E}_{\infty}$. Let $\widehat{X} \to \overline{X}$ be the resolution of singularities along $\overline{E}_{\infty}$, and let $E_{\infty}$ be the proper transform of $\overline{E}_{\infty}$. We call $\widehat{X}$ and $E_{\infty}$ the \emph{natural compactification} and  the \emph{compactifying divisor} of $X$, respectively. Finally, let $\widetilde{X} \to \widehat{X}$ be the minimal resolution of the quotient surface singularity $(X,0)$. In summary, we have
\begin{equation*}
X \hookrightarrow (\overline{X}, \overline{E}_{\infty}) \leftarrow (\widehat{X}, E_{\infty}) \leftarrow (\widetilde{X}, E_{\infty}).
\end{equation*}
\end{definition}

\begin{proposition}[{Pinkham~\cite[Theorem~2.9]{Pinkham-1978}}]
\label{proposition:Extendable-non-cyclic}
Every smoothing of $X$ extends to a deformation of $\overline{X}$ which is locally trivial near $\overline{E}_{\infty}$,  and hence, to a deformation of $\widetilde{X}$ preserving $E_{\infty}$.
\end{proposition}

The dual graph $\Gamma$ of $(X,0)$ is a star-shaped graph with the central curve $E_0$ with $E_0 \cdot E_0 = -b$ and three arms as in Figure~\ref{figure:non-cyclic-minimal-resolution}. It is known that the dual graph $\widehat{\Gamma}$ of the compactifying divisor $E_{\infty}$ is also a star-shaped graph with the central curve $E_{\infty}^0$ with $E_{\infty}^0 \cdot E_{\infty}^0 = b-3$, and three arms dual to the corresponding three arms of $\Gamma$. That is, the dual graph of $E_{\infty}$ is given as in Figure~\ref{figure:non-cyclic-compactifying-divisor}, where
\begin{equation*}
\frac{n}{a} = [b_1, \dotsc, b_r] \qquad \frac{n}{n-a} = [a_1, \dotsc, a_e].
\end{equation*}
The dual graph of $\widetilde{X}$ is given as in Figure~\ref{figure:non-cyclic-widetilde(X)}.

\begin{figure}[t]
\centering

\subfloat [dihedral] {%
\begin{tikzpicture}
\node[bullet] (00) at (0,0) [label=below:{$-2$}] {};
\node[bullet] (10) at (1,0) [label=below:{$b-3$}] {};
\node[bullet] (20) at (2,0) [label=below:{$-a_1$}] {};

\node[empty] (250) at (2.5,0) [] {};
\node[empty] (30) at (3,0) [] {};

\node[bullet] (350) at (3.5,0) [label=below:{$-a_e$}] {};

\node[bullet] (11) at (1,1) [label=left:{$-2$}] {};

\draw [-] (00)--(10);
\draw [-] (10)--(20);
\draw [-] (20)--(250);
\draw [dotted] (20)--(350);
\draw [-] (30)--(350);

\draw [-] (10)--(11);
\end{tikzpicture}}
\\
\subfloat [type $(3,2)$] {%
\begin{tikzpicture}
\node[bullet] (00) at (0,0) [label=below:{$-3$}] {};
\node[bullet] (10) at (1,0) [label=below:{$b-3$}] {};
\node[bullet] (20) at (2,0) [label=below:{$-a_1$}] {};

\node[empty] (250) at (2.5,0) [] {};
\node[empty] (30) at (3,0) [] {};

\node[bullet] (350) at (3.5,0) [label=below:{$-a_e$}] {};

\node[bullet] (11) at (1,1) [label=left:{$-2$}] {};

\draw [-] (00)--(10);
\draw [-] (10)--(20);
\draw [-] (20)--(250);
\draw [dotted] (20)--(350);
\draw [-] (30)--(350);

\draw [-] (10)--(11);
\end{tikzpicture}}
\qquad \qquad
\subfloat [type $(3,1)$] {%
\begin{tikzpicture}
\node[bullet] (-10) at (-1,0) [label=below:{$-2$}] {};
\node[bullet] (00) at (0,0) [label=below:{$-2$}] {};
\node[bullet] (10) at (1,0) [label=below:{$b-3$}] {};
\node[bullet] (20) at (2,0) [label=below:{$-a_1$}] {};

\node[empty] (250) at (2.5,0) [] {};
\node[empty] (30) at (3,0) [] {};

\node[bullet] (350) at (3.5,0) [label=below:{$-a_e$}] {};

\node[bullet] (11) at (1,1) [label=left:{$-2$}] {};

\draw [-] (-10)--(00);
\draw [-] (00)--(10);
\draw [-] (10)--(20);
\draw [-] (20)--(250);
\draw [dotted] (20)--(350);
\draw [-] (30)--(350);

\draw [-] (10)--(11);
\end{tikzpicture}}

\caption{The compactifying divisor $E_{\infty}$ for non-cyclic quotient surface singularities.}

\label{figure:non-cyclic-compactifying-divisor}
\end{figure}
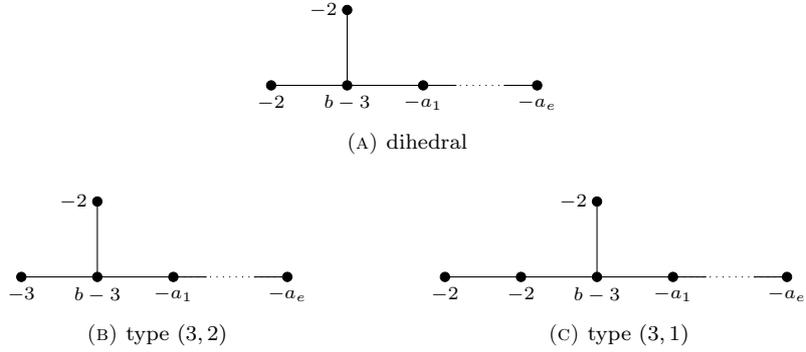

\begin{remark}\label{remark:model-for-widetilde(X)}
The natural compactification $\widehat{X}$ of $(X,0)$ is given by $\widehat{X}=\Proj(A[t])$, the projective cone over the singular compactifying divisor $\overline{E}_{\infty} = \Proj(A)$, if $X=\Spec(A)$. So the minimal resolution $\widetilde{X}$ has a $\mathbb{CP}^1$-fibration structure over the central curve $E_{\infty}^0 \cong \mathbb{CP}^1$ of $E_{\infty}$. In this way, the concrete model for $\widetilde{X}$ can be constructed from the Hirzebruch surface $\mathbb{F}_b$ by blowing ups as follows: Let $S_0$ and $S_{\infty}$ be two sections with $S_0 \cdot S_0 = -b$ and $S_{\infty} \cdot S_{\infty} = b$. Let $F_1$, $F_2$, $F_3$ be three distinct fibers. Then, by blowing up appropriately at $p_i \in F_i \cap S_{\infty}$ ($i=1,2,3$), including infinitely near points over $p_i$, one can construct the rational surface $\widetilde{X}$ with the configuration of smooth rational curves described in Figure~\ref{figure:non-cyclic-widetilde(X)}. It is not difficult to show that $H^2(\overline{X}, T_{\overline{X}})$ vanishes. Hence every smoothing of $X$ can be extended to $\overline{X}$ and $\widehat{X}$.
\end{remark}

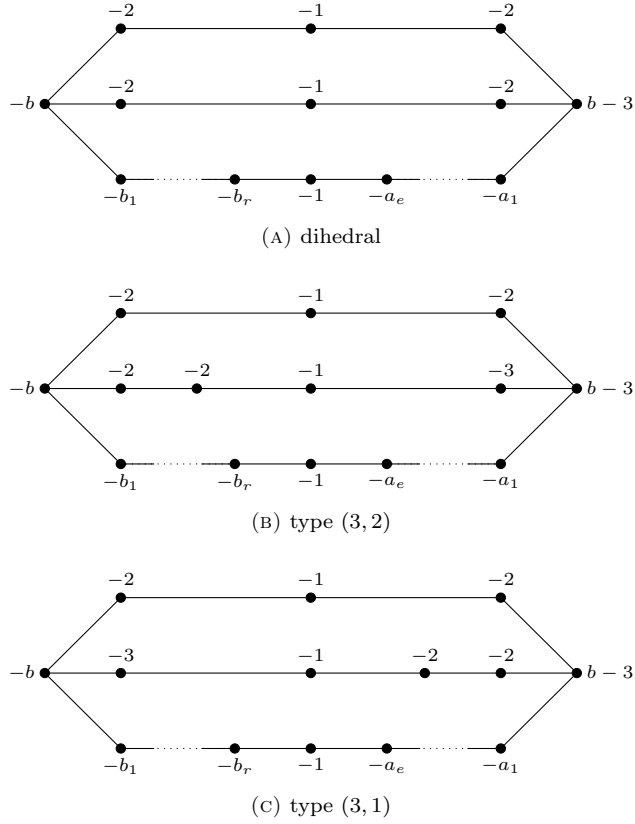
\begin{figure}[t]
\centering

\subfloat [dihedral] {%
\begin{tikzpicture}
\node[bullet] (00) at (0,0) [label=left:{$-b$}] {};
\node[bullet] (70) at (7,0) [label=right:{$b-3$}] {};

\node[bullet] (11) at (1,1) [label=above:{$-2$}] {};

\node[bullet] (351) at (3.5,1) [label=above:{$-1$}] {};

\node[bullet] (61) at (6,1) [label=above:{$-2$}] {};

\draw [-] (00)--(11)--(351)--(61)--(70);

\node[bullet] (10) at (1,0) [label=above:{$-2$}] {};

\node[bullet] (350) at (3.5,0) [label=above:{$-1$}] {};

\node[bullet] (60) at (6,0) [label=above:{$-2$}] {};

\draw [-] (00)--(10)--(350)--(60)--(70);

\node[bullet] (1-1) at (1,-1) [label=below:{$-b_1$}] {};
\node[empty] (15-1) at (1.5,-1) [] {};
\node[empty] (2-1) at (2,-1) [] {};
\node[bullet] (25-1) at (2.5,-1) [label=below:{$-b_r$}] {};

\draw [-] (00)--(1-1);
\draw [-] (1-1)--(15-1);
\draw [dotted] (1-1)--(25-1);
\draw [-] (2-1)--(25-1);

\node[bullet] (35-1) at (3.5,-1) [label=below:{$-1$}] {};
\node[bullet] (45-1) at (4.5,-1) [label=below:{$-a_e$}] {};
\node[empty] (5-1) at (5,-1) [] {};
\node[empty] (55-1) at (5.5,-1) [] {};
\node[bullet] (6-1) at (6,-1) [label=below:{$-a_1$}] {};

\draw [-] (25-1)--(35-1);
\draw [-] (35-1)--(45-1);
\draw [-] (45-1)--(5-1);
\draw [dotted] (45-1)--(6-1);
\draw [-] (55-1)--(6-1);
\draw [-] (6-1)--(70);
\end{tikzpicture}}
\\
\subfloat [type $(3,2)$] {%
\begin{tikzpicture}
\node[bullet] (00) at (0,0) [label=left:{$-b$}] {};
\node[bullet] (70) at (7,0) [label=right:{$b-3$}] {};

\node[bullet] (11) at (1,1) [label=above:{$-2$}] {};
\node[bullet] (351) at (3.5,1) [label=above:{$-1$}] {};
\node[bullet] (61) at (6,1) [label=above:{$-2$}] {};

\draw [-] (00)--(11)--(351)--(61)--(70);

\node[bullet] (10) at (1,0) [label=above:{$-2$}] {};
\node[bullet] (20) at (2,0) [label=above:{$-2$}] {};
\node[bullet] (350) at (3.5,0) [label=above:{$-1$}] {};
\node[bullet] (60) at (6,0) [label=above:{$-3$}] {};

\draw [-] (00)--(10)--(20)--(350)--(60)--(70);

\node[bullet] (1-1) at (1,-1) [label=below:{$-b_1$}] {};
\node[empty] (15-1) at (1.5,-1) [] {};
\node[empty] (2-1) at (2,-1) [] {};
\node[bullet] (25-1) at (2.5,-1) [label=below:{$-b_r$}] {};

\draw [-] (00)--(1-1);
\draw [-] (1-1)--(15-1);
\draw [dotted] (1-1)--(25-1);
\draw [-] (2-1)--(25-1);

\node[bullet] (35-1) at (3.5,-1) [label=below:{$-1$}] {};
\node[bullet] (45-1) at (4.5,-1) [label=below:{$-a_e$}] {};
\node[empty] (5-1) at (5,-1) [] {};
\node[empty] (55-1) at (5.5,-1) [] {};
\node[bullet] (6-1) at (6,-1) [label=below:{$-a_1$}] {};

\draw [-] (25-1)--(35-1);
\draw [-] (35-1)--(45-1);
\draw [-] (45-1)--(5-1);
\draw [dotted] (45-1)--(6-1);
\draw [-] (55-1)--(6-1);
\draw [-] (6-1)--(70);
\end{tikzpicture}}
\\
\subfloat [type $(3,1)$] {%
\begin{tikzpicture}
\node[bullet] (00) at (0,0) [label=left:{$-b$}] {};
\node[bullet] (70) at (7,0) [label=right:{$b-3$}] {};

\node[bullet] (11) at (1,1) [label=above:{$-2$}] {};
\node[bullet] (351) at (3.5,1) [label=above:{$-1$}] {};
\node[bullet] (61) at (6,1) [label=above:{$-2$}] {};

\draw [-] (00)--(11)--(351)--(61)--(70);

\node[bullet] (10) at (1,0) [label=above:{$-3$}] {};
\node[bullet] (350) at (3.5,0) [label=above:{$-1$}] {};
\node[bullet] (50) at (5,0) [label=above:{$-2$}] {};
\node[bullet] (60) at (6,0) [label=above:{$-2$}] {};

\draw [-] (00)--(10)--(350)--(50)--(60)--(70);

\node[bullet] (1-1) at (1,-1) [label=below:{$-b_1$}] {};
\node[empty] (15-1) at (1.5,-1) [] {};
\node[empty] (2-1) at (2,-1) [] {};
\node[bullet] (25-1) at (2.5,-1) [label=below:{$-b_r$}] {};

\draw [-] (00)--(1-1);
\draw [-] (1-1)--(15-1);
\draw [dotted] (1-1)--(25-1);
\draw [-] (2-1)--(25-1);

\node[bullet] (35-1) at (3.5,-1) [label=below:{$-1$}] {};
\node[bullet] (45-1) at (4.5,-1) [label=below:{$-a_e$}] {};
\node[empty] (5-1) at (5,-1) [] {};
\node[empty] (55-1) at (5.5,-1) [] {};
\node[bullet] (6-1) at (6,-1) [label=below:{$-a_1$}] {};

\draw [-] (25-1)--(35-1);
\draw [-] (35-1)--(45-1);
\draw [-] (45-1)--(5-1);
\draw [dotted] (45-1)--(6-1);
\draw [-] (55-1)--(6-1);
\draw [-] (6-1)--(70);
\end{tikzpicture}}

\caption{The dual graph of $\widetilde{X}$ for non-cyclic quotient surface singularities.}

\label{figure:non-cyclic-widetilde(X)}
\end{figure}

\section{Minimal symplectic fillings}
\label{section:minimal-symplectic-fillings}

As mentioned in Introduction above, Lisca~\cite{Lisca-2008} classifies minimal symplectic fillings of a cyclic quotient surface singularity up to diffeomorphism and Bhupal--Ono~\cite{Bhupal-Ono-2012} provides all minimal symplectic fillings of a non-cyclic quotient surface singularity. In this section we summarize Lisca's result~\cite{Lisca-2008} and Bhupal-Ono's result~\cite{Bhupal-Ono-2012}.

\subsection{Minimal symplectic fillings of a cyclic quotient surface singularity}

We first review Lisca's classification \cite{Lisca-2008}. Let $(X,0)$ be a cyclic quotient surface singularity of type $\frac{1}{n}(1,a)$, and let $L$ be its link. Let $W$ be a minimal symplectic filling of $(X,0)$. Then $W$ can be compactified as follows: By blowing-up $\mathbb{F}_1$ successively as in Section~\ref{subsection:compactifying-divisor-cyclic} above, one can obtain a rational complex surface $\widetilde{X}$ containing linear chains $E_0$ and $E_{\infty}$ of $\mathbb{CP}^1$'s whose dual graphs are
\begin{equation*}
\begin{tikzpicture}
\node (050) at (0.5,0) {$E_0 =$};
\node[bullet] (10) at (1,0) [label=above:{$-b_1$}] {};
\node[bullet] (20) at (2,0) [label=above:{$-b_2$}] {};

\node[empty] (250) at (2.5,0) [] {};
\node[empty] (30) at (3,0) [] {};

\node[bullet] (350) at (3.5,0) [label=above:{$-b_{r-1}$}] {};
\node[bullet] (450) at (4.5,0) [label=above:{$-b_r$}] {};

\draw [-] (10)--(20);
\draw [-] (20)--(250);
\draw [dotted] (20)--(350);
\draw [-] (30)--(350);
\draw [-] (350)--(450);
\end{tikzpicture}
\qquad
\begin{tikzpicture}
\node (50) at (5,0) {$E_{\infty}=$};

\node[bullet] (550) at (5.5,0) [label=above:{$-a_e$}] {};
\node[bullet] (650) at (6.5,0) [label=above:{$-a_{e-1}$}] {};
\node[empty] (70) at (7,0) [] {};
\node[empty] (750) at (7.5,0) [] {};
\node[bullet] (80) at (8,0) [label=above:{$-a_2$}] {};
\node[bullet] (90) at (9,0) [label=above:{$1-a_1$}] {};
\node[bullet] (100) at (10,0) [label=above:{$1$}] {};

\draw [-] (550)--(650);
\draw [-] (650)--(70);
\draw [dotted] (650)--(80);
\draw [-] (750)--(80);
\draw [-] (80)--(90);
\draw [-] (90)--(100);
\end{tikzpicture}
\end{equation*}
where
\begin{equation*}
\frac{n}{a} = [b_1, \dotsc, b_r], \qquad \frac{n}{n-a} = [a_1, \dotsc, a_e].
\end{equation*}
Let $\nu(E_0)$ be a regular neighborhood of $E_0$. Then $\partial(\nu(E_0))=L$.

\begin{proposition}[{Lisca~\cite[Theorem~3.2, Proposition~3.4]{Lisca-2008}}]
\label{proposition:symplectic-filling-is-rational-I}
There is a symplectic form on the smooth 4-manifold
\begin{equation*}
Z = W \cup_{L} (\widetilde{X} -\nu(E_0))
\end{equation*}
which is compatible with the symplectic structure on $W$. Furthermore $Z$ is a rational symplectic $4$-manifold.
\end{proposition}

\begin{definition}\label{definition:natural-compactification-of-filling-cyclic}
Let $W$ be a minimal symplectic filling of $X$. The above symplectic $4$-manifold $Z=W \cup_{L} (\widetilde{X} -\nu(E_0))$ is called the \emph{natural compactification} of $W$.
\end{definition}

Since $W \cong Z - \nu(E_{\infty})$ and $W$ is minimal, every $(-1)$-curve in $Z$ intersects $E_{\infty}$. Lisca~\cite{Lisca-2008} parametrizes minimal symplectic fillings of a cyclic quotient surface singularity using the data of intersections of $(-1)$-curves with $E_{\infty}$ as follows:

\begin{definition}[cf.~N\'emethi--Popescu-Pampu~\cite{Nemethi-PPampu-2010}]
\label{definition:K_e(n/n-a)}
\hfill \null

\begin{enumerate}
\item A sequence $(n_1, \dotsc, n_e) \in \mathbb{N}^e$ is \emph{admissible} if the matrix $M(n_1, \dotsc, n_e)$ is positive semi-definite of rank at least $e-1$, where $M(n_1, \dotsc, n_e) \in M_{e, e}(\mathbb{Z})$ is defined by $M_{i,i}=n_i$, $M_{i,j}=-1$ if $\abs{i-j}=1$, and $M_{i,j}=0$ otherwise. Denote by $\adm(\mathbb{N}^e)$ the set of admissible $e$-tuples.

\item For $e \ge 1$, we define
\begin{multline*}
K_e(n/n-a) := \\
\{\underline{n} = (n_1, \dotsc, n_e) \in \adm(\mathbb{N}^e) \mid \text{$[n_1, \dotsc, n_e]=0$ and $0<n_i \le a_i$, $\forall i$}\}.
\end{multline*}
\end{enumerate}
\end{definition}

\begin{remark}\label{remark:representing-0}
For every $(n_1, \dotsc, n_e) \in K_e(n/n-a)$, the Hirzebruch-Jung continued fraction $[n_1, \dotsc, n_e]$ can be obtained from $[0]$ by blowing-ups; see Subsection \ref{subsection:HJ-continued-fraction}.
\end{remark}

Let $\underline{n} = (n_1, \dotsc, n_e) \in K_e(n/n-a)$. We define a smooth 4-manifold $W_{n,a}(\underline{n})$ associated to $\underline{n}$ as follows: Given two lines $L_1$ and $L_2$ in $\mathbb{CP}^2$, we blow up successively a line $L_1$ at a smooth point $p \in L_1 - L_2$ including infinitely near points over $p$ to obtain a linear chain of $\mathbb{CP}^1$'s whose dual graph is
\begin{equation*}
\begin{tikzpicture}
\node[bullet] (650) at (6.5,0) [label=above:{$-n_e$},label=below:{$R_e'$}] {};
\node[empty] (70) at (7,0) [] {};
\node[empty] (750) at (7.5,0) [] {};
\node[bullet] (80) at (8,0) [label=above:{$-n_2$},label=below:{$R_2'$}] {};
\node[bullet] (90) at (9,0) [label=above:{$1-n_1$},label=below:{$R_1'$}] {};
\node[bullet] (100) at (10,0) [label=above:{$+1$},label=below:{$L_2$}] {};

\draw [-] (650)--(70);
\draw [dotted] (650)--(80);
\draw [-] (750)--(80);
\draw [-] (80)--(90);
\draw [-] (90)--(100);
\end{tikzpicture}
\end{equation*}
We then blow up the vertex $R_i'$ at $(a_i-n_i)$ distinct points for each $i=1, \dotsc, e$ so that one obtain a rational complex surface $Z$ which contains a linear chain $E_{\infty}$ of $\mathbb{CP}^1$'s whose dual graph is
\begin{equation*}
\begin{tikzpicture}
\node[bullet] (650) at (6.5,0) [label=above:{$-a_e$},label=below:{$R_e$}] {};
\node[empty] (70) at (7,0) [] {};
\node[empty] (750) at (7.5,0) [] {};
\node[bullet] (80) at (8,0) [label=above:{$-a_2$},label=below:{$R_2$}] {};
\node[bullet] (90) at (9,0) [label=above:{$1-a_1$},label=below:{$R_1$}] {};
\node[bullet] (100) at (10,0) [label=above:{$+1$},label=below:{$L_2$}] {};

\draw [-] (650)--(70);
\draw [dotted] (650)--(80);
\draw [-] (750)--(80);
\draw [-] (80)--(90);
\draw [-] (90)--(100);
\end{tikzpicture}i
\end{equation*}

\begin{definition}[{Lisca~\cite[p.766]{Lisca-2008}}]

For each $\underline{n} = (n_1, \dotsc, n_e) \in K_e(n/n-a)$ in the construction above, the smooth 4-manifold $W_{n,a}(\underline{n})$ is defined by
\begin{equation*}
W_{n,a}(\underline{n}) = Z - \nu(E_{\infty}).
\end{equation*}
\end{definition}

Some of the main results of Lisca~\cite{Lisca-2008} may be summarized as follows:

\begin{theorem}[Lisca~\cite{Lisca-2008}]
\label{theorem:Lisca}
Let $(X,0)$ be a cyclic quotient surface singularity of type $\frac{1}{n}(1,a)$. Then

\begin{enumerate}
\item Every $W_{n,a}(\underline{n})$ is a Stein filling of $(X,0)$.

\item Any minimal symplectic filling of $(X,0)$ is diffeomorphic to  $W_{n,a}(\underline{n})$ for some $\underline{n} \in K_e(n/n-a)$.

\item The manifold $W_{n,a}(\underline{n})$ is orientation-preserving diffeomorphic to a $4$-manifold $W_{n,a}(\underline{m})$ if and only if $\underline{n} = \underline{m}$ or, additionally, $\underbar{n} = \underbar{m}^{\circ}$ in case $a^2 \equiv 1 \pmod{n}$, where $\underbar{m}^{\circ}=(m_e, m_{e-1}, \dotsc, m_1)$.
\end{enumerate}
\end{theorem}

\begin{remark}
If $a^2 \equiv 1 \pmod{n}$, then the dual graph of the minimal resolution of $(X,0)$ is symmetric, or, more explicitly, \emph{palindromic}. That is, it is the same whether one reads it backwards or forwards. Therefore there are obvious additional diffeomorphic pairs in Theorem~\ref{theorem:Lisca}(3).
\end{remark}

In Section~\ref{section:identifying-Milnor-fibers}, we present another proof of the existence of a one-to-one correspondence between minimal symplectic fillings and $\underline{n} \in K_e(n/n-a)$ by using a technique from the minimal model program for 3-folds; Corollary~\ref{corollary:New-proof-of-Stevens-theorem}.

\subsection{Minimal symplectic fillings of non-cyclic quotient surface singularities}

Let $(X,0)$ be a non-cyclic quotient surface singularity and let $L$ be its link. We briefly summarize the main results of Bhupal--Ono~\cite{Bhupal-Ono-2012}.

Let $\widetilde{X}$ be the blown-up of the Hirzebruch surface $\mathbb{F}_b$ discussed in Remark~\ref{remark:model-for-widetilde(X)}. Let $E_0 \subset \widetilde{X}$ be the exceptional divisor of the minimal resolution of $(X,0)$ and let $E_{\infty} \subset \widetilde{X}$ be the compactifying divisor of $(X,0)$.

\begin{proposition}[{Bhupal--Ono~\cite{Bhupal-Ono-2012}}]
\label{proposition:symplectic-filling-is-rational-II}
Let $W$ be a minimal symplectic filling of $(X,0)$. Then there is a symplectic form on the smooth 4-manifold
\begin{equation*}
Z =  W \cup_L \nu(E_{\infty})
\end{equation*}
which is compatible with the symplectic structure on $W$. Furthermore $Z$ is a rational symplectic $4$-manifold.
\end{proposition}

\begin{definition}\label{definition:natural-compactification-of-filling-non-cyclic}
For a minimal symplectic filling $W$ of $X$, the symplectic $4$-manifold $Z=W \cup_{L} \nu(E_{\infty})$ is called the \emph{natural compactification} of $W$.
\end{definition}

\begin{remark}\label{remark:complex-model}
The symplectic deformation type of the pair $(Z, E_{\infty})$ associated to a given minimal symplectic filling $W$ is unique (Bhupal--Ono~\cite[p.35]{Bhupal-Ono-2012}) and one can choose $Z$ as a rational complex surface and $E_{\infty}$ as a union of rational complex curves.
\end{remark}

Bhupal--Ono~\cite{Bhupal-Ono-2012} proves that $Z$ is rational by showing that there is a sequence of blow-downs and blow-ups transforming the compactifying divisor $E_{\infty}$ in $Z$ into a configuration containing a cuspidal curve with positive self-intersection number; then $Z$ is rational by Ohta--Ono~\cite{Ohta-Ono-2005-II}.

We present these blow-downs and blow-ups in Figure~\ref{figure:sequence-dihedral} for dihedral singularities and in Figures~\ref{figure:sequence-TOI-32} and \ref{figure:sequence-TOI-31} for tetrahedral, octahedral, and icosahedral singularities. We first blow-up successively (if necessary) the intersection point of the central curve of $E_{\infty}$ and the third branch until the self-intersection number of the central curve has dropped to $-1$. Then we blow down or blow up as described in Figures~\ref{figure:sequence-dihedral}, \ref{figure:sequence-TOI-32} and \ref{figure:sequence-TOI-31} so that we get a rational 4-manifold $Z_2$ with a cuspidal curve $C$ with $C \cdot C > 0$ and a linear chain of 2-spheres $C_1, \dotsc, C_k$ (plus some extra 2-spheres intersecting $C$ at the cusp). Let $\widetilde{E}_{\infty} \subset Z_2$ be the proper transform of $E_{\infty} \subset Z$. Since the blow-ups and blow-downs occur only on $E_{\infty}$ and its proper transforms, we have
\begin{equation}\label{equation:W=Z-E_infty=Z_2-E_infty}
W \cong Z - \nu(E_{\infty}) \cong Z_2 - \nu(\widetilde{E}_{\infty}).
\end{equation}
Since $W$ is minimal, every $(-1)$-curve in $Z_2$ should intersect $\widetilde{E}_{\infty}$. Let $Z_1$ be the rational 4-manifold obtained by contracting all $(-1)$-curves in $Z_2 - C$, that is, the $(-1)$-curves not intersecting $C$. Then, according to Bhupal--Ono~\cite{Bhupal-Ono-2012}, $Z_1$ can be obtained by a sequence of blowing-ups at $p$ (including infinitely near points over $p$) in a cuspidal curve $C$ in $\mathbb{CP}^2$ or $\mathbb{CP}^1 \times \mathbb{CP}^1$ and (if necessary) some points on $C$ as shown in Figures~\ref{figure:sequence-dihedral}, \ref{figure:sequence-TOI-32} and \ref{figure:sequence-TOI-31}. For details, refer Bhupal--Ono~\cite{Bhupal-Ono-2012}.

\begin{figure}[tp]
\centering
\includegraphics{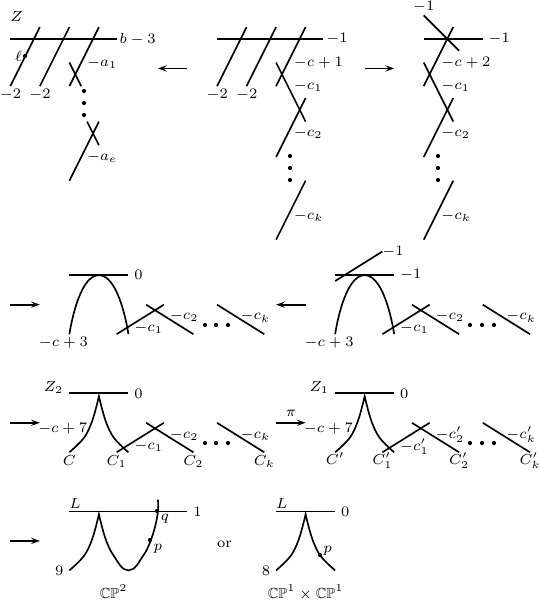}
\caption{For dihedral singularities; Bhupal--Ono~\cite[Figure~1]{Bhupal-Ono-2012}}
\label{figure:sequence-dihedral}
\end{figure}

\begin{figure}[tp]
\centering
\includegraphics{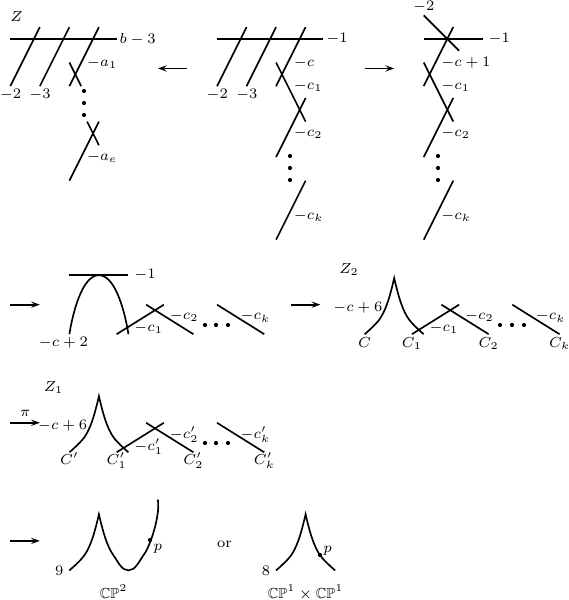}
\caption{For tetrahedral, octahedral, or icosahedral singularities of type $(3,2)$; Bhupal--Ono~\cite[Figure~3]{Bhupal-Ono-2012}}
\label{figure:sequence-TOI-32}
\end{figure}

\begin{figure}[tp]
\centering
\includegraphics{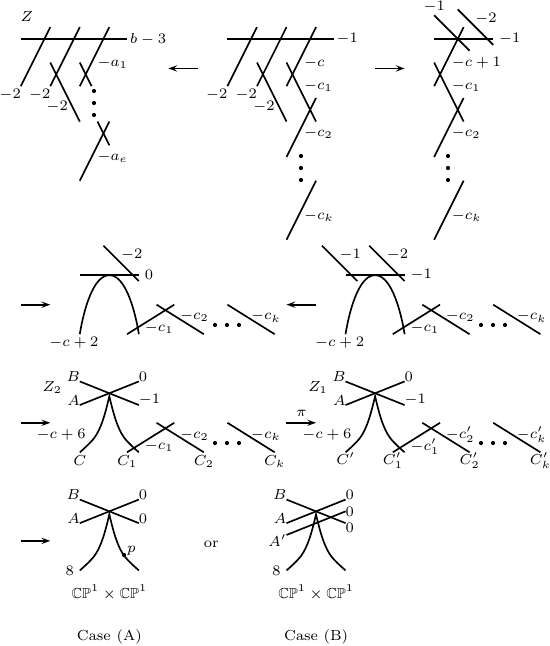}
\caption{For tetrahedral, octahedral, or icosahedral singularities of type $(3,1)$; Bhupal--Ono~\cite[Figure~5, 6, 10]{Bhupal-Ono-2012}}
\label{figure:sequence-TOI-31}
\end{figure}

Conversely, we explain briefly how to blow up $\mathbb{CP}^2$ or $\mathbb{CP}^1 \times \mathbb{CP}^1$ in order to get such a rational 4-manifold $Z_1$. Once we get $Z_1$, we blow up points on $C_1', \dotsc, C_k'$ to get $C_1, \dotsc, C_k$ so that we have $Z_2$. \emph{Note that the sequence of blow-ups $\pi \colon Z_2 \to Z_1$ occur only on $C_i'$.}

In case of dihedral singularities (cf.~Figure~\ref{figure:sequence-dihedral}), if we start with $\mathbb{CP}^2$, then we need a cuspidal cubic (a plane curve of degree $3$ with a cusp), and a line $L$ passing through the cusp. We first blow up at the intersection point $q$ with the line and the cuspidal curve, and then we blow up successively at a smooth point $p$ of the cuspidal curve and at some infinitely near points over $p$ to make the linear chain $C_1', \dotsc, C_k'$. We may need to blow up at some smooth points of the cuspidal curve (if necessary). Then we get the rational surface $Z_1$. If we start with $\mathbb{CP}^1 \times \mathbb{CP}^1$, then we need a cuspidal curve (a curve of type $(2,2)$ with a cusp), and a curve $L$ of type $(1,0)$ or $(0,1)$. As before, we blow up successively at a smooth point $p$ of the cuspidal curve and at some infinitely near points over $p$, to make the linear chain $C_1', \dotsc, C_k'$. If necessary, we blow up at some smooth points of the cuspidal curve. Then we get the rational surface $Z_1$.

In case of tetrahedral, octahedral, icosahedral singularities of type $(3,2)$ (cf.~Figure~\ref{figure:sequence-TOI-32}), we need only a cuspidal curve in $\mathbb{CP}^2$ or in $\mathbb{CP}^1 \times \mathbb{CP}^1$. We then blow up at $p$ including infinitely near points over $p$, and some smooth points of the cuspidal curve (if necessary) to get the rational surface $Z_1$.

Finally, for tetrahedral, octahedral, icosahedral singularities of type $(3,1)$, the rational surface $Z_1$ is obtained only from $\mathbb{CP}^1 \times \mathbb{CP}^1$. But there are two ways of blowing up and blowing down; we refer to Figure~\ref{figure:sequence-TOI-31}. In Case~(A), we start with a cuspidal curve and two curves $A$ and $B$ of type $(1,0)$ and type $(0,1)$, respectively, passing through the cusp singularity. We blow up at $p$ including infinitely near points over $p$ and some smooth points of the cuspidal curve (if necessary) to get the rational surface $Z_1$. In Case~(B), we need a cuspidal curve and two curves $A$, $A'$ of the same type $(1,0)$ and a curve $B$ of type $(0,1)$ passing through the cusp. We first blow up at some point in $B - (A \cup A')$, and then we blow up successively at $p$ including infinitely near points over $p$ to make the linear chain $C_1', \dotsc, C_k'$. Here the proper transform of $A'$ becomes a curve $C_i'$ for some $i$. Also, as before, one may need to blow up at some smooth points of the cuspidal curve. Then we get the rational surface $Z_1$.

\begin{remark}\label{remark:no-extra-blowup}
In any case of Figures~\ref{figure:sequence-dihedral}, \ref{figure:sequence-TOI-32} and \ref{figure:sequence-TOI-31}, if we blow up the cuspidal curve $C$ in $\mathbb{CP}^1 \times \mathbb{CP}^1$ at the other point $q \neq p$ in order to obtain $Z_1$, then we may obtain the same model $(Z, E)$ from $\mathbb{CP}^2$ because $\mathbb{CP}^2 \sharp 2 \overline{\mathbb{CP}}^2 \cong (\mathbb{CP}^1 \times \mathbb{CP}^1) \sharp \overline{\mathbb{CP}}^2$. So, we may assume that there are no extra blow-up points other than $p \in C$ in case we obtain $Z_1$ from $\mathbb{CP}^1 \times \mathbb{CP}^1$.
\end{remark}

\begin{example}[Continued from Example~\ref{example:I_30(b-2)+23}]
\label{example:I_30(b-2)+23-continued}
Let $(X,0)$ be an icosahedral singularity $I_{30(5-2) + 23}$. Since it is of type $(3,1)$, the rational 4-manifold $Z_2$ is given by
\begin{center}
\includegraphics{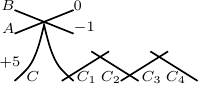}
\end{center}
where $C \cdot C = 5$, $C_1 \cdot C_1 = -2$, $C_2 \cdot C_2 = -2$, $C_3 \cdot C_3 = -3$, $C_4 \cdot C_4 = -3$. There are five models of $Z_2$ and each model is obtained from $\mathbb{CP}^1 \times \mathbb{CP}^1$:

\begin{itemize}
\item Case~(A): \#196 $(1 \times C_3, 2 \times C_4)$, \#197 $(1 \times C_1, 2 \times C_4)$, \#198 $(1 \times C_2, 1 \times C_4)$

\item Case~(B): \#253 $(C_1, C_3; 2 \times C_4)$, New $(C_4, C_4; 1 \times C_2)$
\end{itemize}
Here, in both cases, the number $\# nn$ is the number of the model in Bhupal--Ono~\cite{Bhupal-Ono-2012}, $a_i \times C_i$ means that there are $a_i$ distinct $(-1)$-curves intersecting $C_i$ only in $Z_2$, and we do not write the number of $(-1)$-curves intersecting the cuspidal curve $C$. In Case~(A), there is one $(-1)$-curve intersecting $B$ only, which is not denoted again. In Case~(B), $(C_j, C_k; \cdots)$ means that there are two $(-1)$-curves such that one of them intersects $B$ and $C_j$, and the other intersects $C$ and $C_k$, respectively.
\end{example}

\begin{remark}\label{remark:erratum-Bhypal-Ono}
The model ``New $(C_4, C_4; 1 \times C_2)$'' in Example~\ref{example:I_30(b-2)+23-continued} above is not recorded in Bhupal--Ono~\cite{Bhupal-Ono-2012}. Bhupal--Ono~\cite{Bhupal-Ono-2015} confirms that the number of possible models of $Z_2$ was erroneously claimed to be four in their paper~\cite{Bhupal-Ono-2012} and they check again the list with their own program; Bhupal--Ono~\cite{Bhupal-Ono-2017}.
\end{remark}

Bhupal--Ono~\cite{Bhupal-Ono-2012} shows that there are only finitely many ways to blow up $\mathbb{CP}^2$ or $\mathbb{CP}^1 \times \mathbb{CP}^1$ for constructing $Z_2$. For tetrahedral, octahedral, icosahedral singularities, Bhupal--Ono~\cite[\S5]{Bhupal-Ono-2012} presents all possible ways of blow-ups, or equivalently, all possibilities for the ways that $(-1)$-curves in $Z_2$ intersect $\widetilde{E}_{\infty}$.

\begin{proposition}[{Bhupal--Ono~\cite{Bhupal-Ono-2012}}]
There are only finitely many symplectic deformation types of minimal symplectic fillings for each non-cyclic quotient surface singularities.
\end{proposition}

\subsection{Minimal symplectic fillings of dihedral singularities}
\label{subsection:symplectic-fillings-of-dihedral-singularities}

In this subsection we show that the study of minimal symplectic fillings of dihedral singularities is reduced to the cyclic case.

Let $(X,0)$ be a dihedral singularity whose dual graph of the minimal resolution is given by
\begin{equation}
\label{equation:dual-graph-of-dihedral}
\begin{tikzpicture}
\node[bullet] (005) at (0,0.5) [label=below:{$-2$}, label=above:{$E_1$}] {};
\node[bullet] (0-05) at (0,-0.5) [label=below:{$-2$}, label=above:{$E_2$}] {};

\node[bullet] (10) at (1,0) [label=below:{$-b_3$}] {};
\node[bullet] (20) at (2,0) [label=below:{$-b_4$}] {};

\node[empty] (250) at (2.5,0) [] {};
\node[empty] (30) at (3,0) [] {};

\node[bullet] (350) at (3.5,0) [label=below:{$-b_{r-1}$}] {};
\node[bullet] (450) at (4.5,0) [label=below:{$-b_r$}] {};

\draw [-] (005)--(10);
\draw [-] (0-05)--(10);
\draw [-] (10)--(20);
\draw [-] (20)--(250);
\draw [dotted] (250)--(30);
\draw [-] (30)--(350);
\draw [-] (350)--(450);
\end{tikzpicture}
\end{equation}

We will show in Section~\ref{section:minimal-symplectic-fillings-are-Milnor-fibers} that every minimal symplectic filling of a quotient surface singularity is diffeomorphic to its Milnor fiber. By the way, every Milnor fibers is a general fiber of a smoothing of the corresponding $P$-resolution; cf.~Section~\ref{section:$P$-resolution}. Therefore every minimal symplectic fillings of a quotient surface singularity is a $4$-manifold obtained by rationally blowing-down the corresponding $P$-resolution; cf.~Corollary~\ref{corollary:Milnor-fiber=Rational-blow-down}. Hence it is enough to parametrize Milnor fibers and classify them instead of symplectic fillings.

On the other hand, according to Stevens~\cite[\S7]{Stevens-1991}, every $P$-resolution of a dihedral singularity $(X,0)$ induces a $P$-resolution of the cyclic quotient singularity corresponding to $\frac{n}{a}=[2, b_3, \dotsc, b_r]$, and, conversely, every $P$-resolution of $[2, b_3, \dotsc, b_r]$ is realized as a $P$-resolution on the dihedral singularity $(X,0)$ (or, in two ways for some cases because of the symmetry of the dual graph of the minimal resolution of $(X,0)$). However even though a $P$-resolution of $[2, b_3, \dotsc, b_r]$ is realized on $(X,0)$ in two ways, the corresponding minimal symplectic fillings are diffeomorphic to each other. Hence we have a similar result to Theorem~\ref{theorem:Lisca} on cyclic quotient surface singularities.

\begin{definition}
Let $V_{n,a}(\underline{m})$ be the rationally blown-down $4$-manifold associated to the $P$-resolution of $(X,0)$ corresponding to $\underline{m} \in K_{e}(n/n-a)$ as discussed above.
\end{definition}

\begin{proposition}
\label{proposition:diffeomorphism-type-dihedral}
Let $(X,0)$ be a dihedral singularity whose resolution graph is given as \eqref{equation:dual-graph-of-dihedral}. Let $\frac{n}{a}=[2,b_3,\dotsc,b_r]$.

\begin{enumerate}
\item Any minimal symplectic filling of $(X,0)$ is diffeomorphic to  $V_{n,a}(\underline{m})$ for some $\underline{m} \in K_e(n/n-a)$.

\item The manifold $V_{n,a}(\underline{n})$ is orientation-preserving diffeomorphic to $V_{n,a}(\underline{m})$ if and only if $\underline{n} = \underline{m}$ or, additionally, $\underbar{n} = \underbar{m}^{\circ}$ in case $a^2 \equiv 1 \pmod{n}$, where $\underbar{m}^{\circ}=(m_e, m_{e-1}, \dotsc, m_1)$.
\end{enumerate}
\end{proposition}

\begin{proof}
We need to prove only the second assertion. According to Remark~\ref{remark:the-same--1-data}, $V_{n,a}(\underline{n})$ and $V_{n,a}(\underline{m})$ are diffeomorphic if and only if the two data of $(-1)$-curves intersecting only one components of the compactifying divisor $E_{\infty}$ in the natural compactification $Z$ are the same. On the other hand, we will show in Section~\ref{section:applications} that the data of such $(-1)$-curves is completely determined by the corresponding $P$-resolution, hence, by $\underline{m} \in K_{e}(n/n-a)$. Therefore the second assertion follows as Theorem~\ref{theorem:Lisca} of Lisca on cyclic quotient surface singularities.
\end{proof}

\subsection{Removing duplicate entries in the list of Bhupal--Ono on tetrahedral, octahedral, icosahedral singularities}

Bhupal--Ono~\cite[\S5]{Bhupal-Ono-2012} gives the list of all possible rational surfaces $Z_2$ in Figures~\ref{figure:sequence-TOI-32}, \ref{figure:sequence-TOI-31} for tetrahedral, octahedral, and icosahedral singularities together with the data of $(-1)$-curves intersecting $Z_2$. But they mention the possibility of symplectically deformation equivalent fillings in the list.

We show that there are duplicate entries in their list. That is, we show that certain 6 pairs of entries in the list of Bhupal--Ono~\cite[\S5]{Bhupal-Ono-2012} represent actually the same symplectic fillings; Proposition~\ref{proposition:reduced-BO-list}. We then classify minimal symplectic fillings of tetrahedral, octahedral, and icosahedral singularities up to diffeomorphism in Section~\ref{section:Diffeomorphism-type} using this reduced list.

At first, we briefly explain why there may exist duplicate entries in the list of Bhupal--Ono~\cite[\S5]{Bhupal-Ono-2012}. The duplications occur because of the symmetry of the dual graphs of minimal resolutions of the singularities. Among tetrahedral, octahedral, and icosahedral singularities, tetrahedral singularities $T_m$ with $m=6(b-2)+1$ and $m=6(b-2)+5$ have a symmetry in the minimal resolutions. So they have a symmetry in the compactifying divisors, which may introduce duplicate entries. The following are the dual graph of minimal resolutions and the compactifying divisors of tetrahedral singularities $T_m$ with $m=6(b-2)+1$ and $m=6(b-2)+5$:
\begin{itemize}
\item $T_{6(b-2)+1}$

\begin{tikzpicture}
\node[bullet] (00) at (0,0) [label=below:{$-2$}] {};
\node[bullet] (10) at (1,0) [label=below:{$-2$}] {};

\node[bullet] (20) at (2,0) [label=below:{$-b$}] {};

\node[bullet] (30) at (3,0) [label=below:{$-2$}] {};
\node[bullet] (40) at (4,0) [label=below:{$-2$}] {};

\node[bullet] (21) at (2,1) [label=left:{$-2$}] {};

\draw [-] (00)--(10);
\draw [-] (10)--(20);
\draw [-] (20)--(30);
\draw [-] (30)--(40);
\draw [-] (20)--(21);
\end{tikzpicture}
\qquad \qquad
\includegraphics{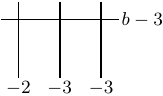}

\item $T_{6(b-2)+5}$

\begin{tikzpicture}
\node[bullet] (10) at (1,0) [label=below:{$-3$}] {};
\node[bullet] (20) at (2,0) [label=below:{$-b$}] {};
\node[bullet] (30) at (3,0) [label=below:{$-3$}] {};

\node[bullet] (21) at (2,1) [label=left:{$-2$}] {};

\draw [-] (10)--(20);
\draw [-] (20)--(30);
\draw [-] (20)--(21);
\end{tikzpicture}
\qquad \qquad
\includegraphics{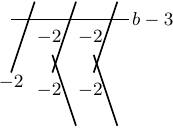}
\end{itemize}

For instance, $T_{6(b-2)+1}$ singularity has two $(-3)$-curves intersecting the node $(b-3)$-curve in the compactifying divisors. So we have two choices of the blow-up center where we blow up and down several times in order to reduce the compactifying divisor to the divisor in $Z_2$ consisting a cuspidal curve with some tails of negative curves. Therefore it may be possible that one has two entries of Bhupal-Ono's list from one minimal symplectic filling if one choose different blowup centers and if the numbers of $(-1)$-curves intersecting the $(-3)$-curves in the compactifying divisor are different.

Now we explain how to find the duplicate entries. In Theorem~\ref{theorem:fillings-diffeomorphic-to-Milnors}, we will show that  any minimal symplectic fillings of non-cyclic quotient surface singularities are diffeomorphic to their Milnor fibers. On the other hand, Milnor fibers are general fibers of smoothings of the corresponding $P$-resolutions (Section~\ref{section:$P$-resolution}) and they can be described as complements of the compactifying divisors embedded in a rational complex surfaces; Section~\ref{section:Milnor-fiber-as-complements}. Furthermore one can compute the data of the $(-1)$-curves intersecting the compactifying divisors from $P$-resolutions via running semi-stable minimal model program; Section~\ref{section:identifying-Milnor-fibers}. Therefore one can reproduce the list of Bhupal--Ono~\cite[\S5]{Bhupal-Ono-2012} if one finds all $P$-resolutions and runs the machinery developed in Section~\ref{section:identifying-Milnor-fibers}. Therefore if there are two $P$-resolutions which are symmetrical to each other, then the corresponding Milnor fibers (hence the corresponding minimal symplectic fillings) are obviously diffeomorphic, even though they introduce two different irreducible components of $\Def(X)$. For example, $T_{6(5-2)+1}$ singularity has the following two $P$-resolutions which are symmetrical to each other.

\begin{equation*}
\begin{tikzpicture}
\node[bullet] (00) at (0,0) [label=below:{$-2$}] {};
\node[rectangle] (10) at (1,0) [label=below:{$-2$}] {};

\node[rectangle] (20) at (2,0) [label=below:{$-5$}] {};

\node[bullet] (30) at (3,0) [label=below:{$-2$}] {};
\node[rectangle] (40) at (4,0) [label=below:{$-2$}] {};

\node[bullet] (21) at (2,1) [label=left:{$-2$}] {};

\draw [-] (00)--(10);
\draw [-] (10)--(20);
\draw [-] (20)--(30);
\draw [-] (30)--(40);
\draw [-] (20)--(21);
\end{tikzpicture}
\qquad \qquad
\begin{tikzpicture}
\node[rectangle] (00) at (0,0) [label=below:{$-2$}] {};
\node[bullet] (10) at (1,0) [label=below:{$-2$}] {};

\node[rectangle] (20) at (2,0) [label=below:{$-5$}] {};

\node[rectangle] (30) at (3,0) [label=below:{$-2$}] {};
\node[bullet] (40) at (4,0) [label=below:{$-2$}] {};

\node[bullet] (21) at (2,1) [label=left:{$-2$}] {};

\draw [-] (00)--(10);
\draw [-] (10)--(20);
\draw [-] (20)--(30);
\draw [-] (30)--(40);
\draw [-] (20)--(21);
\end{tikzpicture}
\end{equation*}

\begin{proposition}[{The reduced Bhupal-Ono's list (HJS~\cite[Proposition~1.4]{HJS-2018})}]
\label{proposition:reduced-BO-list}
There are exactly 6 pairs of entries (of the data of $(-1)$-curves) in the list of Bhupal--Ono~\cite[\S5]{Bhupal-Ono-2012} such that each pairs are associated to the same minimal symplectic fillings. The pairs of duplicate entries are as follows:

    \begin{itemize}
    \item \#7 and \#8
    \item \#10 and \#11
    \item \#129 and \#209
    \item \#132 and \#211
    \item \#212 and \#213
    \item \#135 and \#214
    \end{itemize}
\end{proposition}

\begin{proof}
At first, we find the duplicate entries that come from $P$-resolutions symmetrical to each other. We then show that there are no more duplicate entries.

HJS~\cite{HJS-2018} finds all $P$-resolutions of tetrahedral, octahedral, and icosahedral singularities together with their dual graphs, which reproduces Jan Steven's list \cite{Stevens-1993} of the numbers of $P$-resolutions of each singularities. Furthermore HJS~\cite{HJS-2018} builds up the list of the data of $(-1)$-curves intersecting the compactifying divisors by running the MMP machinery developed in Section~\ref{section:identifying-Milnor-fibers}. Then HJS~\cite{HJS-2018} compares the data of $(-1)$-curves from $P$-resolutions with the data of $(-1)$-curves from minimal symplectic fillings in Bhupal--Ono~\cite[\S5]{Bhupal-Ono-2012}.

According to the list of HJS~\cite{HJS-2018}, there are only 4 tetrahedral singularities which have pairs of two $P$-resolutions symmetrical to each other: $T_{6(5-2)+1}$, $T_{6(6-2)+1}$, $T_{6(3-2)+5}$, $T_{6(3-2)+5}$, $T_{6(4-2)+5}$ singularities. Below are pairs of $P$-resolutions symmetrical to each other and the data of $(-1)$-curves intersecting the compactifying divisors for each singularities, where the label $T_{6(b-2)+1}[n]$ denotes the $n$-th $P$-resolution of $T_{6(b-2)+1}$ singularity in the list of HJS~\cite{HJS-2018}, the label $BO \#m$ denotes the $m$-th symplectic filling in the list of Bhupal--Ono~\cite[\S5]{Bhupal-Ono-2012}, and the red curves in the compactifying divisors are $(-1)$-curves.

For $T_{6(5-2)+1}$ singularity,

\begin{itemize}
\item $T_{6(5-2)+1}[3] = BO \#7$

\begin{tikzpicture}
\node[bullet] (00) at (0,0) [label=below:{$-2$}] {};
\node[rectangle] (10) at (1,0) [label=below:{$-2$}] {};

\node[rectangle] (20) at (2,0) [label=below:{$-5$}] {};

\node[bullet] (30) at (3,0) [label=below:{$-2$}] {};
\node[rectangle] (40) at (4,0) [label=below:{$-2$}] {};

\node[bullet] (21) at (2,1) [label=left:{$-2$}] {};

\draw [-] (00)--(10);
\draw [-] (10)--(20);
\draw [-] (20)--(30);
\draw [-] (30)--(40);
\draw [-] (20)--(21);
\end{tikzpicture}
\qquad
\includegraphics[scale=0.8]{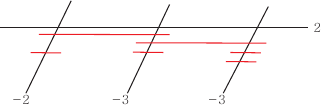}

\item $T_{6(5-2)+1}[4] = BO \#8$

\begin{tikzpicture}
\node[rectangle] (00) at (0,0) [label=below:{$-2$}] {};
\node[bullet] (10) at (1,0) [label=below:{$-2$}] {};

\node[rectangle] (20) at (2,0) [label=below:{$-5$}] {};

\node[rectangle] (30) at (3,0) [label=below:{$-2$}] {};
\node[bullet] (40) at (4,0) [label=below:{$-2$}] {};

\node[bullet] (21) at (2,1) [label=left:{$-2$}] {};

\draw [-] (00)--(10);
\draw [-] (10)--(20);
\draw [-] (20)--(30);
\draw [-] (30)--(40);
\draw [-] (20)--(21);
\end{tikzpicture}
\qquad
\includegraphics[scale=0.8]{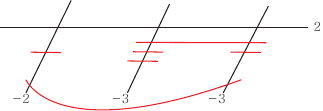}
\end{itemize}

For $T_{6(6-2)+1}$ singularity,

\begin{itemize}
\item $T_{6(6-2)+1}[2] = BO \#10$

\begin{tikzpicture}
\node[rectangle] (00) at (0,0) [label=below:{$-2$}] {};
\node[rectangle] (10) at (1,0) [label=below:{$-2$}] {};

\node[rectangle] (20) at (2,0) [label=below:{$-6$}] {};

\node[bullet] (30) at (3,0) [label=below:{$-2$}] {};
\node[rectangle] (40) at (4,0) [label=below:{$-2$}] {};

\node[bullet] (21) at (2,1) [label=left:{$-2$}] {};

\draw [-] (00)--(10);
\draw [-] (10)--(20);
\draw [-] (20)--(30);
\draw [-] (30)--(40);
\draw [-] (20)--(21);
\end{tikzpicture}
\qquad
\includegraphics[scale=0.8]{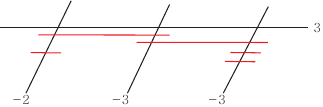}

\item $T_{6(6-2)+1}[3] = BO \#11$

\begin{tikzpicture}
\node[rectangle] (00) at (0,0) [label=below:{$-2$}] {};
\node[bullet] (10) at (1,0) [label=below:{$-2$}] {};

\node[rectangle] (20) at (2,0) [label=below:{$-6$}] {};

\node[rectangle] (30) at (3,0) [label=below:{$-2$}] {};
\node[rectangle] (40) at (4,0) [label=below:{$-2$}] {};

\node[bullet] (21) at (2,1) [label=left:{$-2$}] {};

\draw [-] (00)--(10);
\draw [-] (10)--(20);
\draw [-] (20)--(30);
\draw [-] (30)--(40);
\draw [-] (20)--(21);
\end{tikzpicture}
\qquad
\includegraphics[scale=0.8]{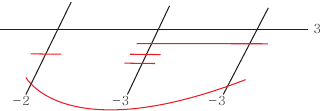}
\end{itemize}

For $T_{6(3-2)+5}$ singularity,

\begin{itemize}
\item $T_{6(3-2)+5}[2] = BO \#209$

\begin{tikzpicture}
\node[rectangle] (10) at (1,0) [label=below:{$-3$}] {};
\node[rectangle] (20) at (2,0) [label=below:{$-3$}] {};
\node[bullet] (30) at (3,0) [label=below:{$-3$}] {};

\node[bullet] (21) at (2,1) [label=left:{$-2$}] {};

\draw [-] (10)--(20);
\draw [-] (20)--(30);
\draw [-] (20)--(21);
\end{tikzpicture}
\qquad
\includegraphics[scale=0.8]{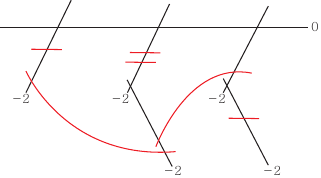}

\item $T_{6(3-2)+5}[3] = BO \#129$

\begin{tikzpicture}
\node[bullet] (10) at (1,0) [label=below:{$-3$}] {};
\node[rectangle] (20) at (2,0) [label=below:{$-3$}] {};
\node[rectangle] (30) at (3,0) [label=below:{$-3$}] {};

\node[bullet] (21) at (2,1) [label=left:{$-2$}] {};

\draw [-] (10)--(20);
\draw [-] (20)--(30);
\draw [-] (20)--(21);
\end{tikzpicture}
\qquad
\includegraphics[scale=0.8]{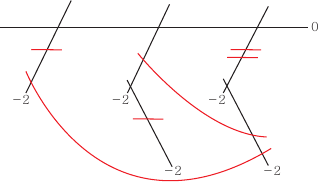}
\end{itemize}

For $T_{6(4-2)+5}$ singularity,

\begin{itemize}
\item $T_{6(4-2)+5}[3] = BO \#211$

\begin{tikzpicture}
\node[rectangle] (00) at (0,0) [label=below:{$-4$}] {};
\node[bullet] (10) at (1,0) [label=below:{$-1$}] {};
\node[rectangle] (20) at (2,0) [label=below:{$-5$}] {};
\node[bullet] (30) at (3,0) [label=below:{$-3$}] {};

\node[rectangle] (21) at (2,1) [label=left:{$-2$}] {};

\draw [-] (00)--(10);
\draw [-] (10)--(20);
\draw [-] (20)--(30);
\draw [-] (20)--(21);
\end{tikzpicture}
\qquad
\includegraphics[scale=0.8]{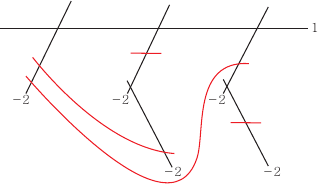}

\item $T_{6(4-2)+5}[5] = BO \#132$

\begin{tikzpicture}
\node[bullet] (10) at (1,0) [label=below:{$-3$}] {};
\node[rectangle] (20) at (2,0) [label=below:{$-5$}] {};
\node[bullet] (30) at (3,0) [label=below:{$-1$}] {};
\node[rectangle] (40) at (4,0) [label=below:{$-4$}] {};

\node[rectangle] (21) at (2,1) [label=left:{$-2$}] {};

\draw [-] (10)--(20);
\draw [-] (20)--(30);
\draw [-] (30)--(40);
\draw [-] (20)--(21);

\end{tikzpicture}
\qquad
\includegraphics[scale=0.8]{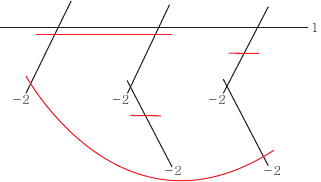}
\end{itemize}
and
\begin{itemize}
\item $T_{6(4-2)+5}[4] = BO \#213$

\begin{tikzpicture}
\node[rectangle] (00) at (0,0) [label=below:{$-4$}] {};
\node[bullet] (10) at (1,0) [label=below:{$-1$}] {};
\node[rectangle] (20) at (2,0) [label=below:{$-5$}] {};
\node[rectangle] (30) at (3,0) [label=below:{$-3$}] {};

\node[rectangle] (21) at (2,1) [label=left:{$-2$}] {};

\draw [-] (00)--(10);
\draw [-] (10)--(20);
\draw [-] (20)--(30);
\draw [-] (20)--(21);
\end{tikzpicture}
\qquad
\includegraphics[scale=0.8]{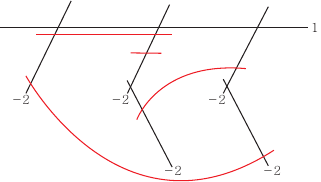}

\item $T_{6(4-2)+5}[6] = BO \#212$

\begin{tikzpicture}
\node[rectangle] (10) at (1,0) [label=below:{$-3$}] {};
\node[rectangle] (20) at (2,0) [label=below:{$-5$}] {};
\node[bullet] (30) at (3,0) [label=below:{$-1$}] {};
\node[rectangle] (40) at (4,0) [label=below:{$-4$}] {};

\node[rectangle] (21) at (2,1) [label=left:{$-2$}] {};

\draw [-] (10)--(20);
\draw [-] (20)--(30);
\draw [-] (30)--(40);
\draw [-] (20)--(21);

\end{tikzpicture}
\qquad
\includegraphics[scale=0.8]{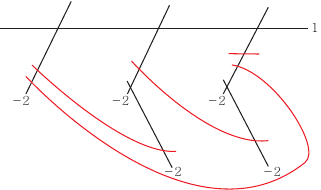}
\end{itemize}

For $T_{6(5-2)+5}$ singularity,

\begin{itemize}
\item $T_{6(5-2)+5}[3] = BO \#214$

\begin{tikzpicture}
\node[rectangle] (10) at (1,0) [label=below:{$-3$}] {};
\node[rectangle] (20) at (2,0) [label=below:{$-5$}] {};
\node[bullet] (30) at (3,0) [label=below:{$-3$}] {};

\node[rectangle] (21) at (2,1) [label=left:{$-2$}] {};

\draw [-] (10)--(20);
\draw [-] (20)--(30);
\draw [-] (20)--(21);
\end{tikzpicture}
\qquad
\includegraphics[scale=0.8]{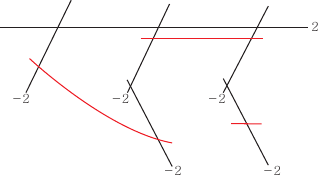}

\item $T_{6(5-2)+5}[4] = BO \#135$

\begin{tikzpicture}
\node[bullet] (10) at (1,0) [label=below:{$-3$}] {};
\node[rectangle] (20) at (2,0) [label=below:{$-5$}] {};
\node[rectangle] (30) at (3,0) [label=below:{$-3$}] {};

\node[rectangle] (21) at (2,1) [label=left:{$-2$}] {};

\draw [-] (10)--(20);
\draw [-] (20)--(30);
\draw [-] (20)--(21);
\end{tikzpicture}
\qquad
\includegraphics[scale=0.8]{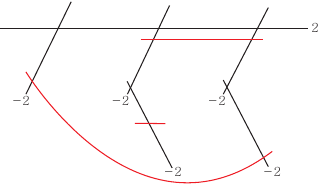}
\end{itemize}

We prove in Theorem~\ref{theorem:diffeomorphism-type} that minimal symplectic fillings in the reduced Bhupal-Ono's list are non-diffeomorphic to each other. Therefore the redundancy in the list of Bhupal-Ono comes only from $P$-resolutions symmetrical to each other.
\end{proof}

\subsection{Milnor fillable contact structure on the complement}

In Proposition~\ref{proposition:symplectic-filling-is-rational-I} and Proposition~\ref{proposition:symplectic-filling-is-rational-II} above, we see that a minimal symplectic filling $W$ of a quotient surface singularity $(X,0)$ is given as a complement $Z-\nu(E_{\infty})$ where $E_{\infty} \subset Z$ is the compactifying divisor of $(X,0)$ embedded in a rational symplectic 4-manifold $Z$. Conversely, Lisca~\cite{Lisca-2008} and Bhupal--Ono~\cite{Bhupal-Ono-2012} proves that if the compactifying divisor $E_{\infty}$ is embedded in a rational symplectic 4-manifold $Z$, then the boundary of the complement $Z-\nu(E_{\infty})$ admits a Milnor fillable contact structure; hence, the complement $Z-\nu(E_{\infty})$ is indeed a symplectic filling of $(X,0)$. But Lisca~\cite{Lisca-2008} and Bhupal--Ono~\cite{Bhupal-Ono-2012} used different methods. Here we present a uniform way to show the existence of Milnor fillable contact structure on $\partial(Z-\nu(E_{\infty}))$.

\begin{proposition}
Assume that the compactifying divisor $E_{\infty}$ of a quotient surface singularity $(X,0)$ is symplectically embedded in a rational symplectic 4-manifold $Z$. Then the boundary of the complement $Z-\nu(E_{\infty})$ admits a Milnor fillable contact structure induced from the symplectic structure of $Z$. Hence $Z-\nu(E_{\infty})$ is a symplectic filling of $(X,0)$.
\end{proposition}

\begin{proof}
Let $E_{\infty} = \sum E_i$, where $E_i$ is a symplectically embedded curve embedded in $Z$. By Bhupal--Ono~\cite[p 34]{Bhupal-Ono-2012}, there is an inward Liouville vector field on $\partial(\nu(E_{\infty}))$. So the inward Liouville vector field induces a contact structure $\alpha_1$ on $\partial(Z-\nu(E_{\infty}))$. The contact structure $\alpha_1$ induces the $\SpinC$-structure which extends to $\nu(E_{\infty})$ as a $\SpinC$-structure $s_1$ with the properties that $c_1(s_1)$ satisfies the adjunction equality for all $E_i$.

On the other hand, let $\widetilde{X}$ be the blown-up of the Hirzebruch surface containing the configuration described in Equation~\eqref{equation:blowing-up-F_1} or in Remark~\ref{remark:model-for-widetilde(X)} defined in Section~\ref{section:Compactifying-divisor}. According to Gay--Stipsicz~\cite{Gay-Stipzicz-2009}, $\partial(\nu(E_{\infty})) \subset \widetilde{X}$ also has an inward Liouville vector field, which induces Milnor fillable contact $\alpha_2$ on $\partial(\nu(E_0)) = \partial(\widetilde{X} - \nu(E_{\infty}))$ by H.~Park--Stipsicz~\cite{HPark-Stipsicz-2014}. The contact structure $\alpha_2$ also induces a $\SpinC$-structure which extends to $\nu(E_{\infty})$ as a $\SpinC$-structure $s_2$ with the properties that $c_1(s_2)$ satisfies the adjunction equality for all $E_i$.

Then $c_1(s_1)=c_1(s_2)$ on $\nu(E_{\infty})$. Furthermore, since $\nu(E_{\infty})$ is simply-connected, we have $s_1=s_2$ on $\nu(E_{\infty})$. Therefore the two restrictions of the $\SpinC$-structure $s_1$ to $\partial(Z-\nu(E_{\infty}))$ and the $\SpinC$-structure $s_2$ to $\partial(\widetilde{X}-\nu(E_{\infty}))$ are the same. Note that $\partial(\nu(E_0))=\partial(Z-\nu(E_{\infty}))=\partial(\widetilde{X}-\nu(E_{\infty}))$ and, furthermore, $\partial(\nu(E_0))$ is an $L$-space by N{\'e}methi~\cite[Theorem~6.3, 8.3]{Nemethi-2005} and Gorsky--N{\'e}methi~\cite[Theorem~4]{Gorsky-Nemethi-2014}. Then two contact structures $\alpha_1$ and $\alpha_2$ are contactomorphic by Gay--Stipsicz~\cite[Theorem~3.1]{Gay-Stipsicz-2007}. Therefore $\alpha_1$ is also Milnor fillable contact structure on $Z-\nu(E_{\infty})$.
\end{proof}

\section{Classification of minimal symplectic fillings}
\label{section:Diffeomorphism-type}

We classify minimal symplectic fillings of \emph{tetrahedral/octahedral/icosahedral} singularities up to diffeomorphism in Theorem~\ref{theorem:diffeomorphism-type} based on the reduced list of Bhupal--Ono~\cite{Bhupal-Ono-2012} (Proposition~\ref{proposition:reduced-BO-list}).

\subsection{Small Seifert fibred $3$-manifolds}

We briefly recall the definition and some properties of small Seifert fibred $3$-manifold to fix conventions and notations. For details we refer to Orlik--Wagreich~\cite{Orlik-Wagreich-1971}, Orlik~\cite{Orlik-1972}, Jankins--Neumann~\cite{Jankins-Neumann-1983}, for example.

A \emph{small Seifert fibred $3$-manifold} $M=M(-b; (\alpha_1, \beta_1), (\alpha_2, \beta_2), (\alpha_3, \beta_3))$ is a $3$-manifold whose surgery diagram is given in Figure~\ref{figure:small-Seifert-3-manifold}, where $-b \in \mathbb{Z}$, $0 < \beta_i < \alpha_i$ and $(\alpha_i, \beta_i)=1$ for all $i=1,2,3$. The data
\begin{equation*}
(-b; (\alpha_1, \beta_1), (\alpha_2, \beta_2), (\alpha_3, \beta_3))
\end{equation*}
is called the \emph{normalized Seifert invariants} of $M$.

\begin{figure}
\centering
\includegraphics{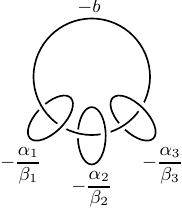}
\caption{Small Seifert $3$-manifolds}
\label{figure:small-Seifert-3-manifold}
\end{figure}

A small Seifert fibred $3$-manifold $M=M(-b; (\alpha_1, \beta_1), (\alpha_2, \beta_2), (\alpha_3, \beta_3))$ can also be described as follows. Let $S$ be an $S^1$-bundle over $S^2$ with Euler number $-b$, that is, $S$ is the lens space $L(b,1)$. Take three disjoint small $2$-disks $D_1$, $D_2$, $D_3$ on the base $S^2$ such that the $S^1$-bundle $S$ is locally trivial over $D_i$, i.e., $S \cong D_i \times S^1$ over $D_i$  for all $i=1, 2, 3$. Then the small Seifert fibred $3$-manifold $M$ is constructed by removing the solid tori $D_i \times S^1$ ($i=1, 2, 3$) in $S$ and gluing back solid tori $T_i=D^2 \times S^1$ ($i=1,2,3$) to each $\partial(D_i \times S^1)$ ($i=1,2,3$), respectively, via the map
\begin{equation*}
\mu_i \mapsto \alpha_i Q_i - \beta_i H_i,
\end{equation*}
where $\mu_i$ is a meridian of the solid torus $D^2 \times S^1$, $Q_i$ is a meridian of the solid torus $D_i \times S^1$, and $H_i$ is a fiber of the $S^1$-bundle $S$ lying on $\partial(D_i \times S^1)$. That is,
\begin{equation*}
M = \left(S - \cup_{i=1}^{3} (D_i \times S^1) \right) \cup_{\mu_1} T_1 \cup_{\mu_2} T_2 \cup_{\mu_3} T_3.
\end{equation*}
Here we call $S - \cup_{i=1}^{3} (D_i \times S^1)$ the \emph{body} of $M$. On the other hand, the complement $M - \left( S - \cup_{i=1}^{3} (D_i \times S^1) \right)$ of the body in $M$ consists of open subsets of three lens spaces: Explicitly,
\begin{equation*}
M - \left( S - \cup_{i=1}^{3} (D_i \times S^1) \right)
= \cup_{i=1}^{3} \left(L(\alpha_i, \beta_i) - (D^2 \times S^1)\right).
\end{equation*}
The three open subsets
\begin{equation*}
L(\alpha_i, \beta_i) - D^2 \times S^1
\end{equation*}
of $M$ are called the \emph{arms} of $M$.

The center $L_i = \{0\} \times S^1$ of $D^2 \times S^1$ glued to $S - (D_i \times S^1)$ becomes a \emph{multiple fiber} with multiplicity $\alpha_i$ of $M$, that is, $\alpha_i L_i = H$, where $H$ is a general fiber of $M$, and
\begin{equation*}
Q_i = \beta_i L_i
\end{equation*}
in $H_1(M; \mathbb{Z})$. We call this multiple fiber $L_i$ the \emph{exceptional fiber} of type $(\alpha_i, \beta_i)$.

We summarize some well-known properties of a small Seifert fibred $3$-manifold.

\begin{lemma}\label{lemma:properties-Seifert}
Let $M=M(-b; (\alpha_1, \beta_1), (\alpha_2, \beta_2), (\alpha_3, \beta_3))$ be a small Seifert fibred $3$-manifold.

\begin{enumerate}
\item Let $H$ be a general fiber of $M$. We may assume that $H$ is contained in the body of $M$, that is,
\begin{equation*}
H \subset S - \cup_{i=1}^{3} (D_i \times S^1).
\end{equation*}
So $H$ may be considered as a $S^1$-fiber over the central lens space $S=L(b,1)$. Then it is a generator of $H_1(S; \mathbb{Z})$.

\item The circle $Q_i=\beta_i L_i$ is contained in the corresponding arm $L(\alpha_i,\beta_i) - (D^2 \times S^1)$. Then $Q_i$ is a generator of $H_1(L(\alpha_i,\beta_i); \mathbb{Z})$.

\item The first homology $H_1(M; \mathbb{Z})$ of $M$ is generated by $Q_1, Q_2, Q_3, H$ with the relations
    \begin{align*}
    &bH+Q_1+Q_2+Q_3=0, \\
    &\text{$\beta_j H - \alpha_j Q_j = 0$ for all $i=1,2,3$}.
    \end{align*}
\end{enumerate}
\end{lemma}

The small Seifert fibred $3$-manifold $M$ can be regarded as a boundary of a plumbed $4$-manifold. Let $\Gamma_0$ and $\Gamma_{\infty}$ be the following star-shape graphs
\begin{equation*}
\begin{tikzpicture}
\node[] () at (-2.5,2.5) [] {$\Gamma_0$};

\node[bullet] (-250) at (-2.5,0) [label=below:{$-b_{1r_1}$}] {};
\node[empty] (-20) at (-2,0) [] {};
\node[empty] (-150) at (-1.5,0) [] {};
\node[bullet] (-10) at (-1,0) [label=below:{$-b_{11}$}] {};

\node[bullet] (00) at (0,0) [label=below:{$-b$}] {};

\node[bullet] (01) at (0,1) [label=left:{$-b_{21}$}] {};
\node[empty] (015) at (0, 1.5) [] {};
\node[empty] (02) at (0,2) [] {};
\node[bullet] (025) at (0,2.5) [label=left:{$-b_{2r_2}$}] {};

\node[bullet] (10) at (1,0) [label=below:{$-b_{31}$}] {};
\node[empty] (150) at (1.5,0) [] {};
\node[empty] (20) at (2,0) [] {};
\node[bullet] (250) at (2.5,0) [label=below:{$-b_{3r_3}$}] {};

\draw [-] (00)--(-10);
\draw [-] (-10)--(-150);
\draw [dotted] (-10)--(-250);
\draw [-] (-20)--(-250);

\draw [-] (00)--(01);
\draw [-] (01)--(015);
\draw [dotted] (01)--(025);
\draw [-] (02)--(025);

\draw [-] (00)--(10);
\draw [-] (10)--(150);
\draw [dotted] (10)--(250);
\draw [-] (20)--(250);
\end{tikzpicture}
\quad
\begin{tikzpicture}
\node[] () at (-2.5,2.5) [] {$\Gamma_{\infty}$};

\node[bullet] (-250) at (-2.5,0) [label=below:{$-a_{1e_1}$}] {};
\node[empty] (-20) at (-2,0) [] {};
\node[empty] (-150) at (-1.5,0) [] {};
\node[bullet] (-10) at (-1,0) [label=below:{$-a_{11}$}] {};

\node[bullet] (00) at (0,0) [label=below:{$b-3$}] {};

\node[bullet] (01) at (0,1) [label=left:{$-a_{21}$}] {};
\node[empty] (015) at (0, 1.5) [] {};
\node[empty] (02) at (0,2) [] {};
\node[bullet] (025) at (0,2.5) [label=left:{$-a_{2e_2}$}] {};

\node[bullet] (10) at (1,0) [label=below:{$-a_{31}$}] {};
\node[empty] (150) at (1.5,0) [] {};
\node[empty] (20) at (2,0) [] {};
\node[bullet] (250) at (2.5,0) [label=below:{$-a_{3e_3}$}] {};

\draw [-] (00)--(-10);
\draw [-] (-10)--(-150);
\draw [dotted] (-10)--(-250);
\draw [-] (-20)--(-250);

\draw [-] (00)--(01);
\draw [-] (01)--(015);
\draw [dotted] (01)--(025);
\draw [-] (02)--(025);

\draw [-] (00)--(10);
\draw [-] (10)--(150);
\draw [dotted] (10)--(250);
\draw [-] (20)--(250);
\end{tikzpicture}
\end{equation*}
where
\begin{align*}
\frac{\alpha_i}{\beta_i} &= [b_{i1}, \dotsc, b_{ir_i}], \qquad b_{ij} \ge 2 \\
\frac{\alpha_i}{\alpha_i - \beta_i} &= [a_{i1}, \dotsc, a_{ie_i}], \qquad a_{ij} \ge 2
\end{align*}
for $i=1,2,3$.

\begin{lemma}\label{lemma:V_0-V_infty}
Let $V_0$ and $V_{\infty}$ be the $4$-manifold obtained by plumbing $\Gamma_0$ and $\Gamma_{\infty}$, respectively. Then
\begin{align*}
\partial(V_0) &= M(-b; (\alpha_1, \beta_1), (\alpha_2, \beta_2), (\alpha_3, \beta_3)) =: M_0 ,\\
\partial(V_{\infty}) &= M(b-3; (\alpha_1, \alpha_1-\beta_1), (\alpha_2, \alpha_2-\beta_2), (\alpha_3, \alpha_3-\beta_3)) =: M_{\infty}.
\end{align*}
Furthermore there is a diffeomorphism
\begin{equation*}
M_0 \to - M_{\infty}
\end{equation*}
which sends a general fiber of $M_0$ to that of $M_{\infty}$.
\end{lemma}

\subsection{Classification up to diffeomorphism}

We classify minimal symplectic fillings of a tetrahedral/octahedral/icosahedral singularity in Theorem~\ref{theorem:diffeomorphism-type}.

But we assume that $(X,0)$ is just a \emph{non-cyclic} quotient surface singularity (that is, $(X,0)$ may be a dihedral singularity) for a while.

Let $E_0$ be the exceptional divisor of its minimal resolution given by the star-shape graph $\Gamma_0$ in Figure~\ref{figure:non-cyclic-minimal-resolution} and let $E_{\infty}$ be the compactifying divisor of $(X,0)$ whose dual graph $\Gamma_{\infty}$ is given in Figure~\ref{figure:non-cyclic-compactifying-divisor}. Let $V_0$ and $V_{\infty}$ be the plumbed $4$-manifolds from $E_0$ (or $\Gamma_0$) and $E_{\infty}$ (or $\Gamma_{\infty}$), respectively. One may identify the natural compactification $Z = W \cup_{L} (\widetilde{X} - \nu(E_0))$ of $W$ defined in Definition~\ref{definition:natural-compactification-of-filling-non-cyclic} as
\begin{equation*}
Z = W \cup_{\partial{W} = - \partial{V_{\infty}}} V_{\infty}.
\end{equation*}
because $\partial{W} = - \partial{V_{\infty}}$.

\begin{lemma}
\label{lemma:V_infty->V_infty}
With the same notations as in Lemma~\ref{lemma:V_0-V_infty}, there is a orien\-tation-preserving diffeomorphism
\begin{equation*}
G_0 \colon V_{\infty} \to V_{\infty}
\end{equation*}
such that its restriction $g_0 \colon M_{\infty} \to M_{\infty}$ to the boundaries induces a homomorphism
\begin{equation*}
{g_0}_{\ast} \colon H_1(M_{\infty}; \mathbb{Z}) \to H_1(M_{\infty}; \mathbb{Z})
\end{equation*}
satisfying ${g_0}_{\ast}([H]) = - [H]$ and ${g_0}_{\ast}([Q_i]) = -[Q_i]$ for all $i$.
\end{lemma}

\begin{proof}
Notice that $V_{\infty}$ is the plumbing of $D^2$-bundles $S(v)$ over $S^2$ corresponding to each vertex $v$ of the graph $E_{\infty}$. We decompose $S(v)$ as the following standard way
\begin{equation*}
S(v) = (B_1^2 \times D^2) \cup_{\phi_n} (B_2^2 \times D^2)
\end{equation*}
where $B_i^2 = \{(x,y) \in \mathbb{R}^2 \mid x^2 + y^2 \le 1 \}$ ($i=1,2,$) and $D^2 = \{(u,v) \in \mathbb{R}^2 \mid u^2 + v^2 \le 1\}$, and the pasting map $\phi_n \colon (\partial B_1^2) \times D^2 \to (\partial B_2^2) \times D^2$ satisfies
\begin{equation*}
\phi_n \mid _{(\partial B_1^2) \times (\partial D^2)} = \begin{pmatrix}
-1 & 0 \\ n & 1
\end{pmatrix}
\end{equation*}
if $v \cdot v=n$. One can define a diffeomorphism $G_0(v) \colon S(v) \to S(v)$ such that
\begin{equation*}
((x,y), (u,v)) \mapsto ((x,-y), (u,-v))
\end{equation*}
on $B_1 \times D^2 \to B_1 \times D^2$. Then it is easy to show that the above diffeomorphisms $G_0(v)$ can be glued to a diffeomorphism $G_0 \colon V_{\infty} \to V_{\infty}$ with the desired properties.
\end{proof}

\begin{proposition}
\label{propsition:Extending-diffeomorphism}
Let $W$ and $W'$ be (minimal) symplectic fillings of a non-cyclic quotient surface singularity  $(X,0)$. Suppose that there is a diffeomorphism $F \colon W' \to W$. Then there is an orientation-preserving diffeomorphism
\begin{equation*}
\widetilde{F} \colon Z' \to Z
\end{equation*}
between the natural compactifications of $W'$ and $W$ such that $\widetilde{F}|_{W'} = F$ and its induced map $\widetilde{F}_{\ast} \colon H_2(Z';\mathbb{Z}) \to H_2(Z;\mathbb{Z})$ satisfies $\widetilde{F}_{\ast}([C_i]) = [C_i]$ for all $i$ or $\widetilde{F}_{\ast}([C_i]) = -[C_i]$ for all $i$, where $C_i$ is the irreducible component of the compactifying divisor $E_{\infty} \subset Z', Z$.
\end{proposition}

\begin{proof}
The restriction $f := F|_{\partial{W'}} \colon \partial{W'} \to \partial{W}$ is a diffeomorphism between small Seifert fibred $3$-manifolds. Then it follows by a general theory of diffeomorphisms of small Seifert fibred $3$-manifolds (cf. Waldhausen~\cite{Waldhausen-1968}, Boileau--Otal~\cite[Proposition~3.1]{Boileau-Otal-1991}, Birman--Rubinstein~\cite[\S0 Main~Theorem]{Birman-Rubinstein-1984}, Rubinstein~\cite[Theorems~5,~6]{Rubinstein-1979}) that $f$ is isotopic to a diffeomorphism which sends fibers of $\partial{W'}$ to that of $\partial{W}$. Therefore one may assume that $f$ is a fiber-preserving diffeomorphism. Then there is a fiber-preserving diffeomorphism $g \colon \partial{V_{\infty}} \to \partial{V_{\infty}}$. Therefore it is enough to show that $g$ can be extended to a diffeomorphism $G \colon V_{\infty} \to V_{\infty}$ with desired properties.

Let $M_{\infty}=\partial{V_{\infty}}$ for simplicity. Then $M_{\infty}=M(b-3; (\alpha_1, \alpha_1-\beta_1), (\alpha_2, \alpha_2-\beta_2), (\alpha_3, \alpha_3-\beta_3))$. Since $g$ preserves fibers of $M_{\infty}$, by arranging the indices, one may assume that $g$ sends the exceptional fiber of type $(\alpha_i, \alpha_i-\beta_i)$ to the same exceptional fiber of type $(\alpha_i, \alpha_i-\beta_i)$. Then its restriction
\begin{equation*}
\overline{g} \colon S-\cup_{i=1}^{3} (D_i \times S^1) \to S-\cup_{i=1}^{3} (D_i \times S^1)
\end{equation*}
is a fiber-preserving diffeomorphism. Then $\overline{g}$ is extended to a diffeomorphism $\overline{G}$ on the $D^2$-bundle over $S^2 - \cup_{i=1}^{3} D_i$, which has three boundary components. We need to extend $\overline{G}$ into the plumbed 4-manifolds corresponding to the linear graphs $[a_{i1}, \dotsc, a_{ie}]$ for $i=1,2,3$. Then, it follows by Bonahon~\cite[Theorem~3]{Bonahon-1983} that the diffeomorphism $G$ can be extended to the plumed $4$-manifolds corresponding to the linear graphs $[a_{i1}, \dotsc, a_{ie_i}]$. Hence, we get a diffeomorphism $\widetilde{F} \colon Z' \to Z$ such that $\widetilde{F}|_{W'} = F$.

We now show that $\widetilde{F}_{\ast} \colon H_2(Z';\mathbb{Z}) \to H_2(Z;\mathbb{Z})$ satisfies the desired properties.

Since $g$ is a diffeomorphism preserving fibers, $g_{\ast} \colon H_1(V_{\infty};\mathbb{Z}) \to H_1(V_{\infty};\mathbb{Z})$ sends $[H]$ to $[H]$ or $-[H]$. Assume first that $g_{\ast}([H])=[H]$.

\medskip

\textit{Case~1: Dihedral singularities}.

\medskip

By Rubinstein~\cite[Theorem~6]{Rubinstein-1979}, since $g_{\ast}([H])=[H]$ in $H_1(M_{\infty};\mathbb{Z})$ for $g \colon M_{\infty} \to M_{\infty}$  and $g$ sends the exceptional fiber of type $(\alpha_i, \alpha_i-\beta_i)$ to the same exceptional fiber of type $(\alpha_i, \alpha_i-\beta_i)$, $g$ is isotopic to the identity map. Therefore it follows by Bonahon~\cite{Bonahon-1983} that $g$ can be extended to the interior of each $L(\alpha_i, \alpha_i - \beta_i)$ and, furthermore, $[C_i] \mapsto [C_i]$.

\medskip

\textit{Case~2: Tetrahedral, octahedral, icosahedral singularities}.

\medskip

We claim that $g_{\ast}([Q_i])=[Q_i]$ for all $i$.

Note that any diffeomorphism $L(\alpha_i, \alpha_i - \beta_i) \to L(\alpha_i, \alpha_i - \beta_i)$ is completely determined by the image of $[Q_i]$ because $[Q_i]$ is a generator of $H_1(L(\alpha_i, \alpha_i - \beta_i);\mathbb{Z})$. There are only four possibilities (cf. Bonahon~\cite[Theorem~3]{Bonahon-1983}, Lisca~\cite[Section~7]{Lisca-2008}): $[Q_i]\mapsto \pm [Q_i]$, or, additionally $[Q_i] \mapsto \pm(\alpha_i - \beta_i)[Q_i]$ in case of $(\alpha_i - \beta_i)^2 \equiv 1 \pmod{\alpha_i}$. But one can easily check that $(\alpha_i - \beta_i)^2 \equiv 1 \pmod{\alpha_i}$ if and only if $\alpha_i - \beta_i = \alpha_i - 1$ in Case~2. Then $(\alpha_i - \beta_i)[Q_i]=-[Q_i]$ in $H_1(L(\alpha_i, \alpha_i - \beta_i);\mathbb{Z})$. Therefore one may assume that there are only two possibilities: $[Q_i]\mapsto [Q_i]$ or $-[Q_i]$. Suppose that $[Q_i]\mapsto -[Q_i]$. Then there is an extension $G\colon V_{\infty} \to V_{\infty}$ sends $[H] \mapsto [H]$ but $[Q_i] \mapsto -[Q_i]$, which contradicts that $G \colon V_{\infty} \to V_{\infty}$ is a orientation-preserving diffeomorphism. Therefore the claim follows. Then one may apply the same argument in Case~1; then, the proof of Case 2 is done.

Assume now that $g_{\ast}([H])=-[H]$.

By Lemma~\ref{lemma:V_infty->V_infty}, there exists a diffeomorphism
\begin{equation*}
G_0 \colon V_{\infty} \to V_{\infty}
\end{equation*}
such that ${g_0}_{\ast}([H])=-[H]$ and ${g_0}_{\ast}([Q_i])=-[Q_i]$, where ${g_0}={G_0}|_{\partial{V_{\infty}}}$. Since $(g_0 \circ g)_{\ast}([H]) = [H]$,  one can apply the above arguments to $g_0 \circ g$ so that there is a diffeomorphism $J \colon V_{\infty} \to V_{\infty}$ such that $J|_{\partial{V_{\infty}}} = g_0 \circ g$. Take $G = G_0^{-1} \circ J$.
\end{proof}

We now classify minimal symplectic fillings of tetrahedral, octahedral, and icosahedral singularities.

\begin{theorem}\label{theorem:diffeomorphism-type}
Let $(X,0)$ be a tetrahedral, octahedral, or icosahedral singularity. Then the minimal symplectic fillings of $(X,0)$ in the reduced list of Bhupal-Ono (cf. Proposition~\ref{proposition:reduced-BO-list}) are non-diffeomorphic to each other.
\end{theorem}

\begin{proof}
We assume that $(X,0)$ is a non-cyclic quotient surface singularity for a while.

Let $W$ and $W'$ be two minimal symplectic fillings of $(X,0)$. Assume that $W$ and $W'$ are diffeomorphic. By Proposition~\ref{propsition:Extending-diffeomorphism}, there is a extension $\widetilde{F} \colon Z' \to Z$. We apply to $Z$ and $Z'$ the same procedure of blow-ups and blow-downs transforming the compactifying divisor into a configuration containing a cuspidal curve $C$ as in Figure~\ref{figure:sequence-dihedral} or Figure~\ref{figure:sequence-TOI-32},~\ref{figure:sequence-TOI-31}. Then we have a diffeomorphism
\begin{equation*}
\phi \colon Z_2' \to Z_2,
\end{equation*}
where $Z_2$ and $Z_2'$ are the rational 4-manifolds (containing the corresponding cuspidal curves) corresponding to $Z$ and $Z'$, respectively, described in Figures~\ref{figure:sequence-dihedral}, \ref{figure:sequence-TOI-32}, \ref{figure:sequence-TOI-31}. According to Equation~\eqref{equation:W=Z-E_infty=Z_2-E_infty},
\begin{equation*}
W = Z_2 - \nu(\widetilde{E}_{\infty}), \quad W' = Z_2' - \nu(\widetilde{E}_{\infty}).
\end{equation*}

Let $H \subset \mathbb{CP}^2$ be a line and let $A, B \subset \mathbb{CP}^1 \times \mathbb{CP}^1$ be curves of type $(1,0), (0,1)$, respectively. We denote their homology classes again by $[H]$, $[A]$, $[B]$. Let $[e_i] \in H_2(Z_2;\mathbb{Z})$ $(i=1,\dotsc,n$) be the homology classes of the exceptional curves of the sequence of blow-ups from $\mathbb{CP}^2$ or $\mathbb{CP}^1 \times \mathbb{CP}^1$ to $Z_2$. Finally, let $E$ be a $(-1)$-curve in $Z_2'$ of the blowing-ups $Z_2' \to Z_1'$ occurring on the curves $C_i' \subset Z_1'$ ($i=1,\dotsc, k$), whose processes are described also in Figures~\ref{figure:sequence-dihedral}, \ref{figure:sequence-TOI-32}, \ref{figure:sequence-TOI-31}. Note that  $\phi_{\ast}[C_i] = [C_i]$ for all $i$ or $\phi_{\ast}[C_i] = -[C_i]$ for all $i$.

We briefly sketch the proof. In the following we will show that $\phi(E)$ (if $\phi_{\ast}[C_i] = [C_i]$) or $-\phi(E)$ (if $\phi_{\ast}[C_i] = -[C_i]$) is a $(-1)$-curve whose homology class is same with that of one of the exceptional curves of the blowing-ups $Z_2 \to Z_1$. Then, since $\phi_{\ast}[C_i] = [C_i]$ or $\phi_{\ast}[C_i] = -[C_i]$ for all $i=1,\dotsc,k$, the two data of intersections of $(-1)$-curves with $C_1, \dotsc, C_k$ in $Z_2$ and $Z_2'$ should be the same.

But one can check that the intersection data in the reduced list of Bhupal-Ono (Proposition~\ref{proposition:reduced-BO-list}) for tetrahedral, octahedral, and icosahedral singularities are different to each other for a given singularity except some exceptional cases. Hence the assertion of Theorem~\ref{theorem:diffeomorphism-type} follows if we can deal with the exceptional cases.

The exceptional cases are as follows: Two minimal symplectic fillings $W$ and $W'$ have the same intersection data, but $W$ is obtained from $\mathbb{CP}^1 \times \mathbb{CP}^1$ and, on the other hand, $W'$ is obtained from $\mathbb{CP}^2$. By the same method as above, one can prove that there is a diffeomorphism $\phi \colon Z_2' \to Z_2$ such that $\phi(E)$ (if $\phi_{\ast}[C_i] = [C_i]$) or $-\phi(E)$ ($\phi_{\ast}[C_i] = -[C_i]$) is a $(-1)$-curve of $Z_2$ such that its homology class is same with that of one of the exceptional curves of the blowing-ups $Z_2 \to Z_1$. Then there is a isomorphism
\begin{equation*}
\phi_{\ast} \colon (H^2(Z_2';\mathbb{Z}), Q_{Z_2'}) \to (H^2(Z_2;\mathbb{Z}), Q_{Z_2})
\end{equation*}
where $Q_{Z_2'}$ is the intersection form on $H^2(Z_2';\mathbb{Z})$ under the basis containing $[H]$, $[C_i]$'s, and the exceptional curves $[E]$'s occurring from $Z_2' \to Z_1'$ and $Q_{Z_2}$ is the intersection form on $H^2(Z_2;\mathbb{Z})$ under the basis containing $[A]$, $[B]$, and, in addition, $[C_i]$'s and $\phi_{\ast}[E]$'s if $\phi_{\ast}[C_i] = [C_i]$ or $-[C_i]$'s and $-\phi_{\ast}[E]$'s if $\phi_{\ast}[C_i] = -[C_i]$, respectively. Then, at first, blowing down $E$'s and $\phi(E)$'s at the same time and then blowing down $C_i' \subset Z_1$ and $C_i'=\phi(C_i') \subset Z_1$ ($i=1,\dotsc,k$), we get a isomorphism
\begin{equation*}
\phi_{\ast} \colon (H^2(\mathbb{CP}^2 \sharp \overline{\mathbb{CP}}^2;\mathbb{Z}), Q_{\mathbb{CP}^2 \sharp \overline{\mathbb{CP}}^2}) \to (H^2(\mathbb{CP}^1 \times \mathbb{CP}^1;\mathbb{Z}), Q_{\mathbb{CP}^1 \times \mathbb{CP}^1}),
\end{equation*}
that is, we get the equivalence between the intersection matrix of $\mathbb{CP}^2 \sharp \overline{\mathbb{CP}}^2$ and that of $\mathbb{CP}^1 \times \mathbb{CP}^1$, which is a contradiction. Therefore the assertion of Theorem~\ref{theorem:diffeomorphism-type} follows.

From now on, we will show that $\phi(E)$ (if $\phi_{\ast}[C_i] = [C_i]$) or $-\phi(E)$ (if $\phi_{\ast}[C_i] = -[C_i]$) is a $(-1)$-curve whose homology class is same with that of one of the exceptional curves of the blowing-ups $Z_2 \to Z_1$.

\medskip

\textit{At first, we assume that $\phi_{\ast}[C_i] = [C_i]$.}

\medskip

\textit{Case~1: Dihedral singularities}

\medskip

We divide Case~1 into two subcases.

\medskip

\textit{Case~1-1: $Z_2$ is obtained from $\mathbb{CP}^1 \times \mathbb{CP}^1$ by a sequence of blow-ups}

\medskip

Note that $\{[A], [B], [e_1], \dotsc, [e_n]\}$ is a basis for $H_2(Z_2; \mathbb{Z})$. Then
\begin{equation*}
\phi_{\ast}([E])=a[A]+b[B]+\sum c_i [e_i]
\end{equation*}
for some integers $a, b, c_i$ in $H_2(Z_2;\mathbb{Z})$. But $\phi(E) \cdot A=0$. So we have $b=0$. Therefore
\begin{equation*}
\phi_{\ast}([E])=a[A]+\sum c_i [e_i].
\end{equation*}
Since $\phi(E) \cdot \phi(E)=-1$, we have $(a[A]+\sum c_i[e_i]) \cdot (a[A]+\sum c_i[e_i])=-1$. Therefore, it follows that only one $c_{i_0}=\pm 1$ and the other $c_j$'s $(j \neq i_0)$ are all zero. Therefore
\begin{equation*}
\phi_{\ast}([E])=a[A] \pm [e_{i_0}].
\end{equation*}
On the other hand, $\phi(E)$ does not intersect the cuspidal curve $C$ in $Z_2$. Since $[C]=2[A]+2[B]-\sum d_i[e_i]$ for some $d_i=0$ or $1$, we have
\begin{equation*}
0 = \phi(E) \cdot C = \text{$2a$ or $2a \pm 1$}
\end{equation*}
Therefore $a=0$. Then
\begin{equation*}
\phi_{\ast}(E) = \pm [e_{i_0}].
\end{equation*}
Assume that $\phi(E)$ intersects $C_s$ in $Z_2$ (cf.~Figure~\ref{figure:sequence-dihedral}). Since
\begin{equation*}
[C_s] = [e_{\ell}]-[e_{j_1}]-\dotsb-[e_{j_t}]
\end{equation*}
for some $\ell, j_1, \dotsc, j_t$, we have
\begin{equation*}
1 = \phi_{\ast}([E]) \cdot [C_s] = (\pm e_{i_0}) \cdot ([e_{\ell}]-[e_{j_1}]-\dotsb-[e_{j_t}]).
\end{equation*}
Therefore $\phi_{\ast}([E])=[e_{j_s}]$ for some $s \ge 1$ or $\phi_{\ast}([E])=-[e_{\ell}]$. Suppose that $\phi_{\ast}([E])=-[e_{\ell}]$. Then $\phi(E)$ intersect with a curve $D$ whose homology class is given by $D'-e_{\ell}$, which is contradict to the fact that the intersection numbers of $\phi(E)$ with any other curves are zero. Therefore $\phi_{\ast}([E])=[e_{j_s}]$, which implies that $\phi(E)$ is a $(-1)$-curve whose homology class is same with that of one of the $(-1)$-curves coming from the blowing-ups $Z_2 \to Z_1$ on $C_{\ell}$.

Therefore one can conclude that the data of intersections of $(-1)$-curves with $C_i$'s from $Z_2$ and $Z_2'$ are the same if $W$ and $W'$ are diffeomorphic.

\medskip

\textit{Case~1-2: $Z_2$ is obtained from $\mathbb{CP}^2$ by a sequence of blow-ups}

\medskip

Since $\{[H], [e_1], \dotsc, [e_n]\}$ is a basis for $H_2(Z_2; \mathbb{Z})$, we have
\begin{equation*}
\phi_{\ast}([E])=a[H]+\sum c_i [e_i]
\end{equation*}
for some integers $a, c_i$. Let $L$ be the line passing through the cusp (cf.~Figure~\ref{figure:sequence-dihedral}). Let $[e_1]$ be the homology class of the exceptional curve of the blowing up at $q \in L$. Let $\widetilde{L}$ be the proper transform of $L$ in $Z_2$. Then $[\widetilde{L}]=[H]-[e_1]$. Since $\phi(E) \cdot \widetilde{L}=0$, we have
\begin{equation*}
0 = \phi(E) \cdot \widetilde{L} = a+c_1.
\end{equation*}
Therefore
\begin{equation*}
\phi_{\ast}([E]) = a[H] - a[e_1] + \sum_{i \ge 2} c_i [e_i].
\end{equation*}
Furthermore, $\phi(E) \cdot \phi(E) = -1$. Then we have $c_{i_0} = \pm 1$ for some $i_0 \ge 2$, and the other $c_j$'s are all zero. Therefore
\begin{equation*}
\phi_{\ast}([E]) = a[H] - a[e_1] \pm [e_{i_0}].
\end{equation*}
Since $\phi(E)$ does not intersect the cuspidal curve $C$ and $[C]=3[H]-[e_1]-\sum_{i \ge 2} d_i [e_i]$ for some $d_i=0$ or $1$, we have
\begin{equation*}
0 = C \cdot \phi(E) = \text{$2a$ or $2a \pm 1$}.
\end{equation*}
Therefore $a=0$; hence, we have
\begin{equation*}
\phi_{\ast}([E]) = \pm [e_{i_0}].
\end{equation*}
By a similar argument as Case~I-1, one can conclude that $\phi_{\ast}([E]) = [e_{i_0}]$ and it is a $(-1)$-curve whose homology class is same with that of one of the $(-1)$-curves originated from the blowing-ups $Z_2 \to Z_1$. Therefore the two data on intersections of $(-1)$-curves with $C_i$'s in $Z_2$ and $Z_2'$ are the same.

\medskip

\textit{Case~2: Tetrahedral, Octahedral, Icosahedral singularities of type (3,2)}

\medskip

Let $C$ be the cuspidal curve in $Z_2$. From the list of Bhupal--Ono~\cite{Bhupal-Ono-2012}, we have $C \cdot C=3$, $4$, or $5$. If $C \cdot C= 3$ or $4$, then one can easily check that the models of $Z_2$ in their list have the different second Betti numbers. Therefore one can distinguish minimal symplectic fillings via the second Betti numbers. So one may assume that
\begin{equation*}
C \cdot C = 5.
\end{equation*}

We divide Case~2 into two subcases.

\medskip

\textit{Case~2-1: $Z_2$ is obtained from $\mathbb{CP}^2$ by a sequence of blow-ups}

\medskip

Since $\{[H], [e_1], \dotsc, [e_n]\}$ is a basis for $H_2(Z;\mathbb{Z})$, we have
\begin{equation*}
\phi_{\ast}([E]) = a[H] + \sum_{i=1}^{n} b_i [e_i]
\end{equation*}
for some $a, b_i \in \mathbb{Z}$. On the other hand, let $C$ be the cuspidal curve in $Z_2$. Since $[C]=3[H]-\sum d_i [e_i]$ for some $d_i=0$ or $1$, and since $[C_i] = [e_{i_0}]-[e_{j_1}]-\dotsb-[e_{j_s}]$ ($i=1,\dotsc,k$) and $\{[C_1],\dotsc,[C_k]\} \subset H_2(Z;\mathbb{Z})$ are linearly independent over $\mathbb{Z}$, one may assume that
\begin{equation*}
\{[C], [C_1], \dotsc, [C_k], [e_1], \dotsc, [e_m]\}
\end{equation*}
is also a basis for $H_2(Z;\mathbb{Q})$ for some $m < n$, where $e_1, \dotsc, e_m$ are the $(-1)$-curves obtained from the blowing-ups $Z_2 \to Z_1$. Then we have
\begin{equation*}
\phi_{\ast}([E]) = \alpha [C] + \sum_{i=1}^{k} \beta_i [C_i] + \sum_{i=1}^{m} \gamma_i [e_i]
\end{equation*}
for some $\alpha, \beta_i, \gamma_i \in \mathbb{Q}$. Then, since $3\alpha = a \in \mathbb{Z}$, we have
\begin{equation*}
\alpha = \frac{a}{3}
\end{equation*}
for some $a \in \mathbb{Z}$. Then, since $\phi(E) \cdot C = 0$, we have
\begin{equation*}
0 = \phi(E) \cdot C = 5 \alpha + \beta_1
\end{equation*}
Therefore we have
\begin{equation*}
\phi_{\ast}([E]) = \alpha [C] - 5\alpha [C_1] + \sum_{i=2}^{k} \beta_i [C_i] + \sum_{i=1}^{m} \gamma_i [e_i].
\end{equation*}
On the other hand,
\begin{equation*}
-1 = \phi(E) \cdot \phi(E) = -5 \alpha^2 + \Delta
\end{equation*}
where
\begin{multline*}
\Delta = \\
\left(- 5\alpha [C_1] + \sum_{i=2}^{k} \beta_i [C_i] + \sum_{i=1}^{m} \gamma_i [e_i]\right) \cdot \left(- 5\alpha [C_1] + \sum_{i=2}^{k} \beta_i [C_i] + \sum_{i=1}^{m} \gamma_i [e_i]\right).
\end{multline*}
Since $[C_i] = [e_{i_0}]-[e_{j_1}]-\dotsb-[e_{j_s}]$ ($i=1,\dotsc,k$), we have $\Delta \le 0$. Therefore we have $5 \alpha^2 - 1 \le 0$. But $\alpha=\frac{a}{3}$ for $a \in \mathbb{Z}$. Therefore $\alpha = 0$ or $\alpha = \pm \frac{1}{3}$. If $\alpha=0$, then one can show that $\phi_{\ast}([E]) = [e_{i_0}]$ for some $i_0 \le m$ by a similar method used in the above cases. Therefore $\phi(E)$ is a $(-1)$-curve coming from the blowing-ups $Z_2 \to Z_1$.

So it remains to prove that $\alpha$ cannot be $\pm \frac{1}{3}$. Suppose on the contrary that $\alpha=\pm \frac{1}{3}$. Then
\begin{equation*}
\phi_{\ast}([E]) = \pm [H] + \sum_{i=1}^{n} b_i [e_i].
\end{equation*}
Since $\phi(E) \cdot \phi(E) = -1$; so one may assume that
\begin{equation*}
\phi_{\ast}([E]) = \pm [H] \pm [e_1] \pm [e_2].
\end{equation*}
But, we then have $\phi(E) \cdot C \neq 0$, which is a contradiction.

\medskip

\textit{Case~2-2: $Z_2$ is obtained from $\mathbb{CP}^1 \times \mathbb{CP}^1$ by a sequence of blow-ups}

\medskip

Since $\{[A], [B], [e_1], \dotsc, [e_n]\}$ is a basis for $H_2(Z;\mathbb{Z})$, we have
\begin{equation*}
\phi_{\ast}([E]) = a[A] + b[B] + \sum_{i=1}^{n} c_i [e_i]
\end{equation*}
for some $a, b, c_i \in \mathbb{Z}$. On the other hand, if $C$ is the cuspidal curve in $Z_2$, we have $[C]=2[A]+2[B]-\sum_{i=1}^{n} d_i [e_i]$ for some $d_i=0$ or $1$. As in Case~2-1, one may assume that
\begin{equation*}
\{[A], [C], [C_1], \dotsc, [C_k], [e_1], \dotsc, [e_m]\}
\end{equation*}
is a basis for $H_2(Z;\mathbb{Q})$ for some $m < n$, where $e_1, \dotsc, e_m$ are the $(-1)$-curves which are the exceptional curves of the blowing-ups $Z_2 \to Z_1$. Then
\begin{equation*}
\phi_{\ast}([E]) = \alpha [A] + \beta [C] + \sum_{i=1}^{k} \gamma_i [C_i] + \sum_{i=1}^{m} \delta_i [e_i]
\end{equation*}
for some $\alpha, \beta, \gamma_i, \delta_i \in \mathbb{Q}$. Note that $\beta=\frac{b}{2}$ for $b \in \mathbb{Z}$.

Since $\phi(E) \cdot C = 0$, we have
\begin{equation*}
2\alpha + 5 \beta + \gamma_1 = 0.
\end{equation*}
Therefore
\begin{equation*}
\phi_{\ast}([E]) = \alpha [A] + \beta [C] + (-2\alpha -5\beta) [C_1] + \sum_{i=2}^{k} \gamma_i [C_i] + \sum_{i=1}^{m} \delta_i [e_i].
\end{equation*}
On the other hand, from the fact that$\phi(E) \cdot \phi(E) = -1$, we have
\begin{equation*}
-5\beta^2 + \Delta = -1
\end{equation*}
where
\begin{equation*}
\Delta = \left( (-2\alpha -5\beta) [C_1] + \sum_{i=2}^{k} \gamma_i [C_i] + \sum_{i=1}^{m} \delta_i [e_i] \right)^2.
\end{equation*}
As in Case~2-1, one can conclude that $\Delta \le 0$. Therefore we have
\begin{equation*}
5 \beta^2 - 1 \le 0
\end{equation*}
which is possible only when $\beta = 0$ because $\beta = \frac{b}{2}$ for $b \in \mathbb{Z}$. Therefore
\begin{equation*}
\phi_{\ast}([E]) = \alpha [A] + \sum_{i=1}^{k} \gamma_i [C_i] + \sum_{i=1}^{m} \delta_i [e_i].
\end{equation*}
But $[C_i] = [e_{i_0}]-[e_{j_1}]-\dotsb-[e_{j_s}]$ ($i=1,\dotsc,k$). Therefore one may assume that
\begin{equation*}
\phi_{\ast}([E]) = a[A] + \sum_{i=1}^{n} c_i [e_i]
\end{equation*}
for $a, c_i \in \mathbb{Z}$. Then, since $\phi(E) \cdot \phi(E) = -1$, we have $\phi_{\ast}([E]) = a[A] \pm [e_i]$ for some $i$. But $0=\phi(E) \cdot C = 2a$ or $2a \pm 1$. So $a=0$. Therefore $\phi_{\ast}([E]) = \pm [e_i]$. As before one can conclude that $\phi_{\ast}([E]) = [e_i]$ for $i \le m$. Therefore $\phi(E)$ is a $(-1)$-curve whose homology class is same with that of one of the exceptional curves of the blowing-ups $Z_2 \to Z_1$.

\medskip

\textit{Case~3: Tetrahedral, Octahedral, Icosahedral singularities of type (3,1)}

\medskip

Tetrahedral, Octahedral, Icosahedral singularities of type (3,1) are divided into two subclasses as given in Figure~\ref{figure:sequence-TOI-31} according to the existence of a $(-1)$-curve connecting the cuspidal curve of $Z_2$ and an arm $C_i$.

\medskip

\textit{Case~3-1: Case~(A) in Figure~\ref{figure:sequence-TOI-31}}

\medskip

Since $\{[A], [B], [e_1], \dotsc, [e_n]\}$ is a basis for $H_2(Z_2; \mathbb{Z})$, we have
\begin{equation*}
\phi_{\ast}([E]) = a[A] + b[B] + \sum_{i=1}^{n} c_i [e_i]
\end{equation*}
for some $a, b, c_i \in \mathbb{Z}$. Since $\phi(E) \cdot A = 0$, we have $b=0$. For $\phi(E) \cdot \phi(E) = -1$, we have $\phi_{\ast}(E) = a[A] \pm [e_{i_0}]$ for some $i_0=1,\dotsc,n$. Finally, from the condition $\phi(E) \cdot C = 0$, we have $a=0$. Therefore $\phi_{\ast}(E)=\pm [e_{i_0}]$. Then, by a similar arguments as before, one can conclude that $\phi_{\ast}([E]) = [e_{i_0}]$; hence $\phi(E)$ is a $(-1)$-curve such that its homology class is same with that of one of the exceptional curves of the blowing-ups $Z_2 \to Z_1$.

\medskip

\textit{Case~3-2 : Case~(B) in Figure~\ref{figure:sequence-TOI-31}}

\medskip

In the list of Bhupal--Ono~\cite{Bhupal-Ono-2012}, minimal symplectic fillings are classified by the data $(m; C \cdot C, -c_1, \dotsc, -c_k; i, j; a_1 \times i_1, \dotsc, a_k \times i_k)$ where the numbers $i$ and $j$ denote the existence of $(-1)$-curves intersecting $B$ and $C_i$, and $D$ and $C_j$, respectively. One can check that if $Z_2$ and $Z_2'$ have the same data $(a_1 \times i_1, \dotsc, a_k \times i_k)$, then their corresponding $(i,j)$'s also coincide. Therefore it is enough to show that $Z_2$ and $Z_2'$ have the same $(a_1 \times i_1, \dotsc, a_k \times i_k)$.

As before, we have $\phi_{\ast}([E]) = a[A] + b[B] + \sum_{i=1}^{n} c_i [e_i]$. Since $\phi(E) \cdot A = 0$, we have $b=0$. Furthermore $\phi(E) \cdot \phi(E) = -1$. Therefore we have $\phi_{\ast}([E]) = a[A] \pm [e_{i_0}]$ for some $i_0$. On the other hand, $\phi(E) \cdot C = 0$. Therefore $\phi_{\ast}([E]) = \pm [e_{i_0}]$. Since $\phi_{\ast}([B-e_1]) \cdot \phi_{\ast}([E]) = 0$, it follows that $[e_{i_0}] \neq [e_1]$. Finally, using the same argument as above, one can conclude that $\phi_{\ast}([E]) = [e_{i_0}]$. Therefore $\phi(E)$ is a $(-1)$-curve with the same homology class with one of the exceptional curves of the blowing-ups $Z_2 \to Z_1$.

It remains to distinguish minimal symplectic fillings of Case~(A) and Case~(B) in Figure~\ref{figure:sequence-TOI-31}. Assume that $W$ is obtained by a sequence of blow-downs and blow-ups described in Case~(A) of Figure~\ref{figure:sequence-TOI-31} and $W'$ is obtained by that of Case~(B) of Figure~\ref{figure:sequence-TOI-31}. Let $E_1 \subset Z_2'$ be the $(-1)$-curve obtained by blowing-up the intersection of $B$ and $A'$ and let $E_N$ be the $(-1)$-curve connecting the cuspidal curve $C \subset Z_2'$ and a curve $C_k$.

Suppose that there is a diffeomorphism $W' \to W$; so we get a diffeomorphism $\phi \colon Z_2' \to Z_2$ as before. We have
\begin{equation*}
\phi_{\ast}([E_N]) = a[A] + b[B] + \sum c_i [e_i]
\end{equation*}
for some $c_i \in \mathbb{Z}$. Since $\phi(E_N) \cdot A = 0$, we have $b=0$. On the other hand, $\phi(E_N) \cdot \phi(E_N) = -1$. Therefore we have
\begin{equation*}
\phi_{\ast}([E_N]) = a[A] \pm [e_{i_0}]
\end{equation*}
for some $i_0$. Note that the homology class of $C_k \subset Z_2$ is given as
\begin{equation}\label{equation:C_k}
[C_k] = [e_{i_1}] - [e_{j_1}] - \dotsb - [e_{j_r}]
\end{equation}
with $e_{i_1} \neq e_1$ and $e_{j_s} \neq e_1$ for any $j_s=j_1, \dotsc, j_r$. Since $\phi(E_N) \cdot C_k = 1$, $e_{i_0} \neq e_1$. On the other hand, $\phi(E_N) \cdot (B-e_1) = 0$. So $a=0$. Therefore we have
\begin{equation*}
\phi_{\ast}([E_N]) = \pm [e_{i_0}]
\end{equation*}
for some $e_{i_0} \neq e_1$.

Let $[C] = 2[A] + 2[B] - \sum_{i=1}^{n} d_i [e_i]$ for some $d_i=0$ or $1$. Since $\phi(E_N) \cdot C = 1$, we must have $d_{i_0} = 1$. Therefore
\begin{equation*}
[C] = 2[A] + 2[B] - [e_{i_0}] - \sum_{i \neq i_0} d_i [e_i].
\end{equation*}
On the other hand, $\phi(E_N) \cdot C_k = 1$. Therefore, $e_{j_s} = e_{i_0}$ for some $j_s$ in Equation~\eqref{equation:C_k}. One can conclude that $e_{i_0}$ is a $(-1)$-curve coming from the blowing-ups $Z_2 \to Z_1$, or there is a $C_i$ ($i \neq k$) such that its homology class is given by $[C_i] = [e_{i_0}] - [e_{k_1}] - \dotsb - [e_{k_p}]$. But $\phi(E_N)$ does not intersect any $C_i$ for $i \neq k$. Therefore $e_{i_0}$ must be  a $(-1)$-curve coming from the blowing-ups $Z_2 \to Z_1$. However any $(-1)$-curve coming from the blowing-ups $Z_2 \to Z_1$ cannot intersect $C$, which is a contradiction.

\medskip

\textit{Next, we assume that $\phi[C_i] = -[C_i]$.}

\medskip

Let $\epsilon = -\phi(E) \subset Z_2$. We can show  with the same method that $\epsilon$ is a $(-1)$-curve whose homology class is same with that of one of the exceptional curves of the blowing-ups $Z_2 \to Z_1$. Therefore the data of the intersections of $(-1)$-curves with $C_i$ and that of $(-1)$-curves with $C_i$ are the same. But one can check that the data in the list of Bhupal--Ono~\cite{Bhupal-Ono-2012} are all different to each other for a given quotient surface singularity. Hence the assertion follows.
\end{proof}

\begin{remark}\label{remark:the-same--1-data}
In the above proof we show that two symplectic fillings $W$ and $W'$ are diffeomorphic if and only if the two data of intersections of $(-1)$-curves with $C_1, \dotsc, C_k$ in $Z_2$ and $Z_2'$ should be the same under the assumption that $Z$ and $Z'$ are obtained at the same time from $\mathbb{CP}^2$ or $\mathbb{CP}^1 \times \mathbb{CP}^1$. But it is not difficult to show that $W$ and $W'$ are diffeomorphic if and only if the two data of intersections of $(-1)$-curves with $E_{\infty}$ in $Z$ and $Z'$ should be the same under the same assumption. By the way HJS~\cite{HJS-2018} finds all $P$-resolutions of each quotient surface singularities and builds up the data of $(-1)$-curves intersecting $E_{\infty}$ via the technique developed in Section~\ref{section:identifying-Milnor-fibers}. We prove that every minimal symplectic filling of a quotient surface singularity is diffeomorphic to one of its Milnor fibers in Theorem~\ref{theorem:fillings-diffeomorphic-to-Milnors}. Therefore one may prove the above Theorem~\ref{theorem:diffeomorphism-type} using the list in HJS~\cite{HJS-2018}.
\end{remark}

The main parts of the above proof of Theorem~\ref{theorem:diffeomorphism-type} consist of calculating the image of the homology classes of $(-1)$-curves $E$ in $Z_2'$ under a given diffeomorphism $\phi \colon Z_2' \to Z_2$. One can easily show that the same arguments also hold in case a homeomorphism $\phi \colon Z_2' \to Z_2$ is given. Therefore:

\begin{corollary}\label{corollary:no-exotic-fillings}
For any quotient surface singularity, there are no exotic symplectic fillings of its link.
\end{corollary}

\section{Partial resolutions and versal deformation spaces}
\label{section:$P$-resolution}

The diffeomorphism types of Milnor fibers of a quotient surface singularity $(X,0)$ are invariants of the irreducible components of the reduced versal base space $\Def(X)$ of deformations of $X$. In this section, we recall that there is a one-to-one correspondence between the components of $\Def(X)$ and certain partial resolutions of $(X,0)$. We refer to KSB~\cite{Kollar-Shepherd-Barron-1988}, Stevens~\cite{Stevens-1991, Stevens-1993}, Behnke--Christophersen~\cite{Behnke-Christophersen-1994} for details.

\begin{definition}
A smoothing $\mathcal{X} \to \Delta$ of a quotient surface singularity $(X,0)$ is \emph{$\mathbb{Q}$-Gorenstein} if $K_{\mathcal{X}}$ is $\mathbb{Q}$-Cartier.
\end{definition}

\begin{definition}
A normal surface singularity is \emph{of class $T$} if it is a quotient surface singularity, and it admits a $\mathbb{Q}$-Gorenstein one-parameter smoothing.
\end{definition}

It is known that a singularity of class $T$ is a rational double point or a cyclic quotient surface singularity of type $\frac{1}{dn^2}(1, dna-1)$ with $d \ge 1$, $n \ge 2$, $1 \le a < n$, and $(n,a)=1$. Due essentially to Wahl~\cite{Wahl-1981}, a cyclic quotient surface singularity of class $T$ may be recognized from its minimal resolution as follows:

\begin{proposition}\hfill
\label{proposition:T-algorithm}
\begin{enumerate}[(i)]
\item The singularities
\begin{tikzpicture}
\node[bullet] (10) at (1,0) [label=above:{$-4$}] {};
\end{tikzpicture}
and
\begin{tikzpicture}
\node[bullet] (10) at (1,0) [label=above:{$-3$}] {};
\node[bullet] (20) at (2,0) [label=above:{$-2$}] {};

\node[empty] (250) at (2.5,0) [] {};
\node[empty] (30) at (3,0) [] {};

\node[bullet] (350) at (3.5,0) [label=above:{$-2$}] {};
\node[bullet] (450) at (4.5,0) [label=above:{$-3$}] {};

\draw [-] (10)--(20);
\draw [-] (20)--(250);
\draw [dotted] (20)--(350);
\draw [-] (30)--(350);
\draw [-] (350)--(450);
\end{tikzpicture}
are of class $T$

\item If the singularity
\begin{tikzpicture}
\node[bullet] (20) at (2,0) [label=above:{$-b_1$}] {};

\node[empty] (250) at (2.5,0) [] {};
\node[empty] (30) at (3,0) [] {};

\node[bullet] (350) at (3.5,0) [label=above:{$-b_r$}] {};

\draw [-] (20)--(250);
\draw [dotted] (20)--(350);
\draw [-] (30)--(350);
\end{tikzpicture}
is of class $T$, then so are
\begin{equation*}
\begin{tikzpicture}
\node[bullet] (10) at (1,0) [label=above:{$-2$}] {};
\node[bullet] (20) at (2,0) [label=above:{$-b_1$}] {};

\node[empty] (250) at (2.5,0) [] {};
\node[empty] (30) at (3,0) [] {};

\node[bullet] (350) at (3.5,0) [label=above:{$-b_{r-1}$}] {};
\node[bullet] (450) at (4.5,0) [label=above:{$-b_r-1$}] {};

\draw [-] (10)--(20);
\draw [-] (20)--(250);
\draw [dotted] (20)--(350);
\draw [-] (30)--(350);
\draw [-] (350)--(450);
\end{tikzpicture}
\text{~and~}
\begin{tikzpicture}
\node[bullet] (10) at (1,0) [label=above:{$-b_1-1$}] {};
\node[bullet] (20) at (2,0) [label=above:{$-b_2$}] {};

\node[empty] (250) at (2.5,0) [] {};
\node[empty] (30) at (3,0) [] {};

\node[bullet] (350) at (3.5,0) [label=above:{$-b_r$}] {};
\node[bullet] (450) at (4.5,0) [label=above:{$-2$}] {};

\draw [-] (10)--(20);
\draw [-] (20)--(250);
\draw [dotted] (20)--(350);
\draw [-] (30)--(350);
\draw [-] (350)--(450);
\end{tikzpicture}
\end{equation*}

\item Every singularity of class $T$ that is not a rational double point can be obtained by starting with one of the singularities described in (i) and iterating the steps described in (ii).
\end{enumerate}
\end{proposition}

The one-parameter $\mathbb{Q}$-Gorenstein smoothing of a singularity of class $T$ is interpreted topologically as a rational blowdown surgery defined by Fintushel--Stern~\cite{Fintushel-Stern-1997}, and extended by J.~Park~\cite{JPark-1977}.

\subsection{$P$-resolutions}

Koll{\'a}r--Shepherd-Barron~\cite{Kollar-Shepherd-Barron-1988} gave a one-to-one correspondence between the (reduced irreducible) components of the versal deformation space $\Def(X)$ of $X$ and the $P$-resolutions of $X$, which are defined as follows:

\begin{definition}
\label{definition:$P$-resolution}
A \emph{$P$-resolution} of a quotient surface singularity $X$ is a partial resolution $f \colon Y \to X$ such that $Y$ has only singularities of class $T$, and $K_Y$ is ample relative to $f$.
\end{definition}

KSB~\cite{Kollar-Shepherd-Barron-1988} also provided an algorithm for finding all $P$-resolutions of a given quotient surface singularity.

\begin{definition}[{KSB~\cite[Definition~3.12]{Kollar-Shepherd-Barron-1988}}]
\label{definition:maximal-resolution}
Let $X$ be a quotient surface singularity. A resolution $f \colon V \to X$ is \emph{maximal} if $K_V \sim f^{\ast}{K_X} - \sum{a_i E_i}$, where $0 < a_i < 1$, and for any proper birational morphism $g \colon U \to V$ that is not an isomorphism, we have $K_U \sim h^{\ast}{K_X} - \sum{b_j F_j}$, where $h=f \circ g$ and some $b_j \le 0$.
\end{definition}

\begin{proposition}[{KSB~\cite[Lemma~3.13]{Kollar-Shepherd-Barron-1988}}]
\label{proposition:maximal-resolution-uniqueness}
A quotient surface singularity $(X,0)$ has a unique maximal resolution.
\end{proposition}

\begin{proof}
We briefly recall the proof of KSB~\cite[Lemma~3.13]{Kollar-Shepherd-Barron-1988} because it provides an algorithm for finding the maximal resolution. Let $\pi \colon U \to X$ be the minimal resolution. We can write
\begin{equation*}
K_U \sim \pi^{\ast}{K_X} + \sum_{j}{(-1+\alpha_j) E_j}
\end{equation*}
where $E_j$ are the exceptional divisors. We successively blow up any point $Q = E_i \cap E_j$ with $\alpha_i + \alpha_j < 1$ until the quantities $\alpha_i$ satisfy $\alpha_i < 1$ for all $i$, but if $E_i \cap E_j \neq \varnothing$, then $\alpha_i + \alpha_j \ge 1$. This procedure stops after finitely many times because if $E_k$ is the exceptional divisor of the blow-up along $Q=E_i \cap E_j$, then $\alpha_k = \alpha_i+\alpha_j$.
\end{proof}

\begin{proposition}[{KSB~\cite[Lemma~3.14]{Kollar-Shepherd-Barron-1988}}]
\label{proposition:maximal-resolution-dominating}
Let $(X,0)$ be a quotient surface singularity and let $f \colon Z \to X$ be a partial resolution such that $Z$ has only quotient surface singularities and $K_Z$ is ample relative to $f$. Then $Z$ is dominated by the maximal resolution $X_m$ of $X$. In particular, every $P$-resolution of $(X,0)$ is dominated by $X_m$.
\end{proposition}

\begin{example}[Continued from Example~\ref{example:cyclic-19/7}]
\label{example:cyclic-19/7-$P$-resolution}
Let $(X,0)$ be a cyclic quotient surface singularity of type $\frac{1}{19}(1,7)$. The minimal resolution is
\begin{equation*}
\begin{tikzpicture}
\node[bullet] (-10) at (-1,0) [label=below:{$-3$},label=above:{$8/19$}] {};

\node[bullet] (00) at (0,0) [label=below:{$-4$},label=above:{$5/19$}] {};

\node[bullet] (10) at (1,0) [label=below:{$-2$},label=above:{$12/19$}] {};

\draw [-] (-10)--(00)--(10);
\end{tikzpicture}
\end{equation*}
where the positive numbers are the $\alpha_i$ in the proof of Proposition~\ref{proposition:maximal-resolution-uniqueness}. So its maximal resolution is
\begin{equation*}
\begin{tikzpicture}
\node[bullet] (-30) at (-3,0) [label=below:{$-4$},label=above:{$8/19$}] {};

\node[bullet] (-20) at (-2,0) [label=below:{$-2$},label=above:{$13/19$}] {};

\node[bullet] (-10) at (-1,0) [label=below:{$-1$},label=above:{$18/19$}] {};

\node[bullet] (00) at (0,0) [label=below:{$-7$},label=above:{$5/19$}] {};

\node[bullet] (10) at (1,0) [label=below:{$-1$},label=above:{$17/19$}] {};

\node[bullet] (20) at (2,0) [label=below:{$-3$},label=above:{$12/19$}] {};

\draw [-] (-30)--(-20)--(-10)--(00)--(10)--(20);
\end{tikzpicture}
\end{equation*}
Then it has three $P$-resolutions:
\begin{equation*}
\begin{tikzpicture}
\node[bullet] (-10) at (-1,0) [label=below:{$-3$},label=above:{}] {};
\node[] (-105) at (-1,0.5) [] {$Y_1$};

\node[bullet] (00) at (0,0) [label=below:{$-4$}] {};

\node[rectangle] (10) at (1,0) [label=below:{$-2$}] {};

\draw [-] (-10)--(00)--(10);
\end{tikzpicture}
\quad
\begin{tikzpicture}
\node[bullet] (-10) at (-1,0) [label=below:{$-3$}] {};
\node[] (-105) at (-1,0.5) [] {$Y_2$};

\node[rectangle] (00) at (0,0) [label=below:{$-4$}] {};

\node[bullet] (10) at (1,0) [label=below:{$-2$}] {};

\draw [-] (-10)--(00)--(10);
\end{tikzpicture}
\quad
\begin{tikzpicture}
\node[rectangle] (-20) at (-2,0) [label=below:{$-4$}] {};
\node[] (-205) at (-2,0.5) [] {$Y_3$};

\node[bullet] (-10) at (-1,0) [label=below:{$-1$}] {};

\node[rectangle] (00) at (0,0) [label=below:{$-5$}] {};

\node[rectangle] (10) at (1,0) [label=below:{$-2$}] {};

\draw [-] (-20)--(-10)--(00)--(10);
\end{tikzpicture}
\end{equation*}
Here a linear chain of vertices decorated by a rectangle $\square$ denotes curves on the minimal resolution of a $P$-resolution which are contracted to a singularity of class $T$ on the $P$-resolution.
\end{example}

\begin{example}\label{example:I_30(7-3)+7}
Let $(X,0)$ be an icosahedral singularity $I_{30(7-3)+7}$ whose dual graph of the minimal resolution is given by
\begin{equation*}
\begin{tikzpicture}
\node[bullet] (-20) at (-2,0) [label=below:{$-2$}] {};
\node[bullet] (-10) at (-1,0) [label=below:{$-2$}] {};

\node[bullet] (00) at (0,0) [label=below:{$-6$}] {};

\node[bullet] (01) at (0,1) [label=left:{$-2$}] {};

\node[bullet] (10) at (1,0) [label=below:{$-2$}] {};
\node[bullet] (20) at (2,0) [label=below:{$-3$}] {};

\draw [-] (00)--(-10)--(-20);
\draw [-] (00)--(01);
\draw [-] (00)--(10)--(20);
\end{tikzpicture}
\end{equation*}
Then there are four $P$-resolutions:
\begin{equation*}
\begin{tikzpicture}
\node[rectangle] (-20) at (-2,0) [label=below:{$-2$}] {};
\node[rectangle] (-10) at (-1,0) [label=below:{$-2$}] {};

\node[bullet] (00) at (0,0) [label=below:{$-6$}] {};

\node[rectangle] (01) at (0,1) [label=left:{$-2$}] {};

\node[rectangle] (10) at (1,0) [label=below:{$-2$}] {};
\node[bullet] (20) at (2,0) [label=below:{$-3$}] {};

\draw [-] (00)--(-10)--(-20);
\draw [-] (00)--(01);
\draw [-] (00)--(10)--(20);
\end{tikzpicture}
\qquad
\begin{tikzpicture}
\node[bullet] (-20) at (-2,0) [label=below:{$-2$}] {};
\node[rectangle] (-10) at (-1,0) [label=below:{$-2$}] {};

\node[rectangle] (00) at (0,0) [label=below:{$-6$}] {};

\node[bullet] (01) at (0,1) [label=left:{$-2$}] {};

\node[rectangle] (10) at (1,0) [label=below:{$-2$}] {};
\node[rectangle] (20) at (2,0) [label=below:{$-3$}] {};

\draw [-] (00)--(-10)--(-20);
\draw [-] (00)--(01);
\draw [-] (00)--(10)--(20);
\end{tikzpicture}
\end{equation*}

\begin{equation*}
\begin{tikzpicture}
\node[rectangle] (-20) at (-2,0) [label=below:{$-2$}] {};
\node[rectangle] (-10) at (-1,0) [label=below:{$-2$}] {};

\node[rectangle] (00) at (0,0) [label=below:{$-6$}] {};

\node[bullet] (01) at (0,1) [label=left:{$-2$}] {};

\node[bullet] (10) at (1,0) [label=below:{$-2$}] {};
\node[bullet] (20) at (2,0) [label=below:{$-3$}] {};

\draw [-] (00)--(-10)--(-20);
\draw [-] (00)--(01);
\draw [-] (00)--(10)--(20);
\end{tikzpicture}
\qquad
\begin{tikzpicture}
\node[rectangle] (-20) at (-2,0) [label=below:{$-2$}] {};
\node[bullet] (-10) at (-1,0) [label=below:{$-2$}] {};

\node[rectangle] (00) at (0,0) [label=below:{$-6$}] {};

\node[rectangle] (01) at (0,1) [label=left:{$-2$}] {};

\node[rectangle] (10) at (1,0) [label=below:{$-2$}] {};
\node[rectangle] (20) at (2,0) [label=below:{$-3$}] {};

\draw [-] (00)--(-10)--(-20);
\draw [-] (00)--(01);
\draw [-] (00)--(10)--(20);
\end{tikzpicture}
\end{equation*}
\end{example}

\begin{example}\label{example:I_30(2-2)+29}
Let $(X,0)$ be an icosahedral singularity $I_{30(2-2)+29}$ whose dual graph of the minimal resolution is given by
\begin{equation*}
\begin{tikzpicture}
\node[bullet] (-10) at (-1,0) [label=below:{$-3$}] {};

\node[bullet] (00) at (0,0) [label=below:{$-2$}] {};

\node[bullet] (01) at (0,1) [label=left:{$-2$}] {};

\node[bullet] (10) at (1,0) [label=below:{$-5$}] {};

\draw [-] (00)--(-10);
\draw [-] (00)--(01);
\draw [-] (00)--(10);
\end{tikzpicture}
\end{equation*}
Then there are three $P$-resolutions $Y_1$, $Y_2$, $Y_3$:
\begin{equation*}
\begin{tikzpicture}
\node at (-1,1) {$Y_1$};

\node[bullet] (-10) at (-1,0) [label=below:{$-3$}] {};

\node[rectangle] (00) at (0,0) [label=below:{$-2$}] {};

\node[rectangle] (01) at (0,1) [label=left:{$-2$}] {};

\node[bullet] (10) at (1,0) [label=below:{$-5$}] {};

\draw [-] (00)--(-10);
\draw [-] (00)--(01);
\draw [-] (00)--(10);
\end{tikzpicture}
\quad
\begin{tikzpicture}
\node at (-1,1) {$Y_2$};

\node[bullet] (-10) at (-1,0) [label=below:{$-3$}] {};

\node[rectangle] (00) at (0,0) [label=below:{$-2$}] {};

\node[bullet] (01) at (0,1) [label=left:{$-2$}] {};

\node[rectangle] (10) at (1,0) [label=below:{$-5$}] {};

\draw [-] (00)--(-10);
\draw [-] (00)--(01);
\draw [-] (00)--(10);
\end{tikzpicture}
\quad
\begin{tikzpicture}
\node at (-1,1) {$Y_3$};

\node[rectangle] (-10) at (-1,0) [label=below:{$-3$}] {};

\node[rectangle] (00) at (0,0) [label=below:{$-5$}] {};

\node[rectangle] (01) at (0,1) [label=left:{$-2$}] {};

\node[bullet] (10) at (1,0) [label=below:{$-1$}] {};

\node[rectangle] (20) at (2,0) [label=below:{$-2$}] {};

\node[rectangle] (30) at (3,0) [label=below:{$-2$}] {};

\node[rectangle] (40) at (4,0) [label=below:{$-6$}] {};

\draw [-] (00)--(-10);
\draw [-] (00)--(01);
\draw [-] (00)--(10)--(20)--(30)--(40);
\end{tikzpicture}
\end{equation*}
\end{example}

\begin{remark}\label{remark:erratum-Stevens}
Stevens~\cite{Stevens-1991, Stevens-1993} provided an algorithm for finding all $P$-resolutions of a given non-cyclic quotient surface singularity. But the numbers of $P$-resolutions of the icosahedral singularities $I_{30(7-3)+7}$ and $I_{30(6-6)+29}$ above in Stevens~\cite{Stevens-1993} are erroneously claimed that to be three and two, respectively. Here we found one more $P$-resolution for each of the singularities, which was confirmed by Stevens~\cite{Stevens-2015}.
\end{remark}

Let $f \colon Y \to X$ be a $P$-resolution. There is an induced map $F \colon \Def(Y) \to \Def(X)$ of deformation spaces by Wahl~\cite{Wahl-1976}, which we refer to as \emph{blowing-down deformations}. On the other hand, there is an irreducible subspace $\DefQG(Y) \subset \Def(Y)$ that corresponds to the $\mathbb{Q}$-Gorenstein deformations of singularities of class $T$ in $Y$.

\begin{proposition}[{KSB~\cite[Theorem~3.9]{Kollar-Shepherd-Barron-1988}}]
Let $X$ be a quotient surface singularity. Then

\begin{enumerate}
\item If $f \colon Y \to X$ is a $P$-resolution, then $F(\DefQG(Y))$ is an irreducible component of $\Def(X)$.

\item If $f_1 \colon Y_1 \to X$ and $f_2 \colon Y_2 \to X$ are two $P$-resolutions of $X$ that are not isomorphic over $X$, and if $F_1$ and $F_2$ are the corresponding maps of deformation spaces, then $F_1(\DefQG(Y_1)) \neq F_2(\DefQG(Y_2))$.

\item Every component of $\Def(X)$ arises in this way.
\end{enumerate}
\end{proposition}

Since Milnor fibers are invariants of irreducible components of $\Def(X)$, there is a one-to-one correspondence between Milnor fibers and $P$-resolutions of $(X,0)$.

\subsection{$M$-resolutions}

One may establish another one-to-one correspondence between the components of $\Def(X)$ and certain partial resolutions of $X$, the so-called \emph{$M$-resolutions}. See Behnke--Christophersen~\cite{Behnke-Christophersen-1994} for details on $M$-resolutions.

\begin{definition}
A \emph{Wahl singularity} is a cyclic quotient surface singularity of class $T$ that admits a smoothing whose Milnor fiber $M$ is a rational homology disk, i.e., $H^i(M, \mathbb{Q})=0$ for all $i \ge 1$.
\end{definition}

We remark that a Wahl singularity is a cyclic quotient surface singularity of type \[\frac{1}{n^2}(1, na-1)\] and it can be obtained by iterating the steps described in Proposition~\ref{proposition:T-algorithm} (ii) to
\begin{tikzpicture}[scale=0.5]
\node[bullet] (00) at (0,0) [label=above:{$-4$}] {};
\end{tikzpicture}.

\begin{definition}[{Behnke--Christophersen~\cite[p.882]{Behnke-Christophersen-1994}}]
\label{definition:$M$-resolution}
An \emph{$M$-resolution} of a quotient surface singularity $(X,0)$ is a partial resolution $f \colon Y_M \to X$ such that
\begin{enumerate}
\item $Y_M$ has only Wahl singularities.

\item $K_{Y_M}$ is nef relative to $f$, i.e., $K_{Y_M} \cdot E \ge 0$ for all $f$-exceptional curves $E$.
\end{enumerate}
\end{definition}

Notice that the minimal resolution is an $M$-resolution.

\begin{theorem}[{Behnke--Christophersen~\cite[3.1.4, 3.3.2, 3.4]{Behnke-Christophersen-1994}}]
Let $(X,0)$ be a quotient surface singularity. Then

\begin{enumerate}
\item Each $P$-resolution $Y \to X$ is dominated by a unique $M$-resolution $Y_M \to X$, i.e., there is a surjection $g \colon Y_M \to Y$, with the property that $K_{Y_M} = g^{\ast}{K_Y}$.

\item There is a surjective map $\DefQG(Y_M) \to \DefQG(Y)$ induced by blowing down deformations.

\item There is a one-to-one correspondence between the components of $\Def(X)$ and $M$-resolutions of $X$.
\end{enumerate}
\end{theorem}

We briefly recall how to construct the $M$-resolution corresponding to a given $P$-resolution. At first, we describe a special $M$-resolution of a singularity of class $T$. Let $Z_0$ be a cyclic quotient surface singularity of type $\frac{1}{dn^2}(1,dna-1)$. The \emph{crepant} $M$-resolution $Y_0 \to Z_0$ of $Z_0$ is defined by the following partial resolution of $Z_0$: $Y_0$ has $d-1$ exceptional components $C_i \cong \mathbb{CP}^1$ ($i=1,\dotsc,d-1$) and $d$ singular points $P_i$ of type $\frac{1}{n^2}(1, na-1)$ as described in the following figure:
\begin{center}
\includegraphics{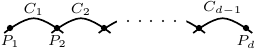}
\end{center}
The proper transforms of $C_i$'s in the minimal resolution $\widetilde{Y}_0$ of $Y_0$ are $(-1)$-curves. So the minimal resolution $\widetilde{Y}_0$ is given by
\begin{equation*}
\begin{tikzpicture}
\node[longrectangle] (00) at (0,0) [] {\ {\Large $\ast$}\raisebox{0.2em}{---}{\Large $\ast$}\ \ };

\node[bullet] (10) at (1,0) [label=above:{$-1$}] {};

\node[longrectangle] (20) at (2,0) [] {\ {\Large $\ast$}\raisebox{0.2em}{---}{\Large $\ast$}\ \ };

\node[bullet] (30) at (3,0) [label=above:{$-1$}] {};

\node[empty] (350) at (3.5,0) [] {};
\node[empty] (40) at (4,0) [] {};

\node[bullet] (450) at (4.5,0) [label=above:{$-1$}] {};

\node[longrectangle] (550) at (5.5,0) [] {\ {\Large $\ast$}\raisebox{0.2em}{---}{\Large $\ast$}\ \ };

\draw [-] (00)--(10);
\draw [-] (10)--(20);
\draw [-] (20)--(30);
\draw [-] (30)--(350);
\draw [dotted] (30)--(450);
\draw [-] (40)--(450);
\draw [-] (450)--(550);
\end{tikzpicture}
\end{equation*}
where
\begin{tikzpicture}
\node[longrectangle] (00) at (0,0) [] {\ {\Large $\ast$}\raisebox{0.2em}{---}{\Large $\ast$}\ \ };
\end{tikzpicture}
is the minimal resolution of the singularity $\frac{1}{n^2}(1,na-1)$. One can check that the above linear chain contracts to the singularity $Z_0=\frac{1}{dn^2}(1,dna-1)$.

The $M$-resolution of a given $P$-resolution $Y \to X$ of a quotient surface singularity $(X,0)$ in the theorem above is obtained by taking the crepant $M$-resolutions for each singularity of class $T$ in $Y$. So we call this $M$-resolution again by the \emph{crepant $M$-resolution} corresponding to a $P$-resolution.

\begin{lemma}\label{lemma:$M$-resolution-by-blowing-up}
The minimal resolution of the crepant resolution $Y_M \to X$ corresponding to a $P$-resolution $Y \to X$ of a quotient surface singularity is obtained by blowing up appropriately the minimal resolution of $Y$.
\end{lemma}

\begin{example}
The singularity $\frac{1}{4d}(1,2d-1)$ is of class $T$ and its minimal resolution is given by
\begin{equation*}
\begin{tikzpicture}
\node[bullet] (10) at (1,0) [label=above:{$-3$}] {};
\node[bullet] (20) at (2,0) [label=above:{$-2$}] {};

\node[empty] (250) at (2.5,0) [] {};
\node[empty] (30) at (3,0) [] {};

\node[bullet] (350) at (3.5,0) [label=above:{$-2$}] {};
\node[bullet] (450) at (4.5,0) [label=above:{$-3$}] {};

\draw [-] (10)--(20);
\draw [-] (20)--(250);
\draw [dotted] (20)--(350);
\draw [-] (30)--(350);
\draw [-] (350)--(450);

\draw [thick, decoration={brace,mirror,raise=0.5em}, decorate] (20) -- (350)

node [pos=0.5,anchor=north,yshift=-0.55em] {$(d-2)$-times};
\end{tikzpicture}
\end{equation*}
Then the minimal resolution of its crepant $M$-resolution is obtained by blowing up all edges:
\begin{equation*}
\begin{tikzpicture}
\node[bullet] (00) at (0,0) [label=above:{$-4$}] {};
\node[bullet] (10) at (1,0) [label=above:{$-1$}] {};
\node[bullet] (20) at (2,0) [label=above:{$-4$}] {};

\node[empty] (250) at (2.5,0) [] {};
\node[empty] (30) at (3,0) [] {};

\node[bullet] (350) at (3.5,0) [label=above:{$-4$}] {};
\node[bullet] (450) at (4.5,0) [label=above:{$-1$}] {};
\node[bullet] (550) at (5.5,0) [label=above:{$-4$}] {};

\draw [-] (00)--(10);
\draw [-] (10)--(20);
\draw [-] (20)--(250);
\draw [dotted] (20)--(350);
\draw [-] (30)--(350);
\draw [-] (350)--(450);
\draw [-] (450)--(550);
\end{tikzpicture}
\end{equation*}
where the number of $(-4)$-curves is $d$ and that of $(-1)$-curves is  $d-1$.
\end{example}


\section{Milnor fibers as complements of compactifying divisors}
\label{section:Milnor-fiber-as-complements}

Let $\mathcal{X} \to \Delta$ be a smoothing of $(X,0)$. We describe a Milnor fiber of the smoothing $\mathcal{X} \to \Delta$ as the complement of its compactifying divisor contained in a certain rational complex surface.

As we have seen in Section~\ref{section:$P$-resolution}, according to Behnke--Christophersen~\cite{Behnke-Christophersen-1994}, there is a one-to-one correspondence between the (reduced irreducible) components of the versal deformation space $\Def(X)$ and the $M$-resolutions of $X$. More explicitly:
\begin{proposition}[Behnke--Christophersen~\cite{Behnke-Christophersen-1994}]
For each smoothing $\mathcal{X} \to \Delta$ of $X$, there exist an $M$-resolution $\phi \colon Y \to X$ and a $\mathbb{Q}$-Gorenstein smoothing $\mathcal{Y} \to \Delta$ of $Y$ such that the smoothing $\mathcal{Y} \to \Delta$ blows down to the smoothing $\mathcal{X} \to \Delta$. These maps give the following commutative diagram
\begin{center}
\includegraphics{figure-XY.pdf}
\end{center}
\end{proposition}

Since the morphism $Y \to X$ is birational and $K_{\mathcal{Y}}$ is nef, we have the following.

\begin{proposition}
\label{proposition:Milnor-fiber=rationally-blown-down}
Every Milnor fiber of the singularity $(X,0)$ is diffeomorphic to the general fiber $Y_t$ of a $\mathbb{Q}$-Gorenstein smoothing $\mathcal{Y} \to \Delta$ of the corresponding $M$-resolution $Y \to X$.
\end{proposition}

In topological viewpoint,  it means that

\begin{corollary}\label{corollary:Milnor-fiber=Rational-blow-down}
Every Milnor fiber of the singularity $(X,0)$ is diffeomorphic to a smooth 4-manifold $Y_t$ which is obtained from the central fiber $Y$ by rationally blowing down Wahl singularities in $Y$.
\end{corollary}

We now show that a general fiber of a smoothing $\mathcal{X} \to \Delta$ of $X$ may be described as a complement of the compactifying divisor in a general fiber of a smoothing of a certain partial resolution of the $\mathbb{C}^\ast$-compactification of $X$, which may be regarded as a compactification of $Y$; Proposition~\ref{proposition:Milnor-fiber=complement}.

Let $\overline{X}$ be the singular natural compactification of $X$ and let $\widehat{X}$ be the natural compactification of $X$ defined in Section~\ref{section:Compactifying-divisor}. Let $\overline{E}_{\infty}$ and $E_{\infty}$ be the corresponding compactifying divisors. We showed in Proposition~\ref{proposition:Extendable-cyclic} and Proposition~\ref{proposition:Extendable-non-cyclic} that any smoothing of $X$ can be extended to deformations of $\overline{X}$ and $\widehat{X}$ preserving $\overline{E}_{\infty}$ and $E_{\infty}$, respectively. So the smoothing $\mathcal{X} \to \Delta$ is extended to a deformation $\overline{\mathcal{X}} \to \Delta$ of $\overline{X}$ which is a locally trivial deformation near $\overline{E}_{\infty}$ and to a smoothing $\widehat{\smash[b]{\mathcal{X}}} \to \Delta$ of $\widehat{X}$ which preserves $E_{\infty}$.

Now we extend the deformations $\mathcal{X}$, $\overline{\mathcal{X}}$, and $\widehat{\smash[b]{\mathcal{X}}}$ to deformations of certain partial resolutions of $X$, $\overline{X}$, and $\widehat{X}$, respectively. Let $f \colon Y \to X$ be the $M$-resolution (or the $P$-resolution) of $X$ corresponding to the smoothing $\mathcal{X} \to \Delta$. Take a partial resolution $\overline{f} \colon \overline{Y} \to \overline{X}$ of $\overline{X}$ corresponding to $f$. Let $\widehat{Y}$ be the minimal resolution of the cyclic quotient surface singularities on $\overline{E}_{\infty} \subset \overline{Y}$.

\begin{definition}
For an $M$-resolution (or a $P$-resolution) $Y \to X$, we call $\widehat{Y}$ the \emph{natural compactification} of $Y$.
\end{definition}

The smoothing $\mathcal{Y} \to \Delta$ extends to a deformation $\overline{\mathcal{Y}} \to \Delta$ of $\overline{Y}$ which is again a locally trivial deformation near $\overline{E}_{\infty}$. Let $\widehat{\smash[b]{\mathcal{Y}}} \to \Delta$ be the simultaneous resolution of the cyclic quotient surface singularities along $\overline{E}_{\infty}$ in each fiber of $\overline{\mathcal{Y}} \to \Delta$. Then the deformations $\mathcal{Y}$, $\overline{\mathcal{Y}}$, and $\widehat{\smash[b]{\mathcal{Y}}}$ blow down to the deformations $\mathcal{X}$, $\overline{\mathcal{X}}$, and $\widehat{\smash[b]{\mathcal{X}}}$ so that the diagram in Figure~\ref{figure:diagram-of-deformations} commutes.

\begin{figure}[t]
\centering
\includegraphics{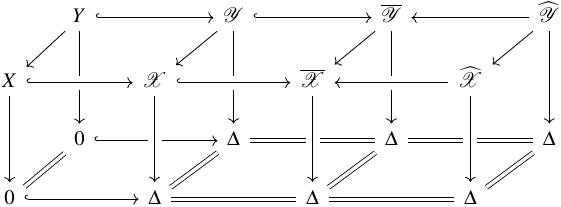}
\caption{A commutative diagram of deformations}
\label{figure:diagram-of-deformations}
\end{figure}


\begin{proposition}\label{proposition:Milnor-fiber=complement}
A general fiber $X_t$ ($t \neq 0$) of the smoothing $\mathcal{X} \to \Delta$ is isomorphic to the complement $\widehat{Y}_t - E_{\infty}$,  where $\widehat{Y}_t$ is a general fiber of the smoothing $\widehat{\smash[b]{\mathcal{Y}}} \to \Delta$.
\end{proposition}

\begin{proof}
A general fiber $X_t$ of the smoothing $\mathcal{X} \to \Delta$ is isomorphic to $\overline{X}_t - \overline{E}_{\infty}$.
On the other hand, $X_t$ is also isomorphic to a general fiber $Y_t$ of the smoothing $\mathcal{Y} \to \Delta$.
Therefore the assertion follows.
\end{proof}

\begin{definition}\label{definition:compactified-Milnor-fiber}
A general fiber $\widehat{Y}_t$ ($t \neq 0$) of the smoothing $\widehat{\smash[b]{\mathcal{Y}}} \to \Delta$ is called a \emph{compactified Milnor fiber} of a smoothing $\mathcal{X} \to \Delta$ of a quotient surface singularity $X$.
\end{definition}

\section{Semi-stable minimal model program}
\label{section:Semi-stable-MMP}

We apply the semi-stable minimal model program to the smoothing $\widehat{\smash[b]{\mathcal{Y}}} \to \Delta$ of the natural compactification $\widehat{Y}$ of the crepant $M$-resolution $Y \to X$ in order to identify the compactified Milnor fiber $\widehat{Y}_t$ as a rational complex surface. It turns out that we will use very particular $3$-fold divisorial contractions and flips. We refer to HTU~\cite{Hacking-Tevelev-Urzua-2013} for a general picture in relation to normal degenerations with only quotient surface singularities.

\subsection{Divisorial contractions and flips}

Let us recall some basics from Koll{\'a}r--Mori~\cite{Kollar-Mori-1992}.

\begin{definition}
A three dimensional \emph{extremal neighborhood} is a proper birational morphism $f \colon (C \subset \mathcal{W}) \to (Q \in \mathcal{Z})$ satisfying the following properties:
\begin{enumerate}[(i)]
\item The canonical class $K_{\mathcal{W}}$ is $\mathbb{Q}$-Cartier and $\mathcal{W}$ has only terminal singularities;

\item $\mathcal{Z}$ is normal with a distinguished point $Q \in \mathcal{Z}$;

\item $C=f^{-1}(Q)$ is an irreducible curve;

\item $K_{\mathcal{W}} \cdot C < 0$.
\end{enumerate}
If the exceptional set of $f$ is an irreducible divisor, then the extremal neighborhood is said to be \emph{divisorial}. Otherwise it is said to be \emph{flipping}.
\end{definition}

In case of flipping, $K_{\mathcal{Z}}$ is not $\mathbb{Q}$-Cartier, and so one performs the following birational operation.

\begin{definition}
The \emph{flip} of a flipping extremal neighborhood $f \colon (C \subset \mathcal{W}) \to (Q \in \mathcal{Z})$  (or, if no confusion is likely, the flip of $\mathcal{W}$) is a proper birational morphism $f^+ \colon (C^+ \subset \mathcal{W}^+) \to (Q \in \mathcal{Z})$, where $\mathcal{W}^+$ is normal with only terminal singularities such that the exceptional set of $f^+$ is $C^+$ and $K_{\mathcal{W}^+}$ is $\mathbb{Q}$-Cartier and $f^+$-ample.
\end{definition}

Note that the flip of $\mathcal{W}$ induces a birational map $\mathcal{W} \dashrightarrow \mathcal{W}^+$ to which we also refer as the flip. By Mori~\cite{Mori-1988} $3$-fold flips always exist; they are unique (see Koll{\'a}r--Mori~\cite{Kollar-Mori-1998}).

HTU~\cite{Hacking-Tevelev-Urzua-2013} explicitly described divisorial and flipping extremal neighborhoods of a special type, which naturally appear in the context of the Koll{\'a}r--Shepherd-Barron--Alexeev compactification of the moduli space of surfaces of general type (see e.g. Urz\'ua~\cite{Urzua-2013,Urzua-2013b}).

\begin{definition}
Let $f \colon W \to Z$ be a partial resolution of a two dimensional cyclic quotient surface singularity germ $(Q \in Z)$ such that $f^{-1}(Q)=C$ is a smooth rational curve with one (or two) Wahl singularity(ies) of $W$ on it. Suppose that $K_W \cdot C < 0$. Let $\mathcal{W} \to \Delta$ be a $\mathbb{Q}$-Gorenstein smoothing of $W$ and let $\mathcal{Z} \to \Delta$ be the corresponding blown-down deformation of $Z$. The induced birational morphism $(C \subset \mathcal{W}) \to (Q \in \mathcal{Z})$ will be called an \emph{extremal neighborhood of type mk1A (or mk2A)}.
\end{definition}

Let $(C \subset \mathcal{W}) \to (Q \in \mathcal{Z})$ be an extremal neighborhood of type mk1A or mk2A. For divisorial contractions, we have

\begin{proposition}[cf.~Urz{\'u}a~{\cite[Proposition~2.8]{Urzua-2013}}]
If $(C \subset \mathcal{W}) \to (Q \in \mathcal{Z})$ is a divisorial extremal neighborhood of type mk1A or mk2A, then $(Q \in \mathcal{Z})$ is a Wahl singularity. The divisorial contraction $\mathcal{W} \to \mathcal{Z}$ induces the blowing down of a $(-1)$-curve between the smooth fibers of $\mathcal{W} \to \Delta$ and $\mathcal{Z} \to \Delta$.
\end{proposition}

On the other hand, a flipping $(C \subset \mathcal{W}) \to (Q \in \mathcal{Z})$ is related to a certain special $P$-resolution of the central fiber $(Q \in Z)$ of $\mathcal{Z}$.

\begin{definition}[HTU~\cite{Hacking-Tevelev-Urzua-2013}]
\label{definition:extremalPres}
An \emph{extremal $P$-resolution} of a two dimensional cyclic quotient surface singularity germ $(Q \in Z)$ is a $P$-resolution $f^+ \colon W^+ \to Z$ such that $C^+=(f^+)^{-1}(Q)$ is a smooth rational curve and $W^+$ has only Wahl singularities (thus at most two; cf. KSB~\cite[Lemma 3.14]{Kollar-Shepherd-Barron-1988}).
\end{definition}

\begin{proposition}[Koll{\'a}r--Mori~{\cite[\S11 and Theorem~13.5]{Kollar-Mori-1992}}]
Suppose that $(C \subset \mathcal{W}) \to (Q \in \mathcal{Z})$ is a flipping extremal neighborhood of type mk1A or mk2A. Let $(C \subset W) \to (Q \in Z)$ be the contraction of $C$ between the central fibers $W$ and $Z$. Then there exists an extremal $P$-resolution $(C^+ \subset W^+) \to (Q \in Z)$ such that the flip $(C^+ \subset \mathcal{W}^+) \to (Q \in \mathcal{Z})$ is obtained by the blown-down deformation of a $\mathbb{Q}$-Gorenstein smoothing of $W^+$. That is, we have the commutative diagram
\begin{center}
\includegraphics{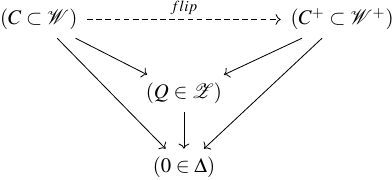}
\end{center}
%
which is restricted to the central fibers as follows:
\begin{equation*}
\includegraphics{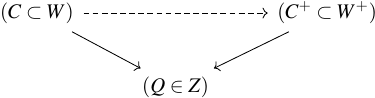}
\end{equation*}
\end{proposition}

\begin{convention}
If $(C \subset \mathcal{W})$ is an extremal neighborhood of type mk1A or m2A, then a flip of $(C \subset \mathcal{W})$  changes only the central fiber $(C \subset W)$ to $(C^+ \subset W^+)$. Hence we sometimes say that a \emph{flip of $(C \subset W)$} (or a \emph{flip of $C$}) instead of a flip of an extremal neighborhood $(C \subset \mathcal{W})$.
\end{convention}

\subsection{Explicit semi-stable minimal model program}

One can describe explicitly the numerical data of the central fibers $(C \subset W) \to (Q \in Z)$ of an extremal neighborhood of type mk1A or mk2A, and of the $(C^+ \subset W^+)$ in an extremal $P$-resolution of $(Q \in Z)$. For details, refer HTU~\cite{Hacking-Tevelev-Urzua-2013} (see also Urz\'ua~\cite[\S2.4]{Urzua-2013}). In this paper we do not meet any extremal neighborhoods of type mk2A. So we consider only extremal neighborhoods of type mk1A.

\subsubsection{$W \to Z$ for mk1A}

Let $(C \subset W)$ be the central fiber of an extremal neighborhood of type mk1A with a Wahl singularity of type $\frac{1}{m^2}(1, ma-1)$ lying on $C$. Let
\begin{equation*}
\frac{m^2}{ma-1} = [e_1, \dotsc, e_s]
\end{equation*}
and let $E_1, \dotsc, E_s$ be the exceptional curves of the minimal resolution $\widetilde{W}$ of $W$ with $E_j^2=-e_j$ for all $j$.

Since $K_W \cdot C < 0$ and $C^2 < 0$, the strict transform of $C$, denoted by $C$ again, is a $(-1)$-curve intersecting only one exceptional curve, say $E_i$, at one point. This data will be written as
\begin{equation*}
[e_1, \dotsc, \overline{e_i}, \dotsc, e_s].
\end{equation*}
Then $(Q \in Z)$ is a cyclic quotient surface singularity of type $\frac{1}{\Delta}(1, \Omega)$ where
\begin{equation*}
\frac{\Delta}{\Omega} = [e_1, \dotsc, e_i-1, \dotsc, e_s].
\end{equation*}

HTU~\cite{Hacking-Tevelev-Urzua-2013} shows that, if $1 < i < s$, then any extremal neighborhood of type mk1A degenerates to two extremal neighborhoods of type mk2A sharing the same birational type (see Urz\'ua~\cite[Prop.2.12]{Urzua-2013}). So one can compute either the flip or the divisorial contraction for any extremal neighborhood of type mk1A through the Mori algorithm \cite{Mori-2002} for extremal neighborhoods of type k2A. But we do not give the details here because we will not use it. The only k1A extremal neighborhoods which appear in this paper are the following.

\begin{proposition}[Usual flips]
The extremal neighborhood $[e_1,\dotsc,e_{s-1},\overline{e_s}]$ is of flipping type. Let $i \in \{1,\dotsc,s\}$ be such that $e_i \geq 3$ and $e_j=2$ for all $j>i$. (If $e_s>2$, then we set $i=s$.)

Then the image of $E_1$ in the $P$-resolution $W^+$ is the curve $C^+$ with the Wahl singularity of type $\frac{1}{m^2}(1, ma-1)$ where $\frac{m^2}{ma-1}=[e_2, \dotsc, e_i-1]$.
\label{proposition:usualflip}
\end{proposition}

\begin{proof}
See Urz\'ua~\cite[Prop.2.15]{Urzua-2013}.
\end{proof}

\subsubsection{Iitaka-Kodaira divisorial contractions} The only divisorial contractions we will encounter in this paper are

\begin{proposition}[Iitaka-Kodaira contractions]
\label{proposition:IitakaKodaira}
Let $F \colon W \subset \mathcal{W} \to 0 \in \Delta$ be a $\mathbb{Q}$-Gorenstein smoothing of a surface $W$ with only Wahl singularities. Let $C$ be a smooth rational curve in $W$ such that it contains no singularity of $W$ and $C^2 =-1$. Then, after possibly shrinking $\Delta$, there is a divisor $\mathcal{C} \subset \mathcal{W}$ such that $F^{-1}(t) \cap \mathcal{C}$ is a $(-1)$-curve in $F^{-1}(t)$ for every $t \in \Delta$, and there exists a contraction $\mathcal{C} \subset \mathcal{W} \to \mathcal{Z}$ of $\mathcal{C}$ giving a new $\mathbb{Q}$-Gorenstein smoothing $F' \colon Z \subset \mathcal{Z} \to 0 \in \Delta$ such that ${F'}^{-1}(t)$ is the blow-down of $F^{-1}(t) \cap \mathcal{C}$ for every $t \in \Delta$.
\end{proposition}

\begin{proof}
See BHPV~\cite[p.154]{BHPV-2004}.
\end{proof}

\subsubsection{Degenerations of curves} \label{subsection:degenerationCurves}
Let $F \colon W \subset \mathcal{W} \to 0 \in \Delta$ be a proper $\mathbb{Q}$-Gorenstein smoothing of a surface $W$ with only Wahl singularities. Let $\Gamma$ be a divisor in $\mathcal{W}$ such that $\Gamma_t :=F^{-1}(t) \cap \Gamma$ is an irreducible proper curve for every $t$. Assume that $\Gamma_0$ does not pass through the singularities of $W$.

We now consider $F \colon W \subset \mathcal{W} \to 0 \in \Delta$ as either Iitaka-Kodaira contraction (Proposition~\ref{proposition:IitakaKodaira}) for some $\mathcal{C} \subset \mathcal{W}$ or usual flip (Proposition~\ref{proposition:usualflip}). We want to know what happens to $\Gamma$ after we perform the corresponding birational operation. If $\Gamma_0$ does not touch $C$, then $\Gamma$ has no change in the new family.

The only extra case we need to consider is when $\Gamma_0$ intersects transversally $C$ at one point; for other situations see Urz\'ua~\cite[\S4]{Urzua-2013b}. For a divisorial contraction, if it does, then it is clear how $\Gamma$ is transformed everywhere.

\begin{proposition} Suppose that $\Gamma_0$ intersects transversally $C$ at one point, and that $F \colon C \subset W \subset \mathcal{W} \to 0 \in \Delta$ is a usual flip. Let $F \colon C^+ \subset W^+ \subset \mathcal{W}^+ \to 0 \in \Delta$ be the flip, where $C^+=E_1$ as in Proposition~\ref{proposition:usualflip}, and let $\Gamma^+ \subset \mathcal{W}^+$ be the proper transform of $\Gamma$. Then $\Gamma^+_0 = \Gamma_0 + C^+$, which forms the toric boundary of the corresponding Wahl singularity in $W^+$.

\label{proposition:degenerationCurves}
\end{proposition}

\begin{proof}
Direct from Proposition~\ref{proposition:usualflip}; see Urz\'ua~\cite[Prop.4.1]{Urzua-2013b} for details.
\end{proof}

In this way, flips will be responsible for broken curves after running MMP.


\section{Identifying Milnor fibers}
\label{section:identifying-Milnor-fibers}

In this section we identify Milnor fibers of quotient surface singularities with their minimal symplectic fillings via running (a controlled) semi-stable minimal model program.

Let $(X,0)$ be a quotient surface singularity, and let $\pi \colon \mathcal{X} \to \Delta$ be a smoothing of $X$ whose Milnor fiber is $M$; that is, $M$ is a general fiber $X_t$ of $\pi$. Let $Y \to X$ be the crepant $M$-resolution corresponding to $\pi$, and let $\mathcal{Y} \to \Delta$ be a $\mathbb{Q}$-Gorenstein smoothing of $Y$ corresponding to $\pi$. We showed in Proposition~\ref{proposition:Milnor-fiber=complement} that the Milnor fiber $M$ can be described as a complement of $E_{\infty}$ in a general fiber $\widehat{Y}_t$ of the smoothing $\widehat{\smash[b]{\mathcal{Y}}} \to \Delta$ of the natural compactification $\widehat{Y}$ of $Y$. That is,
\begin{equation*}
M = X_t = \widehat{Y}_t - E_{\infty}.
\end{equation*}

On the other hand, let $W$ be a minimal symplectic filling of $(X,0)$ and let $Z$ be its natural compactification mentioned in Definition~\ref{definition:natural-compactification-of-filling-cyclic} and Definition~\ref{definition:natural-compactification-of-filling-non-cyclic}. Then we showed in Section~\ref{section:Diffeomorphism-type} that the diffeomorphism type of $W$ is completely determined by the data of intersections of $(-1)$-curves in $Z$ with the compactifying divisor $E_{\infty} \subset Z$. Therefore, in order to identify Milnor fibers as minimal symplectic fillings, we have to get the data from $(E_{\infty} \subset \widehat{Y}_t)$. See Subsection~\ref{subsection:procedure-for-identifying}.

\subsection{Reduction to controlled MMP and smooth deformation}

We first work out certain key properties of Wahl singularities, which will be useful to run MMP in our situation.

\begin{definition}\label{definition:initialCurve}
We know that any Wahl singularity $\frac{1}{n^2}(1,na-1)=[e_1,\dotsc,e_s]$ can be obtained via the algorithm in Proposition~\ref{proposition:T-algorithm} (ii) applied to
\begin{tikzpicture}[scale=0.5]
\node[bullet] (00) at (0,0) [label=above:{$-4$}] {};
\end{tikzpicture}.
Let $e_i$ be the image of
\begin{tikzpicture}[scale=0.5]
\node[bullet] (00) at (0,0) [label=above:{$-4$}] {};
\end{tikzpicture}
under this procedure. The corresponding exceptional curve $E_i$ is called the \emph{initial curve} of the Wahl singularity.
\end{definition}

Notice that $E_i$ is unique up to reordering the indices. For example, the initial curve of $[2,6,2,3]$ is $E_2$, of $[3,2,6,2]$ is $E_3$, and for $[2,2,2,2,2,7,2,2,7]$ is $E_6$.

Next, we want to get control on the discrepancies of Wahl singularities. For this, we explain the point of view in Urz\'ua~\cite[\S4]{Urzua-2013}.

The exceptional divisor of any Wahl singularity $\frac{1}{n^2}(1,na-1)$ can be obtained from a nodal $I_1$ elliptic singular fiber by blowing up over the node. We blow up the node of $I_1$ and subsequent nodes coming from the new $(-1)$-curves. The exceptional divisor appears as the chain of curves of the total
transform of $I_1$ which does not contain the (last) $(-1)$-curve. We denote by $\sigma \colon \widetilde{Z} \to Z$ the composition of blow-ups used. Let $\{E_1, \dotsc, E_s \}$ be the corresponding Wahl configuration where $\frac{1}{n^2}(1,na-1)=[e_1,\dotsc,e_s]$, and $E_i^2=-e_i$. Write $\sigma^*(I_1)= \sum_{i=1}^{s+1} \nu_i E_i$, where $E_{s+1}$ is the $(-1)$-curve, and $\nu_i \geq 1$ are integers. Notice that the initial curve of $\frac{1}{n^2}(1,na-1)$ is the $E_i$ with $\nu_i=1$, i.e., the proper transform of $I_1$.

\begin{lemma}\label{lemma:numerics}
In a situation as above, we have $n=\nu_{s+1}$, $a=\nu_{s+1}-\nu_s$, and the discrepancy of $E_i$ is $-1 + \frac{\nu_i}{\nu_{s+1}}$ for all $i=1,\dotsc,s$.
\end{lemma}

\begin{proof}
See Urz\'ua~\cite[Lemma 4.1]{Urzua-2013}.
\end{proof}

\begin{lemma}\label{lemma:keyInitial}
Let $f^+ \colon C^+ \subset W^+ \to 0 \in Z$ be an extremal $M$-resolution (i.e., an extremal $P$-resolution as in Definition~\ref{definition:extremalPres} but with nef canonical class only) of a cyclic quotient surface singularity $(Z,0)$. Let $\widetilde{W}^+ \to W^+$ be the minimal resolution of the Wahl singularities, if any, and let $\widetilde{Z} \to Z$ be the minimal resolution of $(Z,0)$. Finally, let $g \colon \widetilde{W}^+ \to \widetilde{Z}$ be the corresponding induced map. Then $g$ does not contract the initial curves of the Wahl singularities.
\end{lemma}

\begin{proof}
If $C^+$ is not a $(-1)$-curve, then the claim is trivially true. Therefore, let $C^+$ be a $(-1)$-curve. We will represent the $P$-resolution by
\begin{equation*}
[c_1,\dotsc,c_r]-[e_1,\dotsc,e_s]
\end{equation*}
which gives the continued fractions of both Wahl singularities, and the dash says that $C^+$ is intersecting the exceptional curves corresponding to $c_r$ and $e_1$; is the notation in Urz\'ua~\cite[\S2.4]{Urzua-2013}. We notice that
\begin{equation*}
K_{W^+} \cdot C^+ = -1 - d(c_r) - d(e_1) \geq 0
\end{equation*}
where $d(c_r)$ and $d(e_1)$ are the discrepancies of the curves corresponding to $c_r$ and $e_1$ respectively (computed as in Lemma~\ref{lemma:numerics}).

We are going to use induction on the length $s$ of $[e_1,\dotsc,e_s]$. If $s=1$, then $e_1=4$, and so $c_r>2$. This is because by Lemma~\ref{lemma:numerics}, the discrepancy corresponding to an end curve of self-intersection $-2$ is $> -\frac{1}{2}$. In this way, since the discrepancy of $[4]$ is $-\frac{1}{2}$, we must have $c_r>2$ because of the nonnegativity of $K_{W^+} \cdot C^+$. Hence, no initial curve of the Wahl singularities is contracted by $g$.

We now assume that it is true for a fixed $s$ and all Wahl continued fractions $[c_1,\dotsc,c_r]$. For length $s+1$ we have two possibilities:

\medskip

\textit{Case~1}. The $P$-resolution looks like $[2,g_2,\dotsc,g_r]-[2,e_1,\dotsc,e_s +1]$.

\medskip

We want to know if the contraction kills the initial curve of $[2,e_1,\dotsc,e_s +1]$. This is the same as asking if the contraction of $[g_2,\dotsc,g_r -1]-[e_1,\dotsc,e_s]$ kills the initial curve of $[e_1,\dotsc,e_s]$, since the curve corresponding to $g_1=2$ is not necessary ($P$-resolutions with a central $(-1)$-curve are formed by blowing-up nodes in the minimal resolution). By Lemma~\ref{lemma:numerics}, the discrepancies for $g_r$ and $e_0:=2$ in $[2,g_2,\dotsc,g_r]-[2,e_1,\dotsc,e_s +1]$ can be written as
\begin{equation*}
\text{$d(g_r)= -\dfrac{\gamma + \delta}{\gamma + 2\delta}$ and $d(e_0)=-\dfrac{\beta}{\alpha+2\beta}$},
\end{equation*}
where $\alpha,\beta,\gamma,\delta$ are positive integers. The nonnegativity of $-1-d(g_r)-d(e_0)$ implies that
\begin{equation*}
\alpha \delta \leq \gamma \beta.
\end{equation*}
On the other hand and by Lemma~\ref{lemma:numerics}, the discrepancies for $g_r-1$ and $e_1$ in $[g_2,\dotsc,g_r-1]-[e_1,\dotsc,e_s]$ are $d(g_r-1)=-\frac{\gamma}{\gamma+\delta}$ and $d(e_1)=-\frac{\beta}{\alpha+\beta}$, and $\alpha \delta \leq \gamma \beta$ implies that $-1-d(g_r-1)-d(e_1)\geq 0$, and so $[g_2,\dotsc,g_r-1]-[e_1,\dotsc,e_s]$ is an extremal $M$-resolution. By induction, it does not contract the initial curves.

\medskip

\textit{Case~2}. The $P$-resolution looks like $[g_1,\dotsc,g_{r-1},2]-[e_1+1,\dotsc,e_s,2]$.

\medskip

We use the same strategy as above, the analogous computations show that the hypothesis of induction works to show that the initial curves are not contracted.
\end{proof}

We now return to our situation. We have a $\mathbb{Q}$-Gorenstein smoothing $\widehat{\smash[b]{\mathcal{Y}}} \to \Delta$ of the natural compactification $\widehat{Y}$ of $Y$ by $E_{\infty}$. The minimal resolution $\widetilde{Y} \to \widehat{Y}$ of all Wahl singularities in $\widehat{Y}$ has the natural fibration (defined in Section~\ref{section:Compactifying-divisor}) $\widetilde{Y} \to \mathbb{CP}^1$, such that the exceptional divisors of the Wahl singularities lie on some finite fibers (in our case at most three) together with the section $E_0$. These fibers will be called \emph{arms} (following Stevens~\cite[\S2]{Stevens-1993}). If $E_0$ is part of the exceptional divisor of a Wahl singularity, then any arm not containing a curve of that exceptional divisor will be a \emph{noncentral arm}.

The following theorem is a particular case of Urz\'ua~\cite[Thm.3.4]{Urzua-2013b}, but it includes an explicit and controlled way to run MMP. It also generalizes to any normal $M$-resolution of a rational singularity with star shaped resolution, as in Pinkham~\cite{Pinkham-1977}.

\begin{theorem}\label{theorem:smoothable-by-flips}
By applying only Iitaka-Kodaira divisorial contractions and usual flips to curves coming from the arms of $\widetilde{Y}$, we can run MMP to $\widetilde{Y} \subset \widehat{{\mathcal{Y}}} \to 0 \in \Delta$ until we obtain a deformation $W \subset \mathcal{W} \to 0 \in \Delta$ whose central fiber $W$ is smooth.
\end{theorem}

\begin{proof}
The procedure is very simple. We first work with noncentral arms, let us consider one of them. Let $[e_1,\dotsc,e_s]$ be the ending exceptional divisor of a Wahl singularity in this arm, i.e. the one closer to $E_{\infty}$, with the curve $E_s$ (corresponding to $e_s$) the closest to $E_{\infty}$. Then, between $E_s$ and $E_{\infty}$ there is a $(-1)$-curve $C$ (in the arm). If $C$ does not touch $E_s$, then we apply an Iitaka-Kodaira contraction. If it touches $E_s$, then we apply an usual flip.

After that, we have two possibilities:

\textit{Case~1}. If there is no other exceptional divisor of Wahl singularity in this arm, then there has to be more $(-1)$-curves in that arm which will be contracted (and so Iitaka-Kodaira applies) until one touches the curve $E_s$. When that happens we have a usual flip, which we do apply. We repeat this process until there is no Wahl singularity to flip in this arm. It is easy to see that there will be always $(-1)$-curves ``over" the proper transform of $[e_1,\dotsc,e_s]$, which allow usual flips or Iitaka-Kodaira divisorial contractions.

\textit{Case~2}. If there is another exceptional divisor of Wahl singularity in this arm, say $[g_1,\dotsc,g_r]$, then $[g_1,\dotsc,g_r]$ and $[e_1,\dotsc,e_s]$ together with the curves in between them, form an $M$-resolution of some cyclic quotient surface singularity. If we have more than one curve in between, then none of them can be a $(-1)$-curve (otherwise intersection with canonical class in the contraction would be negative). So, if there are more than one curve in between (or one curve but no $(-1)$-curve), we can keep contracting and/or flipping as in \textit{Case~1}. (so ``over" $[e_1,\dotsc,e_s]$) until there is no Wahl singularity related to $[e_1,\dotsc,e_s]$. Then we repeat the process with the next Wahl singularity.

If there is one $(-1)$-curve between $[g_1,\dotsc,g_r]$ and $[e_1,\dotsc,e_s]$, then it follows by Lemma~\ref{lemma:keyInitial} that the successive contractions of
\begin{equation*}
[g_1,\dotsc,g_r]-[e_1,\dotsc,e_s]
\end{equation*}
do not contract the initial curves of any of them, in particular, of $[e_1,\dotsc,e_s]$. So, to kill the initial curve we need $(-1)$-curves ``over" $[e_1,\dotsc,e_s]$. The point is that usual flips never kill the initial curve, and so we keep going until there is no Wahl singularity left from $[e_1,\dotsc,e_s]$. Then we repeat the process with the next Wahl singularity.

After we eliminate all Wahl singularities from noncentral arms, we work with the central arms. Notice that there are at most two of them. We take one of them, and do \textit{Case~1}. and \textit{Case~2}. above until a Wahl singularity is not completely contained in the arm, this is, the curve $E_0$ is part of the exceptional divisor. We do the same with the other arm.

After that we keep flipping and/or contracting in each arm, this last Wahl singularity, until the only left part is the curve $E_0$. This is possible because each arm has to end up as a $\mathbb{CP}^1$ fiber of the natural fibration to $\mathbb{CP}^1$. At the end, the only possibility is to have $E_0^2=-4$, but in that case we should have flipped a usual $[\overline{2},2,\dotsc,2,m]$ where $E_0^2=-m$, and the flip of that is smooth by Proposition~\ref{proposition:usualflip}.

Therefore, at the end of running this controlled MMP, we obtain a deformation $W \subset \mathcal{W} \to 0 \in \Delta$ whose central fiber $W$ is smooth.
\end{proof}

\subsection{Procedure for identifying Milnor fibers}
\label{subsection:procedure-for-identifying}

A flip changes only the central fiber. So, in order to track down how a general fiber is changed during the MMP process above, we need to take care of only divisorial contractions, which are just blow-downs of $(-1)$-curves on a general fiber.

\begin{corollary}
In Theorem~\ref{theorem:smoothable-by-flips}, a general fiber $\widehat{Y}_t$ ($t \neq 0$) of the smoothing $\widehat{\smash[b]{\mathcal{Y}}} \to \Delta$ is obtained by blowing up several times a general fiber $W_t$ of the smoothing $\mathcal{W} \to \Delta$.
\end{corollary}

Since $\mathcal{W} \to \Delta$ is a deformation with only smooth fibers, the central fiber $W_0$ is diffeomorphic to a general fiber $W_t$. By comparing $W_0$ and $W_t$, one can get the data of positions of $(-1)$-curves in $W_t$. Finally one can get the data of intersections of $(-1)$-curves with the compactifying divisor $E_{\infty}$ in $\widehat{Y}_t$ by tracking the blow-downs $\widehat{Y}_t \to W_t$ given by flips and divisorial contractions, as explained in Subsection~\ref{subsection:degenerationCurves}.

\begin{example}[continued from Example~\ref{example:cyclic-19/7-$P$-resolution}]
\label{example:cyclic-19/7-MMP}

Let $X$ be a cyclic quotient surface singularity of type $\frac{1}{19}(1,7)$. Let $\widehat{Y} (:=\widehat{Y}_3)$ be the compactified $M$-resolution given by the following dual graph:
\begin{equation*}
\begin{tikzpicture}
\node[] (-30) at (-3,0) {$\widehat{Y}$};

\node[rectangle] (-20) at (-2,0) [label=above:{$-4$},label=above:{}] {};

\node[bullet] (-10) at (-1,0) [label=above:{$-1$}] {};

\node[rectangle] (00) at (0,0) [label=above:{$-5$}] {};

\node[rectangle] (10) at (1,0) [label=above:{$-2$}] {};

\node[bullet] (20) at (2,0) [label=above:{$-1$}] {};

\node[bullet] (30) at (3,0) [label=above:{$-3$}] {};

\node[bullet] (40) at (4,0) [label=above:{$-2$}] {};

\node[bullet] (50) at (5,0) [label=above:{$-3$}] {};

\node[bullet] (60) at (6,0) [label=above:{$-1$}] {};

\node[bullet] (70) at (7,0) [label=above:{$+1$}] {};

\draw [-] (-20)--(-10)--(00)--(10)--(20)--(30)--(40)--(50)--(60)--(70);
\end{tikzpicture}
\end{equation*}
We perform usual flips twice to get a new deformation $\mathcal{W} \to \Delta$ as in Figure~\ref{figure:flips-cyclic}, where $F_i$ is the flipping curve and $F_i^+$ is the flipped curve, and we describe only how the central fiber is changed (because a flip changes only the central fiber).
\begin{figure}
\centering
\begin{tikzpicture}
\node[] (-30) at (-3,0) {$\widehat{Y}$};

\node[rectangle] (-20) at (-2,0) [label=above:{$-4$}] {};

\node[bullet] (-10) at (-1,0) [label=above:{$-1$}] {};

\node[rectangle] (00) at (0,0) [label=above:{$-5$}] {};

\node[rectangle] (10) at (1,0) [label=above:{$-2$}] {};

\node[bullet] (20) at (2,0) [label=above:{$-1$},label=below:{$F_1$}] {};

\node[bullet] (30) at (3,0) [label=above:{$-3$}] {};

\node[bullet] (40) at (4,0) [label=above:{$-2$}] {};

\node[bullet] (50) at (5,0) [label=above:{$-3$}] {};

\node[bullet] (60) at (6,0) [label=above:{$-1$}] {};

\node[bullet] (70) at (7,0) [label=above:{$+1$}] {};

\draw [-] (-20)--(-10)--(00)--(10)--(20)--(30)--(40)--(50)--(60)--(70);

\draw [->] (2.5,-0.5)--(2.5,-1);

\node[rectangle] (-2-15) at (-2,-1.5) [label=above:{$-4$},label=above:{}] {};

\node[bullet] (-1-15) at (-1,-1.5) [label=above:{$-1$},label=below:{$F_2$}] {};

\node[bullet] (0-15) at (0,-1.5) [label=above:{$-4$},label=below:{$F_1^+$}] {};

\node[bullet] (3-15) at (3,-1.5) [label=above:{$-1$}] {};

\node[bullet] (4-15) at (4,-1.5) [label=above:{$-2$}] {};

\node[bullet] (5-15) at (5,-1.5) [label=above:{$-3$}] {};

\node[bullet] (6-15) at (6,-1.5) [label=above:{$-1$}] {};

\node[bullet] (7-15) at (7,-1.5) [label=above:{$+1$}] {};

\draw [-] (-2-15)--(-1-15)--(0-15)--(3-15)--(4-15)--(5-15)--(6-15)--(7-15);

\draw [->] (2.5,-2)--(2.5,-2.5);

\node[] (-3-3) at (-3,-3) {$W_0$};

\node[bullet] (-2-3) at (-2,-3) [label=above:{$-3$},label=below:{$F_2^+$}] {};

\node[bullet] (0-3) at (0,-3) [label=above:{$-3$}] {};

\node[bullet] (3-3) at (3,-3) [label=above:{$-1$},label=below:{$E$}] {};

\node[bullet] (4-3) at (4,-3) [label=above:{$-2$}] {};

\node[bullet] (5-3) at (5,-3) [label=above:{$-3$}] {};

\node[bullet] (6-3) at (6,-3) [label=above:{$-1$}] {};

\node[bullet] (7-3) at (7,-3) [label=above:{$+1$}] {};

\draw [-] (-2-3)--(0-3)--(3-3)--(4-3)--(5-3)--(6-3)--(7-3);
\end{tikzpicture}
\caption{Flips in Example~\ref{example:cyclic-19/7-MMP}}
\label{figure:flips-cyclic}
\end{figure}

We now compare $W_0$ and $W_t = \widehat{Y}_{t}$.
\begin{equation*}
\begin{tikzpicture}
\node[] (20) at (2,0) {$\widehat{Y}_t$};

\node[bullet] (30) at (3,0) [label=above:{$-3$},label=below:{$A_4$}] {};

\node[bullet] (40) at (4,0) [label=above:{$-2$},label=below:{$A_3$}] {};

\node[bullet] (50) at (5,0) [label=above:{$-3$},label=below:{$A_2$}] {};

\node[bullet] (60) at (6,0) [label=above:{$-1$},label=below:{$A_1$}] {};

\node[bullet] (70) at (7,0) [label=above:{$+1$}] {};

\draw [-] (30)--(40)--(50)--(60)--(70);
\end{tikzpicture}
\end{equation*}

Since we do not apply any divisorial contraction, a general fiber $W_t$ is just a rational surface $\widehat{Y}_{t}$ containing the compactifying divisor $E_{\infty}$ whose dual graph is
\begin{equation*}
\begin{tikzpicture}
\node[] (00) at (0,0) {$W_t$};

\node[bullet] (30) at (3,0) [label=above:{$-3$},label=below:{$A_4$}] {};

\node[bullet] (40) at (4,0) [label=above:{$-2$},label=below:{$A_3$}] {};

\node[bullet] (50) at (5,0) [label=above:{$-3$},label=below:{$A_2$}] {};

\node[bullet] (60) at (6,0) [label=above:{$-1$},label=below:{$A_1$}] {};

\node[bullet] (70) at (7,0) [label=above:{$+1$}] {};

\draw [-] (30)--(40)--(50)--(60)--(70);

\node[] (0-1) at (0,-1.5) {$W_0$};

\node[bullet] (1-1) at (1,-1.5) [label=above:{$-3$},label=below:{$B_6$}] {};

\node[bullet] (2-1) at (2,-1.5) [label=above:{$-3$},label=below:{$B_5$}] {};

\node[bullet] (3-1) at (3,-1.5) [label=above:{$-1$},label=below:{$B_4$}] {};

\node[bullet] (4-1) at (4,-1.5) [label=above:{$-2$},label=below:{$B_3$}] {};

\node[bullet] (5-1) at (5,-1.5) [label=above:{$-3$},label=below:{$B_2$}] {};

\node[bullet] (6-1) at (6,-1.5) [label=above:{$-1$},label=below:{$B_1$}] {};

\node[bullet] (7-1) at (7,-1.5) [label=above:{$+1$}] {};

\draw [-] (1-1)--(2-1)--(3-1)--(4-1)--(5-1)--(6-1)--(7-1);
\end{tikzpicture}
\end{equation*}

By comparing $W_0$ and $W_t (=\widehat{Y}_t)$ using Proposition~\ref{proposition:degenerationCurves}, one can conclude that the $(-3)$-curve $A_4$ in $W_t$ breaks into three curves
\begin{tikzpicture}[scale=0.5]
\node[bullet] (00) at (0,0) [label=above:{$-3$}] {};
\node[bullet] (10) at (1,0) [label=above:{$-3$}] {};
\node[bullet] (20) at (2,0) [label=above:{$-1$}] {};

\draw [-] (00)--(10)--(20);
\end{tikzpicture}
containing $B_6$, $B_5$, $B_4$ during the deformation $\mathcal{W} \to \Delta$. Note that the $(-1)$-curve $B_4$ in $W_0$ should survive in $W_t$. Since $B_4$ intersects $B_3$, its image $E_1$ in $W_t$ intersects only $A_3$ at one point.

We now apply Iitaka-Kodaira divisorial contractions twice starting from the divisorial extremal neighborhood $(B_4 \subset \mathcal{W})$ by blowing down $B_4 \subset W_0$ (and the corresponding $E_1 \subset W_t$) as in Figure~\ref{figure:divisorial-contractions-cyclic}. We then obtain a new deformation $\mathcal{W}' \to \Delta$ with central fiber $W_0'$ and general fiber $W_t'$.

\begin{figure}[t]
\centering
\begin{tikzpicture}
\node[] (00) at (0,0) {$W_t$};

\node[bullet] (30) at (3,0) [label=above:{$-3$},label=below:{$A_4$}] {};

\node[bullet] (40) at (4,0) [label=above left:{$-2$},label=below right:{$A_3$}] {};
\node[bullet] (405) at (4,0.5) [label=above:{$-1$}] {};

\node[bullet] (50) at (5,0) [label=above:{$-3$},label=below:{$A_2$}] {};

\node[bullet] (60) at (6,0) [label=above:{$-1$},label=below:{$A_1$}] {};

\node[bullet] (70) at (7,0) [label=above:{$+1$}] {};

\draw [-] (30)--(40)--(50)--(60)--(70);
\draw [-] (40)--(405);

\node[] (0-1) at (0,-1) {$W_0$};

\node[bullet] (1-1) at (1,-1) [label=above:{$-3$},label=below:{$B_6$}] {};

\node[bullet] (2-1) at (2,-1) [label=above:{$-3$},label=below:{$B_5$}] {};

\node[bullet] (3-1) at (3,-1) [label=above:{$-1$},label=below:{$B_4$}] {};

\node[bullet] (4-1) at (4,-1) [label=above:{$-2$},label=below:{$B_3$}] {};

\node[bullet] (5-1) at (5,-1) [label=above:{$-3$},label=below:{$B_2$}] {};

\node[bullet] (6-1) at (6,-1) [label=above:{$-1$},label=below:{$B_1$}] {};

\node[bullet] (7-1) at (7,-1) [label=above:{$+1$}] {};

\draw [-] (1-1)--(2-1)--(3-1)--(4-1)--(5-1)--(6-1)--(7-1);

\draw [->] (4,-1.55)--(4,-1.95);


\node[bullet] (3-25) at (3,-2.5) [label=above:{$-3$},label=below:{$A_4$}] {};

\node[bullet] (4-25) at (4,-2.5) [label=above:{$-1$},label=below:{$A_3$}] {};

\node[bullet] (5-25) at (5,-2.5) [label=above:{$-3$},label=below:{$A_2$}] {};

\node[bullet] (6-25) at (6,-2.5) [label=above:{$-1$},label=below:{$A_1$}] {};

\node[bullet] (7-25) at (7,-2.5) [label=above:{$+1$}] {};

\draw [-] (3-25)--(4-25)--(5-25)--(6-25)--(7-25);


\node[bullet] (1-35) at (1,-3.5) [label=above:{$-3$},label=below:{$B_6$}] {};

\node[bullet] (2-35) at (2,-3.5) [label=above:{$-2$},label=below:{$B_5$}] {};


\node[bullet] (4-35) at (4,-3.5) [label=above:{$-1$},label=below:{$B_3$}] {};

\node[bullet] (5-35) at (5,-3.5) [label=above:{$-3$},label=below:{$B_2$}] {};

\node[bullet] (6-35) at (6,-3.5) [label=above:{$-1$},label=below:{$B_1$}] {};

\node[bullet] (7-35) at (7,-3.5) [label=above:{$+1$}] {};

\draw [-] (1-35)--(2-35)--(4-35)--(5-35)--(6-35)--(7-35);

\draw [->] (4,-4.05)--(4,-4.45);

\node[] (0-5) at (0,-5) {$W_t'$};

\node[bullet] (3-5) at (3,-5) [label=above:{$-2$},label=below:{$A_4'$}] {};


\node[bullet] (5-5) at (5,-5) [label=above:{$-2$},label=below:{$A_2'$}] {};

\node[bullet] (6-5) at (6,-5) [label=above:{$-1$},label=below:{$A_1'$}] {};

\node[bullet] (7-5) at (7,-5) [label=above:{$+1$}] {};

\draw [-] (3-5)--(5-5)--(6-5)--(7-5);

\node[] (0-6) at (0,-6) {$W_0'$};

\node[bullet] (1-6) at (1,-6) [label=above:{$-3$},label=below:{$B_6'$}] {};

\node[bullet] (2-6) at (2,-6) [label=above:{$-1$},label=below:{$B_5'$}] {};



\node[bullet] (5-6) at (5,-6) [label=above:{$-2$},label=below:{$B_2'$}] {};

\node[bullet] (6-6) at (6,-6) [label=above:{$-1$},label=below:{$B_1'$}] {};

\node[bullet] (7-6) at (7,-6) [label=above:{$+1$}] {};

\draw [-] (1-6)--(2-6)--(5-6)--(6-6)--(7-6);
\end{tikzpicture}
\caption{Divisorial contractions in Example~\ref{example:cyclic-19/7-MMP}}
\label{figure:divisorial-contractions-cyclic}
\end{figure}
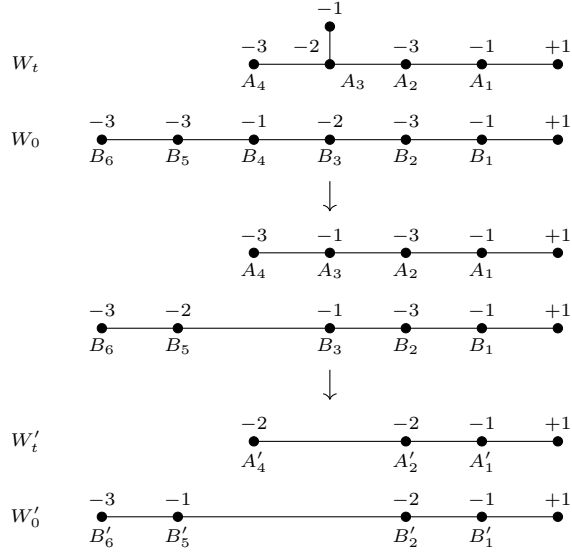

Then the $(-2)$-curve $A_4' \subset W_t'$ is decomposed into two curves
\begin{tikzpicture}[scale=0.5]
\node[bullet] (00) at (0,0) [label=above:{$-3$}] {};
\node[bullet] (10) at (1,0) [label=above:{$-1$}] {};

\draw [-] (00)--(10);
\end{tikzpicture}
in $W_0'$ consisting of $B_6'$ and $B_5'$. As before, the image $E_2'$ of $B_5'$ in $W_t'$ intersects only $A_2'$ at one point. Therefore the data of the intersections of $(-1)$-curves in $\widehat{Y}_t (=W_t)$ with the compactifying divisor $E_{\infty}$ is given as follows:
\begin{equation*}
\begin{tikzpicture}
\node[bullet] (30) at (3,0) [label=below:{$-3$}] {};

\node[bullet] (40) at (4,0) [label=below:{$-2$}] {};
\node[circle] (405) at (4,0.5) [label=left:{$-1$}] {};

\node[bullet] (50) at (5,0) [label=below:{$-3$}] {};
\node[circle] (505) at (5,0.5) [label=left:{$-1$}] {};

\node[bullet] (60) at (6,0) [label=below:{$-1$}] {};

\node[bullet] (70) at (7,0) [label=below:{$+1$}] {};

\draw [-] (30)--(40)--(50)--(60)--(70);
\draw [-] (40)--(405);
\draw [-] (50)--(505);
\end{tikzpicture}
\end{equation*}
Thus, we will see that the $M$-resolution $Y (:=Y_3)$ above corresponds to the sequence $\underline{n}=(2,2,1,3) \in K_4(19/19-7)$ defined in Definition~\ref{definition:K_e(n/n-a)}. Similarly, one can show that the other sequences $(1,2,2,1)$ and $(1,3,1,2)$ in $K_4(19/19-7)$ correspond to the other two $P$-resolutions $Y_1$ and $Y_2$ in Example~\ref{example:cyclic-19/7-$P$-resolution}, respectively.
\end{example}

\begin{example}[Continued from Example~\ref{example:I_30(2-2)+29}]
\label{example:I_30(2-2)+29-I}
Let $(X,0)$ be an icosahedral singularity $I_{30(2-2)+29}$ and let $\widehat{Y}_2$ be the natural compactification of the $M$-resolution $Y_2$ given by the following dual graph:
\begin{equation*}
\begin{tikzpicture}
\node[rectangle] (00) at (0,0) [label=left:{$-2$}] {};

\node[bullet] (11) at (1,1) [label=below:{$-2$}] {};
\node[bullet] (21) at (2,1) [label=below:{$-1$}] {};
\node[bullet] (61) at (6,1) [label=below:{$-2$}] {};

\node[bullet] (10) at (1,0) [label=below:{$-3$}] {};
\node[bullet] (20) at (2,0) [label=below:{$-1$}] {};
\node[bullet] (50) at (5,0) [label=below:{$-2$}] {};
\node[bullet] (60) at (6,0) [label=below:{$-2$}] {};

\node[rectangle] (1-1) at (1,-1) [label=below:{$-5$}] {};
\node[bullet] (2-1) at (2,-1) [label=below:{$-1$}] {};
\node[bullet] (3-1) at (3,-1) [label=below:{$-2$}] {};
\node[bullet] (4-1) at (4,-1) [label=below:{$-2$}] {};
\node[bullet] (5-1) at (5,-1) [label=below:{$-2$}] {};
\node[bullet] (6-1) at (6,-1) [label=below:{$-2$}] {};

\node[bullet] (70) at (7,0) [label=right:{$-1$}] {};

\draw [-] (00)--(11)--(21)--(61)--(70);
\draw [-] (00)--(10)--(20)--(50)--(60)--(70);
\draw [-] (00)--(1-1)--(2-1)--(3-1)--(4-1)--(5-1)--(6-1)--(70);
\end{tikzpicture}
\end{equation*}

We perform two usual flips as in Figure~\ref{figure:flips-I_(30(2-2)+29)-I}, where $F_i$ and $F_i^+$ ($i=1,2$) are the flipping and flipped curves, respectively. As in the example above, applying divisorial contractions as described in Figure~\ref{figure:divisorial-contractions-I_(30(2-2)+29)-I}, we get the information of intersections of $(-1)$-curves with the compactifying divisor given as below.
\begin{equation*}
\begin{tikzpicture}
\node[bullet] (121) at (12,1) [label=below:{$-2$},label=above right:{$D$}] {};

\node[bullet] (110) at (11,0) [label=above right:{$-2$},label=left:{$B$}] {};
\node[bullet] (120) at (12,0) [label=below:{$-2$},label=above:{$B'$}] {};

\node[bullet] (9-1) at (9,-1) [label=below:{$-2$},label=left:{$C_3$}] {};
\node[bullet] (10-1) at (10,-1) [label=below:{$-2$},label=above right:{$C_2$}] {};
\node[bullet] (11-1) at (11,-1) [label=below right:{$-2$},label=above:{$C_1$}] {};
\node[bullet] (12-1) at (12,-1) [label=below:{$-2$},label=above:{$C$}] {};

\node[bullet] (130) at (13,0) [label=right:{$-1$}] {};

\draw [-] (121)--(130);
\draw [-] (110)--(120)--(130);
\draw [-] (9-1)--(10-1)--(11-1)--(12-1)--(130);

\node[circle] (1215) at (12,1.5) [label=above:{$-1$},label=left:{$E_1$}] {};
\node[circle] (1105) at (11,0.5) [label=above:{$-1$},label=left:{$E_2$}] {};
\node[circle] (11-15) at (11,-1.5) [label=below:{$-1$},label=left:{$E_3$}] {};

\draw [-] (121)--(1215);
\draw [-] (110)--(1105);
\draw [-] (11-1)--(11-15);

\node [circle] (91) at (9,1) [label=above:{$-1$},label=left:{$E_4$}] {};
\node [circle] (10-05) at (10,-0.5) [label=above:{$-1$},label=left:{$E_5$}] {};

\draw [-] (9-1)--(91)--(121);
\draw [-] (9-1)--(10-05)--(120);
\end{tikzpicture}
\end{equation*}
The Milnor fiber of the $M$-resolution $Y_2$ corresponds to the minimal symplectic filling \#201 in the list of Bhupal--Ono~\cite{Bhupal-Ono-2012}.

\begin{figure}[tp]
\centering
\begin{tikzpicture}[scale=0.9]
\node[rectangle] (00) at (0,0) [label=left:{$-2$}] {};

\node[bullet] (11) at (1,1) [label=below:{$-2$}] {};
\node[bullet] (21) at (2,1) [label=below:{$-1$}] {};
\node[bullet] (61) at (6,1) [label=below:{$-2$}] {};

\node[bullet] (10) at (1,0) [label=below:{$-3$}] {};
\node[bullet] (20) at (2,0) [label=below:{$-1$}] {};
\node[bullet] (50) at (5,0) [label=below:{$-2$}] {};
\node[bullet] (60) at (6,0) [label=below:{$-2$}] {};

\node[rectangle] (1-1) at (1,-1) [label=below:{$-5$}] {};
\node[bullet] (2-1) at (2,-1) [label=below:{$-1$},label=above:{$F_1$}] {};
\node[bullet] (3-1) at (3,-1) [label=below:{$-2$}] {};
\node[bullet] (4-1) at (4,-1) [label=below:{$-2$}] {};
\node[bullet] (5-1) at (5,-1) [label=below:{$-2$}] {};
\node[bullet] (6-1) at (6,-1) [label=below:{$-2$}] {};

\node[bullet] (70) at (7,0) [label=right:{$-1$}] {};

\draw [-] (00)--(11)--(21)--(61)--(70);
\draw [-] (00)--(10)--(20)--(50)--(60)--(70);
\draw [-] (00)--(1-1)--(2-1)--(3-1)--(4-1)--(5-1)--(6-1)--(70);
\end{tikzpicture}

$\downarrow$ \medskip

\begin{tikzpicture}[scale=0.9]
\node[bullet] (00) at (0,0) [label=left:{$-2$},label=below:{$F_1^+$}] {};

\node[bullet] (11) at (1,1) [label=below:{$-2$}] {};
\node[bullet] (21) at (2,1) [label=below:{$-1$}] {};
\node[bullet] (61) at (6,1) [label=below:{$-2$}] {};

\node[bullet] (10) at (1,0) [label=below:{$-3$}] {};
\node[bullet] (20) at (2,0) [label=below:{$-1$}] {};
\node[bullet] (50) at (5,0) [label=below:{$-2$}] {};
\node[bullet] (60) at (6,0) [label=below:{$-2$}] {};

\node[rectangle] (1-1) at (1,-1) [label=below:{$-4$}] {};
\node[bullet] (3-1) at (3,-1) [label=below:{$-1$},label=above:{$F_2$}] {};
\node[bullet] (4-1) at (4,-1) [label=below:{$-2$}] {};
\node[bullet] (5-1) at (5,-1) [label=below:{$-2$}] {};
\node[bullet] (6-1) at (6,-1) [label=below:{$-2$}] {};

\node[bullet] (70) at (7,0) [label=right:{$-1$}] {};

\draw [-] (00)--(11)--(21)--(61)--(70);
\draw [-] (00)--(10)--(20)--(50)--(60)--(70);
\draw [-] (00)--(1-1)--(3-1)--(4-1)--(5-1)--(6-1)--(70);
\end{tikzpicture}

$\downarrow$ \medskip

\begin{tikzpicture}[scale=0.9]
\node[bullet] (00) at (0,0) [label=left:{$-2$}] {};

\node[bullet] (11) at (1,1) [label=below:{$-2$}] {};
\node[bullet] (21) at (2,1) [label=below:{$-1$}] {};
\node[bullet] (61) at (6,1) [label=below:{$-2$}] {};

\node[bullet] (10) at (1,0) [label=below:{$-3$}] {};
\node[bullet] (20) at (2,0) [label=below:{$-1$}] {};
\node[bullet] (50) at (5,0) [label=below:{$-2$}] {};
\node[bullet] (60) at (6,0) [label=below:{$-2$}] {};

\node[bullet] (1-1) at (1,-1) [label=below:{$-3$},label=above:{$F_2^+$}] {};
\node[bullet] (4-1) at (4,-1) [label=below:{$-1$}] {};
\node[bullet] (5-1) at (5,-1) [label=below:{$-2$}] {};
\node[bullet] (6-1) at (6,-1) [label=below:{$-2$}] {};

\node[bullet] (70) at (7,0) [label=right:{$-1$}] {};

\draw [-] (00)--(11)--(21)--(61)--(70);
\draw [-] (00)--(10)--(20)--(50)--(60)--(70);
\draw [-] (00)--(1-1)--(4-1)--(5-1)--(6-1)--(70);
\end{tikzpicture}

\caption{Flips in Example~\ref{example:I_30(2-2)+29-I}}
\label{figure:flips-I_(30(2-2)+29)-I}
\end{figure}

\begin{figure}[tp]
\begin{tikzpicture}[scale=0.85]
\node[bullet] (00) at (0,0) [label=left:{$-2$},label=above:{$C_3$}] {};

\node[bullet] (11) at (1,1) [label=below:{$-2$}] {};
\node[circle] (21) at (2,1) [label=below:{$-1$},label=above:{$E_1$}] {};
\node[bullet] (61) at (6,1) [label=below:{$-2$},label=above:{$D$}] {};

\node[bullet] (10) at (1,0) [label=below:{$-3$}] {};
\node[circle] (20) at (2,0) [label=below:{$-1$},label=above:{$E_2$}] {};
\node[bullet] (50) at (5,0) [label=below:{$-2$},label=above:{$B$}] {};
\node[bullet] (60) at (6,0) [label=below:{$-2$},label=above:{$B'$}] {};

\node[bullet] (1-1) at (1,-1) [label=below:{$-3$}] {};
\node[circle] (4-1) at (4,-1) [label=below:{$-1$},label=above:{$E_3$}] {};
\node[bullet] (5-1) at (5,-1) [label=below:{$-2$},label=above:{$C_1$}] {};
\node[bullet] (6-1) at (6,-1) [label=below:{$-2$},label=above:{$C$}] {};

\node[bullet] (70) at (7,0) [label=right:{$-1$}] {};

\draw [-] (00)--(11)--(21)--(61)--(70);
\draw [-] (00)--(10)--(20)--(50)--(60)--(70);
\draw [-] (00)--(1-1)--(4-1)--(5-1)--(6-1)--(70);
\draw [thick, decoration={brace,raise=0.5em}, decorate] (1-1) -- (4-1)
node [pos=0.5,anchor=north,yshift=2em] {$C_2$};

\draw [-] (8,-1)--(8,1);

\node[bullet] (121) at (12,1) [label=below:{$-2$},label=left:{$D$}] {};

\node[bullet] (110) at (11,0) [label=below:{$-2$},label=left:{$B$}] {};
\node[bullet] (120) at (12,0) [label=below:{$-2$},label=above:{$B'$}] {};

\node[bullet] (9-1) at (9,-1) [label=below:{$-2$},label=above:{$C_3$}] {};
\node[bullet] (10-1) at (10,-1) [label=below:{$-2$},label=above:{$C_2$}] {};
\node[bullet] (11-1) at (11,-1) [label=below right:{$-2$},label=above:{$C_1$}] {};
\node[bullet] (12-1) at (12,-1) [label=below:{$-2$},label=above:{$C$}] {};

\node[bullet] (130) at (13,0) [label=right:{$-1$}] {};

\node[circle] (1215) at (12,1.5) [label=above:{$-1$},label=left:{$E_1$}] {};
\node[circle] (1105) at (11,0.5) [label=above:{$-1$},label=left:{$E_2$}] {};
\node[circle] (11-15) at (11,-1.5) [label=below:{$-1$},label=left:{$E_3$}] {};

\draw [-] (121)--(130);
\draw [-] (110)--(120)--(130);
\draw [-] (9-1)--(10-1)--(11-1)--(12-1)--(130);

\draw [-] (121)--(1215);
\draw [-] (110)--(1105);
\draw [-] (11-1)--(11-15);
\end{tikzpicture}

$\downarrow$\medskip

\begin{tikzpicture}[scale=0.85]
\node[bullet] (00) at (0,0) [label=left:{$-2$},label=above:{$C_3$}] {};

\node[circle] (11) at (1,1) [label=below:{$-1$},label=above:{$E_4$}] {};
\node[bullet] (61) at (6,1) [label=below:{$-1$},label=above:{$D$}] {};

\node[circle] (10) at (1,0) [label=below:{$-1$},label=above:{$E_5$}] {};
\node[bullet] (60) at (6,0) [label=below:{$-1$},label=above:{$B'$}] {};

\node[bullet] (1-1) at (1,-1) [label=below:{$-2$},label=above:{$C_2$}] {};
\node[bullet] (5-1) at (5,-1) [label=below:{$-1$},label=above:{$C_1$}] {};
\node[bullet] (6-1) at (6,-1) [label=below:{$-2$},label=above:{$C$}] {};

\node[bullet] (70) at (7,0) [label=right:{$-1$}] {};

\draw [-] (00)--(11)--(61)--(70);
\draw [-] (00)--(10)--(60)--(70);
\draw [-] (00)--(1-1)--(5-1)--(6-1)--(70);

\draw [-] (8,-1)--(8,1);

\node[bullet] (121) at (12,1) [label=below:{$-1$},label=above left:{$D$}] {};

\node[bullet] (120) at (12,0) [label=below:{$-1$},label=above:{$B'$}] {};

\node[bullet] (9-1) at (9,-1) [label=below:{$-2$},label=left:{$C_3$}] {};
\node[bullet] (10-1) at (10,-1) [label=below:{$-2$},label=above right:{$C_2$}] {};
\node[bullet] (11-1) at (11,-1) [label=below:{$-1$},label=above:{$C_1$}] {};
\node[bullet] (12-1) at (12,-1) [label=below:{$-2$},label=above:{$C$}] {};

\node[bullet] (130) at (13,0) [label=right:{$-1$}] {};

\node [circle] (91) at (9,1) [label=above:{$-1$},label=left:{$E_4$}] {};
\node [circle] (10-05) at (10,-0.5) [label=above:{$-1$},label=left:{$E_5$}] {};


\draw [-] (121)--(130);
\draw [-] (120)--(130);
\draw [-] (9-1)--(10-1)--(11-1)--(12-1)--(130);


\draw [-] (9-1)--(91)--(121);
\draw [-] (9-1)--(10-05)--(120);
\end{tikzpicture}

\caption{Divisorial contractions in Example~\ref{example:I_30(2-2)+29-I}}
\label{figure:divisorial-contractions-I_(30(2-2)+29)-I}
\end{figure}
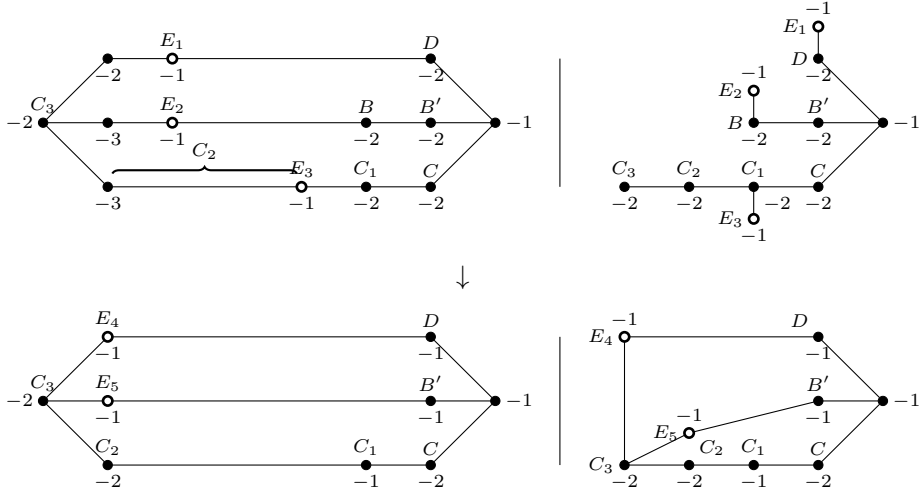
\end{example}

\begin{example}[Continued from Example~\ref{example:I_30(2-2)+29}]
\label{example:I_30(2-2)+29-II}
Let $(X,0)$ be an icosahedral singularity $I_{30(2-2)+29}$ and let $Y_3$ be the $M$-resolution of $X$ given in Example~\ref{example:I_30(2-2)+29}. Applying flips and divisorial contractions to the natural compactification $\widehat{Y}_3$ as described in Figure~\ref{figure:flips-I_(30(2-2)+29)-II}, we can conclude that the Milnor fiber corresponds to the minimal symplectic filling \#254 in the list of Bhupal--Ono~\cite{Bhupal-Ono-2012}.
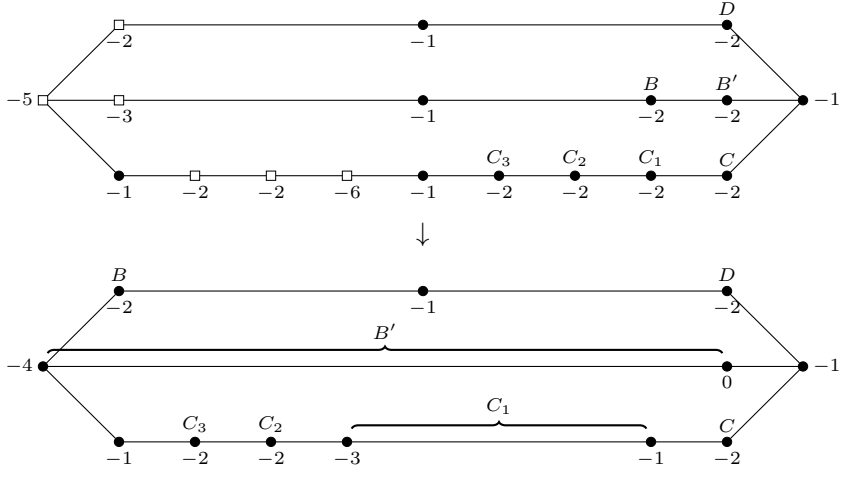
\begin{figure}
\centering
\begin{tikzpicture}
\node[rectangle] (00) at (0,0) [label=left:{$-5$}] {};

\node[rectangle] (11) at (1,1) [label=below:{$-2$}] {};
\node[bullet] (51) at (5,1) [label=below:{$-1$}] {};
\node[bullet] (91) at (9,1) [label=below:{$-2$},label=above:{$D$}] {};

\node[rectangle] (10) at (1,0) [label=below:{$-3$}] {};
\node[bullet] (50) at (5,0) [label=below:{$-1$}] {};
\node[bullet] (80) at (8,0) [label=below:{$-2$},label=above:{$B$}] {};
\node[bullet] (90) at (9,0) [label=below:{$-2$},label=above:{$B'$}] {};

\node[bullet] (1-1) at (1,-1) [label=below:{$-1$}] {};
\node[rectangle] (2-1) at (2,-1) [label=below:{$-2$}] {};
\node[rectangle] (3-1) at (3,-1) [label=below:{$-2$}] {};
\node[rectangle] (4-1) at (4,-1) [label=below:{$-6$}] {};
\node[bullet] (5-1) at (5,-1) [label=below:{$-1$}] {};
\node[bullet] (6-1) at (6,-1) [label=below:{$-2$},label=above:{$C_3$}] {};
\node[bullet] (7-1) at (7,-1) [label=below:{$-2$},label=above:{$C_2$}] {};
\node[bullet] (8-1) at (8,-1) [label=below:{$-2$},label=above:{$C_1$}] {};
\node[bullet] (9-1) at (9,-1) [label=below:{$-2$},label=above:{$C$}] {};

\node[bullet] (100) at (10,0) [label=right:{$-1$}] {};

\draw [-] (00)--(11)--(51)--(91)--(100);
\draw [-] (00)--(10)--(50)--(80)--(90)--(100);
\draw [-] (00)--(1-1)--(2-1)--(3-1)--(4-1)--(5-1)--(6-1)--(7-1)--(8-1)--(9-1)--(100);
\end{tikzpicture}

$\downarrow$\medskip

\begin{tikzpicture}
\node[bullet] (00) at (0,0) [label=left:{$-4$}] {};

\node[bullet] (11) at (1,1) [label=below:{$-2$},label=above:{$B$}] {};
\node[bullet] (51) at (5,1) [label=below:{$-1$}] {};
\node[bullet] (91) at (9,1) [label=below:{$-2$},label=above:{$D$}] {};

\node (10) at (1,0) {};
\node (50) at (5,0) {};
\node (80) at (8,0) {};
\node[bullet] (90) at (9,0) [label=below:{$0$}] {};

\node[bullet] (1-1) at (1,-1) [label=below:{$-1$}] {};
\node[bullet] (2-1) at (2,-1) [label=below:{$-2$},label=above:{$C_3$}] {};
\node[bullet] (3-1) at (3,-1) [label=below:{$-2$},label=above:{$C_2$}] {};
\node[bullet] (4-1) at (4,-1) [label=below:{$-3$}] {};
\node (5-1) at (5,-1)  {};
\node (6-1) at (6,-1)  {};
\node (7-1) at (7,-1)  {};
\node[bullet] (8-1) at (8,-1) [label=below:{$-1$}] {};
\node[bullet] (9-1) at (9,-1) [label=below:{$-2$},label=above:{$C$}] {};

\node[bullet] (100) at (10,0) [label=right:{$-1$}] {};

\draw [-] (00)--(11)--(51)--(91)--(100);
\draw [-] (00)--(1,0)--(5,0)--(8,0)--(90)--(100);
\draw [-] (00)--(1-1)--(2-1)--(3-1)--(4-1)--(5,-1)--(6,-1)--(7,-1)--(8-1)--(9-1)--(100);

\draw [thick, decoration={brace,raise=0.5em}, decorate] (00) -- (90)
node [pos=0.5,anchor=north,yshift=2em] {$B'$};

\draw [thick, decoration={brace,raise=0.5em}, decorate] (4-1) -- (8-1)
node [pos=0.5,anchor=north,yshift=2em] {$C_1$};
\end{tikzpicture}

\caption{Flips and divisorial contractions in Example~\ref{example:I_30(2-2)+29-II}}
\label{figure:flips-I_(30(2-2)+29)-II}
\end{figure}
\end{example}


\section{Applications}\label{section:applications}


\subsection{Cyclic quotient surface singularities}
\label{subsection:algorithm}

We provide another proof of Stevens' result \cite{Stevens-1991} on the one-to-one correspondence between the set of $P$-resolutions of a cyclic quotient surface singularity $(X,0)$ of type $\frac{1}{n}(1,a)$, and the set of zero continued fractions $K_e(n/n-a)$ described in Definition~\ref{definition:K_e(n/n-a)}.

\begin{corollary}\label{corollary:New-proof-of-Stevens-theorem}
Let $(X,0)$ be a cyclic quotient surface singularity of type $\frac{1}{n}(1,a)$. Then there is a one-to-one correspondence between the set of $P$-resolutions of $(X,0)$ and the set $K_e(n/n-a)$.
\end{corollary}

\begin{proof}
We consider the situation and notation in Subsection~\ref{subsection:compactifying-divisor-cyclic}. In there we start with the Hirzebruch surface $\mathbb{F}_1$, two sections $S_0$ and $S_{\infty}$, and a fiber $F$. We blow up appropriately over $S_0 \cap F$ to obtain a surface $\widetilde{X}$ with the configurations of curves of the exceptional divisors of $\frac{1}{n}(1,a)$ and $\frac{1}{n}(1,n-a)$. The proper transforms of $S_0$ and $S_{\infty}$ are $\overline{S}_0$ and $\overline{S}_{\infty}$, and the relevant chain of $\mathbb{CP}^1$'s is

\begin{equation}\label{equation:blowing-up-F_1-1}
\begin{tikzpicture}
\node[bullet] (10) at (1,0) [label=above:{$-b_1$},label=below:{$\overline{S}_0$}] {};
\node[bullet] (20) at (2,0) [label=above:{$-b_2$}] {};

\node[empty] (250) at (2.5,0) [] {};
\node[empty] (30) at (3,0) [] {};

\node[bullet] (350) at (3.5,0) [label=above:{$-b_{r-1}$}] {};
\node[bullet] (450) at (4.5,0) [label=above:{$-b_r$}] {};

\node[bullet] (550) at (5.5,0) [label=above:{$-1$}] {};

\node[bullet] (650) at (6.5,0) [label=above:{$-a_e$},label=below:{$D_e$}] {};
\node[empty] (70) at (7,0) [] {};
\node[empty] (750) at (7.5,0) [] {};
\node[bullet] (80) at (8,0) [label=above:{$-a_2$},label=below:{$D_2$}] {};
\node[bullet] (90) at (9,0) [label=above:{$1-a_1$},label=below:{$D_1$}] {};
\node[bullet] (100) at (10,0) [label=above:{$+1$},label=below:{$\overline{S}_{\infty}$}] {};

\draw [-] (10)--(20);
\draw [-] (20)--(250);
\draw [dotted] (20)--(350);
\draw [-] (30)--(350);
\draw [-] (350)--(450);
\draw [-] (450)--(550);

\draw [-] (550)--(650);
\draw [-] (650)--(70);
\draw [dotted] (650)--(80);
\draw [-] (750)--(80);
\draw [-] (80)--(90);
\draw [-] (90)--(100);
\end{tikzpicture}
\end{equation}
where
\begin{equation*}
\frac{n}{a} = [b_1, \dotsc, b_r], \quad \frac{n}{n-a}=[a_1, \dotsc, a_e].
\end{equation*}

We blow up one more time at $D_1 \cap \overline{S}_{\infty}$, obtaining a surface $X'$ together with a $(-1)$-curve $D_0$ between $D_1$ and $\overline{S}_{\infty}$. Let $X'$ be the contraction of the exceptional divisor corresponding to $\frac{1}{n}(1,a)$.

We know that there is a (very concrete) one-to-one correspondence between $M$-resolutions and $P$-resolutions of $\frac{1}{n}(1,a)$, and between $P$-resolutions and irreducible components of the reduced versal deformation space of $\frac{1}{n}(1,a)$ by Koll\'ar--Shepherd-Barron \cite{Kollar-Shepherd-Barron-1988}.

Let $Y' \to X'$ be an $M$-resolution of $\frac{1}{n}(1,a)$. We recall that there are $\mathbb{Q}$-Gorenstein smoothings $Y' \subset \mathcal{Y}' \to 0 \in \Delta$ such that the general fiber contains the boundary divisor $E'_{\infty}$ given by trivially deforming $D_e,D_{e-1},\dotsc,D_1,D_0$, and $\overline{S}_{\infty}$. We denote these curves by the same letters in the general fiber, when there is no confusion.

We will apply the controlled MMP in Theorem~\ref{theorem:smoothable-by-flips} to $Y' \subset \mathcal{Y}' \to 0 \in \Delta$, in order to obtain $(-1)$-curves in the general fiber which intersect transversally some curves in $D_e,D_{e-1},\dotsc,D_1$. The number of $(-1)$-curves attached to $D_i$ when we run MMP will be denoted by $d_i$. Each of these $(-1)$-curves will intersect transversally $D_i$ at one point, and no other $D_j$ with $j\neq i$ or $(-1)$-curves which appear in this process. When we arrive to a smooth deformation (Theorem~\ref{theorem:smoothable-by-flips}), we keep Iitaka-Kodaira contracting $(-1)$-curves in the fiber until $D_0$ becomes a $\mathbb{CP}^1$ fiber. On the general fiber, we would have contracted all $(-1)$-curves attached to the $D_i$'s and the ones coming from the $D_i$'s, so that we arrive to $D_1$ with $0$ self-intersection. That says that the continued fraction
\begin{equation*}
[a_e-d_e,a_{e-1}-d_{e-1},\dotsc,a_1-d_1]
\end{equation*}
is equal to zero. This is also saying that in order to obtain $Y' \subset \mathcal{Y}' \to 0 \in \Delta$, we start with a trivial deformation $\mathbb{CP}^1 \times \mathbb{CP}^1 \times \Delta \to \Delta$, and then anti-flip (usual flip) or anti-contract (Iitaka-Kodaira contraction) a finite number of times.

To obtain the $(-1)$-curves attached to the $D_i$, we apply the following algorithm (part of what is done in Theorem~\ref{theorem:smoothable-by-flips}). We recall that an $M$-resolution is formed by a chain of $\mathbb{CP}^1$'s such that at the nodes is either a smooth point of the surface or a Wahl singularity, and the same happens for some point in the two ending $\mathbb{CP}^1$'s.

All $d_1,\dotsc,d_e$ start as $0$. Let $\widetilde{Y'} \to Y'$ be the minimal resolution of the Wahl singularities. Consider the $(-1)$-curve in the special fiber of $Y'$ which touches $D_e$. We have two situations for this curve in $Y'$:

\textit{Case~1}. If it does not touch a Wahl singularity (i.e. the corresponding end curve of the $M$-resolution has a smooth point of $Y'$), then (by Iitaka-Kodaira) the curve $D_e$ in the general fiber intersects at one point the corresponding $(-1)$-curve, and so the $d_e$ adds $1$. We now contract this $(-1)$-curve and continue contracting any $(-1)$-curve appearing from $D_e,D_{e-1},\dotsc,D_1$ until there are no such curves. We do this in the family via Iitaka-Kodaira contractions.

\textit{Case~2}. If it does touch a Wahl singularity, then we apply usual flips several times (as in Theorem~\ref{theorem:smoothable-by-flips}) until that singularity becomes smooth. When that happens, there exists $r\geq 0$ such that the configuration of curves $D_e,D_{e-1},\dotsc,D_1$ in the general fiber deforms to
\begin{equation*}
D'_e,D'_{e-1},\dotsc,D_{e-(r-1)},D'_{e-r},E,D'_{e-(r+1)},\dotsc\ldots,D'_1
\end{equation*}
where $E$ is a $(-1)$-curve, $D_{e-r}$ deforms to $D'_{e-r}+E$, and $D_j$ deforms to $D'_j$ for $j \neq e-r$. This can be verified applying Proposition~\ref{proposition:degenerationCurves} several times. Then, by Iitaka-Kodaira, we have a $(-1)$-curve attached to $D_{e-(r+1)}$ which intersects transversally at one point, and so we add $1$ to $d_{e-(r+1)}$. After that, we Iitaka-Kodaira contract $E$ and all $(-1)$-curves coming from the new configuration
\begin{equation*}
D'_e,D'_{e-1},\dotsc,D_{e-(r-1)},D'_{e-r},E,D'_{e-(r+1)},\dotsc,D'_1,
\end{equation*}
until there are no such curves.

\bigskip

After that, we are in a situation of a new $M$-resolution of a some cyclic quotient surface singularity whose dual exceptional divisor is what is left from either
\begin{equation*}
\text{$D_e, \dotsc, D_1$ or $D'_e, D'_{e-1}, \dotsc, D_{e-(r-1)}, D'_{e-r}, E, D'_{e-(r+1)},\dotsc, D'_1$}.
\end{equation*}
Then we apply the algorithm above again.

Conversely, the following is an explicit algorithm to recover the $P$-resolution from a zero continued fraction. We recall that a $P$-resolution of a cyclic quotient surface singularity is formed by a chain of $\mathbb{CP}^1$'s, such that at the nodes is either a T-singularity ($\frac{1}{dn^2}(1,dna-1)$ or rational double point of type $A_m$, $m \geq 1$) or a smooth point (consistently denoted by $A_0$) of the surface, and the same happens for some point in the two ending $\mathbb{CP}^1$'s.

Let $\frac{1}{n}(1,a)$ be a cyclic quotient surface singularity. Let $(n_1, \dotsc, n_e) \in K_e(n/n-a)$, and $d_i:=a_1-n_i$. Consider a chain of $\mathbb{CP}^1$'s $D_e,D_{e-1},\dotsc,D_1$ corresponding to the $a_i$'s, just as in the proof above, and let us attach $d_i$ disjoint $(-1)$-curves to $D_i$, each transversally at one point.

As before, we explain the $P$-resolution associated to $(n_1,\dotsc,n_e)$ constructing the T-singularities and $\mathbb{CP}^1$'s, starting with the $(-1)$-curve $E$ (in the compactified $P$-resolution) connecting $D_e$ with the first $\mathbb{CP}^1$ of the $P$-resolution.

\begin{enumerate}[I.]
\item

    \begin{enumerate}
    \item[(a)] If $d_e \neq 0$, then we have a $A_{d_e-1}$ singularity in the first $\mathbb{CP}^1$.

    \item[(b)] If $d_e=0$, we find the smallest nonnegative integer $r$ such that $d_{e-(r+1)} \neq 0$. Then we have a T-singularity
        \[\dfrac{1}{d_{e-(r+1)}n'^2}(1,d_{e-(r+1)}n'a'-1)\]
        with
        \[\frac{n'}{a'}=[a_e,\dotsc,a_{e-r}].\]
    \end{enumerate}

\item Now we contract all $(-1)$-curves attached to $D_e$ (if I.a) or to $D_{e-(r+1)}$ (if I(b)), and all $(-1)$-curves after that coming from $D_e, D_{e-1},\dotsc,D_1$, until there are none.

\item In the contraction in II, the curve $D_e$ may become a $(-1)$-curve. Let $m$ be the number of blow-downs starting with the contraction of $D_e$ (after it becomes $(-1)$-curve), or $m=0$ if $D_e$ never becomes $(-1)$-curve. Then the first $\mathbb{CP}^1$ (in the minimal resolution of the $P$-resolution) has self intersection $-(2+m)$ if I.a, or $-(1+m)$ if I.b.
\end{enumerate}

After this, we obtain a $P$-resolution for the new cyclic quotient surface singularity, whose dual exceptional divisor is what is left in II. from $D_e,D_{e-1},\dotsc,D_1$. We now repeat the algorithm.
\end{proof}

\begin{remark}
The above procedure only uses the controlled MMP described in Theorem~\ref{theorem:smoothable-by-flips}, recovering Stevens' result \cite{Stevens-1991} on the one-to-one correspondence between the set of $P$-resolutions (for us, passing through the $M$-resolutions of Behnke--Christophersen \cite{Behnke-Christophersen-1994}) of a cyclic quotient surface singularity $(X,0)$ of type $\frac{1}{n}(1,a)$, and the set of zero continued fractions $K_e(n/n-a)$.
More importantly, Corollary~\ref{corollary:New-proof-of-Stevens-theorem} above shows a geometric way to connect directly Lisca's \cite{Lisca-2008}, and Koll\'ar--Shepherd-Barron's \cite{Kollar-Shepherd-Barron-1988} one-to-one correspondences. By N\'emethi--Popescu-Pampu \cite{Nemethi-PPampu-2010}, it also connects Christophersen--Stevens' \cite{Christophersen-1991,Stevens-1991} correspondence in relation to equations of the versal deformation space of $(X,0)$, and in particular we answer the question raised in N{\'e}methi--Popescu-Pampu~\cite[\S11.2]{Nemethi-PPampu-2010}.
\end{remark}


\subsection{Non-cyclic quotient surface singularities}

We apply the machinery developed in the previous section to prove that the Milnor fibres corresponding to irreducible components of the reduced versal deformation space of a non-cyclic quotient surface singularity are pairwise non-diffeomorphic by orientation-preserving diffeomorphisms.

At first, we show that any $P$-resolution $Y$ of a quotient surface singularity $(X,0)$ is completely determined by the data of $(-1)$-curves in the compactified Milnor fiber $\widehat{Y}_t$ intersecting the compactifying divisor $E_{\infty}$. That is,

\begin{lemma}\label{lemma:intersection-data}
Let $Y$ and $Z$ be $P$-resolutions of $(X,0)$. If the dual graphs of the data of $(-1)$-curves and the compactifying divisor $E_{\infty}$ in the compactified Milnor fibers $\widehat{Y}_t$ and $\widehat{Z}_t$ intersecting $E_{\infty}$ are the same, then $Y$ and $Z$ coincide.
\end{lemma}

\begin{proof}
As in the proof of Corollary~\ref{corollary:New-proof-of-Stevens-theorem}, we present an explicit algorithm to recover the $P$-resolution from the data of $(-1)$-curves intersecting the compactifying divisor $E_{\infty}$ in the compactified Milnor fibers $\widehat{Y}_t$.

Let
\begin{equation*}
\begin{tikzpicture}
\node[bullet] (20) at (2,0) [label=above:{$-a_e$}, label=below:{$A_e$}] {};

\node[empty] (250) at (2.5,0) [] {};
\node[empty] (30) at (3,0) [] {};

\node[bullet] (350) at (3.5,0) [label=above:{$-a_1$}, label=below:{$A_1$}] {};

\draw [-] (20)--(250);
\draw [dotted] (20)--(350);
\draw [-] (30)--(350);
\end{tikzpicture}
\end{equation*}
be an arm of $E_{\infty}$ in $\widehat{Y}_t$ where $A_i^2 = -a_i$ ($i=1,\dotsc,e$) and $A_1$ is the closest rational curve to the node of $E_{\infty}$. Let $d_j$ be the number of $(-1)$-curves that intersect only $A_j$ and let $j_0$ be the smallest $j$ among $\{1, \dotsc, e\}$ such that $d_{j_0} \neq 0$, that is, $d_{j_0} \neq 0$ but $d_j = 0$ for all $j \le j_0$.

I. We apply the algorithm described in the proof of Corollary~\ref{corollary:New-proof-of-Stevens-theorem} (to recover a $P$-resolution from a zero continued fraction), instead of a zero continued fraction, with the sequence $(d_e, \dotsc, d_{j_0})$ on the chain $A_e, \dotsc, A_{j_0}$ of $\mathbb{CP}^1$'s, for each arms of $E_{\infty}$.

Each singularities of class $T$ or $A_n$-singularities on any arms of a compactified $P$-resolution $\widehat{Y}$ such that the node $E_0$ is not a part of their exceptional divisors produce bunch of $(-1)$-curves that intersect only the rational curves on the corresponding arms of $E_{\infty}$ in the compactified Milnor fiber $\widehat{Y}_t$. On the other hand, if a singularity contains the node $E_0$ as a part of exceptional divisor, then, after usual flips and Iitaka-Kodaira divisorial contractions, we obtain $(-1)$-curves connecting two different rational curves of $E_{\infty}$ in $\widehat{Y}_t$.

So we can recover all singularities on the given compactified $P$-resolution $\widehat{Y}$ such that the node $E_0$ is not a part of the exceptional divisor of $\widehat{Y}$ from the data $(d_e, \dotsc, d_{j_0})$ of each arms of $E_{\infty}$.

II. We contract all $(-1)$-curves attached to $A_e, \dotsc, A_{j_0}$ starting from $A_e$.

Such contractions correspond to sequence of usual flips and Iitaka-Kodaira divisorial contractions of the whole family $\widehat{\mathcal{Y}}$. Therefore, as we saw in the proof of Theorem~\ref{theorem:smoothable-by-flips}, after keep contracting all the $(-1)$-curves, there may be \emph{at most} one singularity on the central fiber whose exceptional divisor contain the node $E_0$. We denote the central fiber (after the usual flips and Iitaka-Kodaira divisorial contraction) again by $\widehat{Y}$ and the general fiber again $\widehat{Y}_t$.

III. Find the duals of each arms of $\widehat{Y}_t$. Then the minimal resolution of $\widehat{Y}$ (if there is a singularity on $\widehat{Y}$) or $\widehat{Y}$ itself (if there is no singularity) consists of the duals.

On the minimal resolution of $\widehat{Y}$ (if there is a singularity on $\widehat{Y}$) or $\widehat{Y}$ itself (if there is no singularity), the chain of $\mathbb{CP}^1$'s connecting $E_0$ and the node of $E_{\infty}$ is blown down to a $\mathbb{CP}^1$ fiber of the natural fibration to $\mathbb{CP}^1$. Therefore the chain should consists of the arms of $E_{\infty}$ connected with their duals by $(-1)$-curves, which implies that we can recover $\widehat{Y}$ from $\widehat{Y}_t$.
\end{proof}

\begin{remark}
The proof of the above theorem says more: A $P$-resolution is completely determined by the data of $(-1)$-curves that intersect only one of the rational curves of $E_{\infty}$.
\end{remark}

\begin{theorem}\label{theorem:components-Def(X)}
The Milnor fibers associated with irreducible components of the reduced versal deformation space of a quotient surface singularity are non-diffeomorphic to each other, with obvious exceptions from the pairs of $P$-resolutions symmetrical to each other.
\end{theorem}

\begin{proof}
Let $Y$ and $Z$ be $P$-resolutions of a quotient surface singularity $(X,0)$ corresponding to two different irreducible components of the reduced versal deformation space of $X$, respectively. The Milnor fibers associated to $Y$ and $Z$ are given by $\widehat{Y}_t \setminus E_{\infty}$ and $\widehat{Z}_t \setminus E_{\infty}$.

At first, suppose that the rational surfaces $\widehat{Y}_t$ and $\widehat{Z}_t$ are obtained from $\mathbb{CP}^2$ at the same time (or from $\mathbb{CP}^1 \times \mathbb{CP}^1$) by blowing ups. Then it follows by Remark~\ref{remark:the-same--1-data} that two Milnor fibers are diffeomorphic if and only if the two intersection data of $(-1)$-curves with $E_{\infty}$ in $\widehat{Y}_t$ and $\widehat{Z}_t$ are the same, hence, by Lemma~\ref{lemma:intersection-data}, if and only if two $P$-resolutions $Y$ and $Z$ are the same.

Suppose now that $\widehat{Y}_t$ is obtained from $\mathbb{CP}^2$ but $\widehat{Z}$ is obtained from $\mathbb{CP}^1 \times \mathbb{CP}^1$. Then as we saw in the proof of Theorem~\ref{theorem:diffeomorphism-type} that two Milnor fibers $\widehat{Y}_t \setminus E_{\infty}$ and $\widehat{Z}_t \setminus E_{\infty}$ are not diffeomorphic to each other whether the two intersection data of $(-1)$-curves with $E_{\infty}$ in $\widehat{Y}_t$ and $\widehat{Z}_t$ are the same or not. Hence the assertion follows.
\end{proof}

\section{Minimal symplectic fillings are diffeomorphic to Milnor fibers}
\label{section:minimal-symplectic-fillings-are-Milnor-fibers}

In this section we show that every minimal symplectic filling of a \emph{non-cyclic} quotient surface singularity is diffeomorphic to one of its Milnor fibers; Theorem~\ref{theorem:fillings-diffeomorphic-to-Milnors}.

Let $(X,0)$ be a non-cyclic quotient surface singularity, and let $W$ be a minimal symplectic filling of $(X,0)$. As we saw in Remark~\ref{remark:complex-model}, there is a rational complex surface $Z$ containing a divisor $E$ which is a union of smooth rational complex curves whose intersection dual graph is consistent with that of the compactifying divisor $E_{\infty}$ of $(X,0)$ such that $W$ is diffeomorphic to the complement $Z-\nu(E)$. So we need to show that there is a smoothing of $X$ whose Milnor fiber is diffeomorphic to $Z-\nu(E)$.

Pinkham~\cite[Theorem~6.7]{Pinkham-1978} provides a way to construct a smoothing whose Milnor fiber is diffeomorphic to the complement of $E$ under certain cohomological conditions.

\begin{proposition}[{SSW~\cite[Theorem~8.1]{Stipzicz-Szabo-Wahl-2008} and Fowler~\cite[Theorem~2.2.3]{Fowler-2014}}]
\label{proposition:supporting-ample-divisor}
Let $Z$ be a smooth projective rational surface and let $E \subset Z$ be a union of smooth rational curves whose intersection dual graph is consistent with that of $E_{\infty}$. If $E$ supports an ample divisor, then there is a smoothing of $X$ such that its Milnor fiber is diffeomorphic to $Z-\nu(E)$.
\end{proposition}

\begin{proof}
We briefly sketch the proof for the convenience of the reader. Assume, in general, that the dual graph of $E$ (and hence that of $E_{\infty}$) is given as in Figure~\ref{figure:compactifying-divisor-ampleness}, where we denote by $E_1$ the central curve with $E_1 \cdot E_1 = b-3$ and by $E_{ij}$ ($j=1, \dotsc, e_i$) the curve on the $i$-th arm with $E_{ij} \cdot E_{ij} = -a_{ij}$.

\begin{figure}
\centering
\begin{tikzpicture}
\node[bullet] (-250) at (-2.5,0) [label=below:{$-a_{1e_1}$}, label=above:{$E_{1e_1}$}] {};
\node[empty] (-20) at (-2,0) [] {};
\node[empty] (-150) at (-1.5,0) [] {};
\node[bullet] (-10) at (-1,0) [label=below:{$-a_{11}$}, label=above:{$E_{11}$}] {};

\node[bullet] (00) at (0,0) [label=below:{$b-3$},, label=above right:{$E_1$}] {};

\node[bullet] (01) at (0,1) [label=left:{$-a_{21}$}, label=right:{$E_{21}$}] {};
\node[empty] (015) at (0, 1.5) [] {};
\node[empty] (02) at (0,2) [] {};
\node[bullet] (025) at (0,2.5) [label=left:{$-a_{2e_2}$}, label=right:{$E_{2e_2}$}] {};

\node[bullet] (10) at (1,0) [label=below:{$-a_{31}$},, label=above:{$E_{31}$}] {};
\node[empty] (150) at (1.5,0) [] {};
\node[empty] (20) at (2,0) [] {};
\node[bullet] (250) at (2.5,0) [label=below:{$-a_{3e_3}$}, label=above:{$E_{3e_3}$}] {};

\draw [-] (00)--(-10);
\draw [-] (-10)--(-150);
\draw [dotted] (-10)--(-250);
\draw [-] (-20)--(-250);

\draw [-] (00)--(01);
\draw [-] (01)--(015);
\draw [dotted] (01)--(025);
\draw [-] (02)--(025);

\draw [-] (00)--(10);
\draw [-] (10)--(150);
\draw [dotted] (10)--(250);
\draw [-] (20)--(250);
\end{tikzpicture}
\caption{The dual graph of the compactifying divisor $E_{\infty}$}
\label{figure:compactifying-divisor-ampleness}
\end{figure}
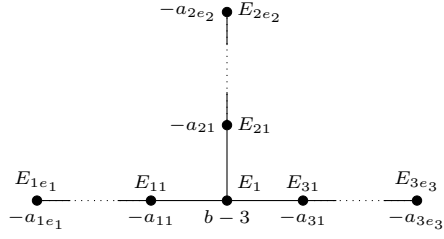

According to Pinkham~\cite[Theorem~6.7]{Pinkham-1978}, in order to prove the existence of such a smoothing, we need to check first that
\begin{equation*}
H^1(Z, E^{(k)})=0
\end{equation*}
for all $k \ge 0$, where
\begin{equation*}
E^{(k)}=kE_1 + \sum_{i=1}^{3} \sum_{j=1}^{e_i} \ceil{k(a_{ij}/a_{i1})} E_{ij}
\end{equation*}
with $E^{(1)}=E_1$. Here $\ceil{x}$ denotes the least integer greater than or equal to $x$. But, by SSW~\cite[Theorem~8.1]{Stipzicz-Szabo-Wahl-2008}, we have $H^1(Z, E^{(k)})=0$ because $Z$ is a rational surface.

Let
\begin{equation*}
R=\bigoplus_{k \ge 0} H^0(Z, E^{(k)}).
\end{equation*}
Set $\mathcal{X}=\Spec{R}$. Let $t$ be a section in $H^0(Z, \sheaf{O(E_1)})$ that defines $E_1$. Then it follows by Pinkham~\cite[Theorem~6.7]{Pinkham-1978} that $t$ defines a flat morphism $\pi \colon \mathcal{X} \to \Delta$ which is a smoothing whose central fiber $\Spec{R/(t)}$ is isomorphic to the singularity $(X,0)$.

It remains to show that the Milnor fiber of the smoothing $\pi$ is diffeomorphic to $Z-\nu(E)$. Contracting the three arms of $E$ on $Z$, we get a singular surface $\overline{Y}$ with three cyclic quotient surface singularities along the curve $\overline{E} \subset \overline{Y}$, where $\overline{E}$ is the image of $E_1$. Since $E$ supports an ample divisor, $Z-E \cong \overline{Y}-\overline{E}$ is affine. Therefore $\overline{E}$ is an ample divisor on $\overline{Y}$. The divisor $E^{(k)}$ is the total transform of $k\overline{E}$ on $\overline{Y}$; cf. Pinkham~\cite[p.80]{Pinkham-1978}. Therefore there is an isomorphism $H^0(\overline{Y}, k\overline{E}) \cong H^0(Z, E^{(k)})$; Pinkham~\cite[Propisition~6.5]{Pinkham-1978}. Thus $\overline{Y}$ is isomorphic to $\Proj{R}$. So the smoothing $\mathcal{X}$ is the affine cone $C(\overline{Y})$ over $\overline{Y}$ embedded by the ample linear series $\abs{\overline{E}}$ into a projective space $\mathbb{CP}^N$, and the map $\pi$ is given by taking hyperplane sections. Therefore a general fiber of $\pi$ is a hyperplane section of the cone $C(\overline{Y})$ with a hyperplane not passing through the vertex of $C(\overline{Y})$. Thus the Milnor fiber is diffeomorphic to $\overline{Y}-\nu(\overline{E}) \cong Z-\nu(E) \cong W$.
\end{proof}

\begin{theorem}\label{theorem:fillings-diffeomorphic-to-Milnors}
Let $(X,0)$ be a non-cyclic quotient surface singularity. For any minimal symplectic filling $W$ of $(X,0)$, there is a smoothing $\mathcal{X} \to \Delta$ of $(X,0)$ such that $W$ is diffeomorphic to a general fiber $X_t$ of $\mathcal{X} \to \Delta$.
\end{theorem}

\begin{proof}
For a given minimal symplectic filling $W$, let $Z$ be the corresponding rational complex surface containing a union $E$ of smooth rational complex curves whose intersection dual graph is consistent with that of $E_{\infty}$ such that $W$ is diffeomorphic to $Z-\nu(E)$. By Proposition~\ref{proposition:supporting-ample-divisor} it is enough to show that $E$ supports an ample divisor in $Z$. That is, we need to prove that there is an effective divisor $F$ supported on $E$ such that $F \cdot F > 0$ and $F \cdot D > 0$ for all irreducible curve $D$ on $Z$.

We will show in Proposition~\ref{proposition:Milnor-fiber=complement} that every Milnor fiber of $X$ is also given as $\widehat{Y}_t - E_{\infty}$, the complement of $E_{\infty}$ from a smooth rational surface $\widehat{Y}_t$. So $\widehat{Y}_t - E_{\infty}$ is affine. Therefore $E_{\infty}$ supports an ample divisor in $\widehat{Y}_t$; hence, there is an effective divisor $E_{\infty}'$ supported on $E_{\infty}$ such that $E_{\infty}' \cdot E_{\infty}' > 0$ and $E_{\infty}' \cdot D > 0$ for any irreducible curve $D$ on $\widehat{Y}_t$; in particular, for any irreducible component, say $E_i$, of the divisor $E_{\infty}$, we have $E_{\infty}' \cdot E_i > 0$.

By assumption, $E$ and $E_{\infty}$ have the same intersection dual graph. So if we take a divisor $F$ on $Z$ supported on $E$ whose coefficients are same with those of $E_{\infty}'$, then we have $F \cdot F > 0$ and $F \cdot E_i > 0$ for any irreducible component $E_i$ of $E$. Hence it remains to show that $F \cdot D > 0$, or equivalently $E \cdot D > 0$ (that is, $E \cap D \neq \varnothing$), for any irreducible curve $D$ on $Z$ which is not a component of $E$.

\medskip

\textit{Case~1. Dihedral singularities.}

\medskip

As we saw in Section~\ref{section:minimal-symplectic-fillings}, the rational surface $Z$ with the divisor $E$ can be obtained from $\mathbb{CP}^2$ or $\mathbb{CP}^1 \times \mathbb{CP}^1$ by a sequence of blow-ups and blow-downs. Refer Figure~\ref{figure:sequence-dihedral}.

Let $D \subset Z$ be an irreducible curve which is not a component of $E$. If $D$ is a $(-1)$-curve then $D$ should intersect with $E$ because $W=Z-E$ is minimal by the assumption. So one may assume that $D$ is not a $(-1)$-curve. Then $D$ is transformed to an irreducible curve $D_0$ in $\mathbb{CP}^2$ or $\mathbb{CP}^1 \times \mathbb{CP}^1$.

\medskip

\textit{Case~1-1. $(Z,E)$ is obtained from $\mathbb{CP}^2$.}

\medskip

If the degree of $D_0$ is greater than one, then $D_0 \cdot L \ge 2$. Then it is easy to see that $D$ intersects with $E$ at $\ell \in Z$, the image of $L$ in $Z$. So we assume that $D_0$ is a line in $\mathbb{CP}^2$. If $q \notin D_0$, then $D$ intersects $E$ at $\ell$. So we assume that $q \in D_0$. We may choose $p \in C$ such that the tangent line at $p$ to $C$ does not pass through $q$. Then the line $D_0$ should intersect $C$ at some points which are not equal to $p, q$. So if we choose the blowing-up points $r$ (if any) on $C$ so that $p,q,r$ are not collinear, then $D$ intersects with $E$ on $Z$.

\medskip

\textit{Case~1-2. $(Z,E)$ is obtained from $\mathbb{CP}^1 \times \mathbb{CP}^1$.}

\medskip

If $D_0$ intersects $L$ in $\mathbb{CP}^1 \times \mathbb{CP}^1$, then $D$ intersects $E$ at $\ell \in E$. So assume that $D_0$ and $L$ are curves with the same type $(1,0)$ in $\mathbb{CP}^1 \times \mathbb{CP}^1$. It is easy to show that there are only finitely many points $r$ on the cuspidal curve $C$ such that a line of type $(1,0)$ is the tangent line to $C$ at $r \in C$. Therefore, if we choose $p \neq r$, then $D_0$ should intersect $C$ at some point $q \neq p$. Since we may assume that there is no point in $C$ other that $p$ where we blow up $C$ (Remark~\ref{remark:no-extra-blowup}), $D$ should intersect $E$.

\medskip

\textit{Case~2. Tetrahedral, octahedral, icosahedral singularities of type $(3,2)$.}

\medskip

In order to obtain the rational surface $Z_1$ from $\mathbb{CP}^2$ in Figure~\ref{figure:sequence-TOI-32}, we blow up the cuspidal curve $C$ at the point $p$ (including infinitely near points over $p$), and, if necessary, we blow up at some extra (smooth) points of $C$ different from $p$. So we divide the proof into two subcases.

\medskip

\textit{Case~2-1.} We need extra blow-ups to obtain $Z_1$ from $\mathbb{CP}^2$. Or $Z_1$ is obtained from $\mathbb{CP}^1 \times \mathbb{CP}^1$ (whether we need extra blow-ups or not).

\medskip

In case of $\mathbb{CP}^2$, let $M$ be the line passing through the cusp singularity of $C$ and one of the centers of the extra blow-ups. And, in case of $\mathbb{CP}^1 \times \mathbb{CP}^1$, let $M$ be a curve of type $(1,0)$ or $(0,1)$ passing through the cusp singularity. Then it is easy to show that the proper transform $m$ of $M$ in $Z$ is a $(-1)$-curve  intersecting the $(-3)$-curve inside $Z$. See Figure~\ref{figure:sequence-TOI-(32)-ampleness}. So, instead of following Figure~\ref{figure:sequence-TOI-32}, we first blow down the $(-1)$-curve $m$. We then apply the same procedure of blow-downs and blow-ups for dihedral singularities (described in Figure~\ref{figure:sequence-dihedral}) to our case as in Figure~\ref{figure:sequence-TOI-(32)-ampleness}. Then one can prove the ampleness of the compactifying divisor $E$ in $Z$ by the same method as in the case of dihedral singularities.

\begin{figure}
\begin{center}
\includegraphics{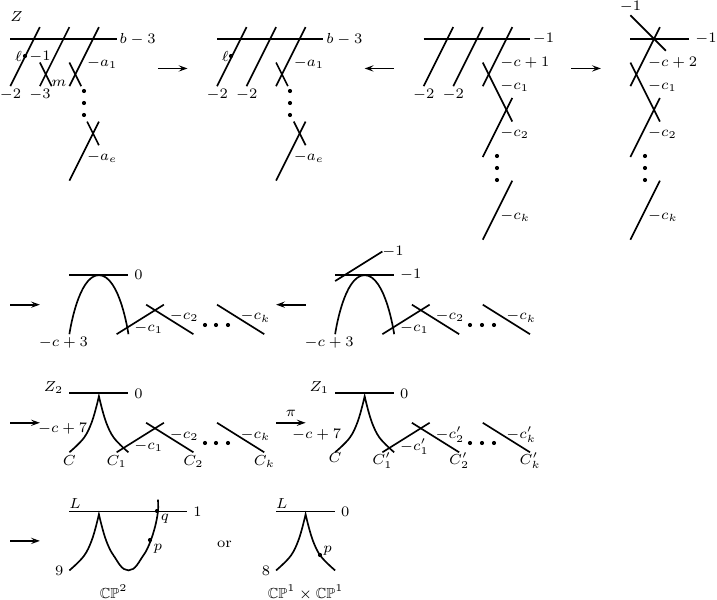}
\caption{New sequence of blow-ups and blow-downs for certain tetrahedral, octahedral, or icosahedral singularities of type $(3,2)$}
\label{figure:sequence-TOI-(32)-ampleness}
\end{center}
\end{figure}

\medskip

\textit{Case~2-2.} We do not need extra blow-ups to obtain $Z_1$ from $\mathbb{CP}^2$.

\medskip

Suppose that there is an irreducible curve $D$ in $Z$ which does not intersect $E$. Then its image $D_0$ in $\mathbb{CP}^2$ should be an irreducible curve which intersects the cuspidal curve $C$ only at $p \in C$. Assume that $D_0$ is a curve of degree $n$. Then the intersection multiplicity of $D_0$ and $C$ at $p$ is $3n$. But, if we choose a general point $p$ of $C$ so that $p$ is not a inflection point of the cuspidal curve $C$, then it is impossible that the intersection multiplicity of an irreducible plane curve of degree $n$ and the cuspidal curve at $p$ is $3n$. For instance, if we parametrize an affine part of the cuspidal curve $C$ by $\{(t^2, t^3) : t \in \mathbb{C}\}$ and if $p=(t_0^2, t_0^3)$ for some $t_0 \neq 0$, then, an irreducible curve $D_0 = \{(x,y) \in \mathbb{C}^2 : f(x,y) = 0\}$ of degree $n$ intersects with $C$ at $p$ with multiplicity $3n$ if and only if $f(t^2,t^3)=a(t-t_0)^{3n}$ for some $a \in \mathbb{C}$. But $f(t^2,t^3)$ does not have a linear term in $t$, while $a(t-t_0)^{3n}$ has a linear term in $t$, which is a contradiction.

Therefore any irreducible curve $D_0$ in $\mathbb{CP}^2$ should intersect the cuspidal curve $C$ at a point different from $p$. Then its image $D$ in $Z$ intersects the compactifying divisor $E$.

\medskip

\textit{Case~3. Tetrahedral, octahedral, icosahedral singularities of type $(3,1)$.}

\medskip

Instead of following the sequence of blowing ups and blowing downs described in Figure~\ref{figure:sequence-TOI-31}, we blow up and blow down as described in Figure~\ref{figure:sequence-TOI-(31)-ampleness}. Then the divisor $E \subset Z$ is the total transform of a line $L$ and a smooth cubic curve in $\mathbb{CP}^2$ intersecting as in Figure~\ref{figure:sequence-TOI-(31)-ampleness}. Since every irreducible curve in $\mathbb{CP}^2$ intersects with $L$, the assertion follows.\qedhere

\begin{figure}[tp]
\centering
\includegraphics{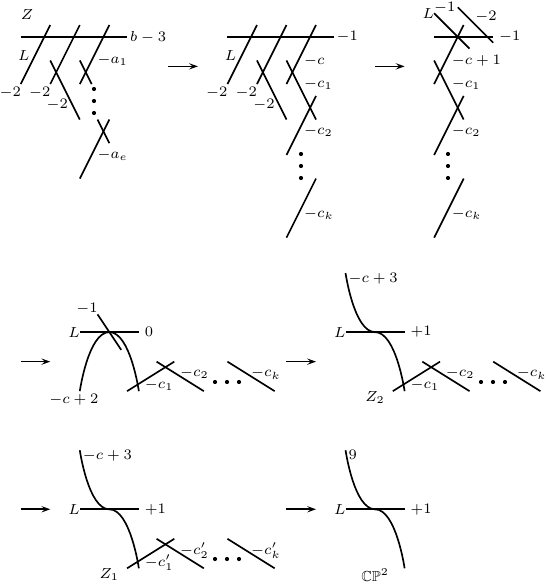}
\caption{New sequence of blow-ups and blow-downs for tetrahedral, octahedral, icosahedral singularities of type $(3,1)$}
\label{figure:sequence-TOI-(31)-ampleness}
\end{figure}
\end{proof}

According to the result on the one-to-one correspondence between $P$-resolutions and the irreducible components of the reduced semi-universal deformation space in KSB~\cite{Kollar-Shepherd-Barron-1988}, any Milnor fibers of a quotient surface singularity can be obtained by smoothing the corresponding $P$-resolution. Combined with Theorem~\ref{theorem:fillings-diffeomorphic-to-Milnors}:

\begin{theorem}\label{theorem:minimal-symplectic-filling-maximal-resolution}
Any minimal symplectic filling of a quotient surface singularity is obtained by a sequence of rational blow-downs from its unique maximal resolution.
\end{theorem}

\begin{proof}
By Theorem~\ref{theorem:fillings-diffeomorphic-to-Milnors}, every minimal symplectic filling is diffeomorphic to a Milnor fiber. On the other hand, by Corollary~\ref{corollary:Milnor-fiber=Rational-blow-down} above, any Milnor fiber is topologically a rationally blown-down $P$-resolution. Therefore the assertion follows.
\end{proof}

\begin{remark}
As mentioned in Introduction, Bhupal--Ozbagci~\cite{Bhupal-Ozbagci-2013} also proves a similar result for cyclic quotient surface singularities by using a Lefschetz fibration technique.
\end{remark}

\begin{corollary}\label{corollary:Milnor-1-1-minimal-symplectic}
For any quotient surface singularities $(X,0)$, there is a one-to-one correspondence (up to diffeomorphism type) between minimal symplectic fillings and Milnor fibers of irreducible components of the versal deformation space of $(X,0)$
\end{corollary}

\begin{remark}
The number of $P$-resolutions of a given quotient surface singularity given in Stevens~\cite{Stevens-1993} and that of minimal symplectic fillings of the singularity in Bhupal--Ono~\cite{Bhupal-Ono-2012} coincide; (cf.~Remarks~\ref{remark:erratum-Bhypal-Ono}, \ref{remark:erratum-Stevens}). So, combined with the classification of minimal symplectic fillings in Theorem~\ref{theorem:diffeomorphism-type}, one-to-one correspondence in Corollary~\ref{corollary:Milnor-1-1-minimal-symplectic} provides another proof of Theorem~\ref{theorem:components-Def(X)} on the fact that two Milnor fibers associated to different components of $\Def(X)$ are non-diffeomorphic given, which, however, depends heavily on the lists of Stevens~\cite{Stevens-1993} and Bhupal--Ono~\cite{Bhupal-Ono-2012}.
\end{remark}

\begin{remark}\label{remark:Ohta-Ono}
Ohta and Ono prove similar results not only for ADE singularities but also simple elliptic singularities. They show in \cite{Ohta-Ono-2005} that the diffeomorphism type of a minimal symplectic filling $W$ of an ADE singularity is unique. So $W$ is diffeomorphic to its unique Milnor fiber, hence to its minimal resolution. Furthermore, for any simple elliptic singularity (which is not a quotient surface singularity), they also show in \cite{Ohta-Ono-2003} that a minimal symplectic filling $W$ is diffeomorphic to its Milnor fiber or its minimal resolution. Hence it would be an interesting problem to characterize normal surface singularities for which such correspondence holds.
\end{remark}



\begin{thebibliography}{99}
\bibitem{Arndt-1988} J. Arndt, \textit{Verselle deformationen zyklischer quotientensingularit{\"a}ten}. Dissertation Universit{\"a}t Hamburg, 1988.

\bibitem{BHPV-2004} W. P. Barth, K. Hulek, C. A. M. Peters, A. Van de Ven, \textit{Compact complex surfaces}. Ergebnisse der Mathematik und ihrer Grenzgebiete. 3. Folge., second edition, vol. 4, Springer-Verlag, Berlin, 2004.

\bibitem{Behnke-Christophersen-1994} K. Behnke, J. A. Christophersen, \textit{$M$-resolutions and deformations of quotient surface singularities}. Amer. J. Math. 116 (1994),  no. 4, 881--903.

\bibitem{Bhupal-Ono-2012} M. Bhupal, K. Ono, \textit{Symplectic fillings of links of quotient surface singularities}. Nagoya Math. J. \textbf{207} (2012), 1--45.

\bibitem{Bhupal-Ono-2015} M. Bhupal, K. Ono, \textit{A personal communication}.

\bibitem{Bhupal-Ono-2017} M. Bhupal, K. Ono, \textit{Symplectic fillings of links of quotient surface singularities – Corrigendum}. Nagoya Math. J. \textbf{225} (2017), 207--212.

\bibitem{Bhupal-Ozbagci-2013} M. Bhupak, B. Ozbagci, \textit{Symplectic fillings of lens spaces as Lefschetz fibrations}. J. Eur. Math. Soc. \textbf{18} (2016), no. 7, 1515--1535.

\bibitem{Birman-Rubinstein-1984} J. S. Birman, J. H. Rubinstein, \textit{One-sided Heegaard splittings and homeotopy groups of some 3-manifolds}. Proc. London Math. Soc. (3) \textbf{49} (1984), no. 3, 517--536.

\bibitem{Boileau-Otal-1991} M. Boileau, J.-P. Otal, \textit{Scindements de Heegaard et groupe des hom\'eotopies des petites vari\'et\'es de Seifert}. Invent. Math. \textbf{106} (1991),  no. 1, 85--107.

\bibitem{Bonahon-1983} F. Bonahon, \textit{Diff{\'e}otopies des espaces lenticulaires}. Topology \textbf{22} (1983), no. 3, 305--314.

\bibitem{Christophersen-1991} J. A. Christophersen, \textit{On the components and discriminant of the versal base space of cyclic quotient surface singularities}.  Singularity theory and its applications, Part I (Coventry, 1988/1989),  81--92, Lecture Notes in Math. \textbf{1462}, Springer, Berlin, 1991.

\bibitem{deJong-vanStraten-1998} T. de Jong, D. van Straten, \textit{Deformation theory of sandwiched singularities}. Duke Math. J.  \textbf{95} (1998),  no. 3, 451--522.

\bibitem{Fintushel-Stern-1997} R. Fintushel, R. Stern, \textit{Rational blowdowns of smooth 4-manifolds}. J. Differential Geom. \textbf{46}  (1997), no. 2, 181--235.

\bibitem{Flenner-Zaidenberg-1994} H. Flenner, M. Zaidenberg, \textit{$\mathbb{Q}$-acyclic surfaces and their deformations}. Contemp. Math. \textbf{162} (1994), 143--208.

\bibitem{Fowler-2014} J. Fowler, \textit{Rational homology disk smoothing     components of weighted-homogeneous surface singularities}. PhD Thesis 2014.

\bibitem{Gay-Stipsicz-2007} D. T. Gay, A. I. Stipsicz, \textit{Symplectic rational blow-down along Seifert fibered 3-manifolds}. Int. Math. Res. Not. IMRN 2007, no. 22, Art. ID rnm084, 20 pp.

\bibitem{Gay-Stipzicz-2009} D. T. Gay, A. I. Stipsicz, \textit{Symplectic surgeries and normal surface singularities}. Algebr. Geom. Topol. \textbf{9} (2009), no. 4, 2203--2223.

\bibitem{Gorsky-Nemethi-2014} E. Gorsky, A. N{\'e}methi, \textit{Links of plane curve singularities are $L$-space links}. Algebr. Geom. Topol. \textbf{16} (2016), no. 4, 1905--1912.

\bibitem{Hacking-Tevelev-Urzua-2013} P. Hacking, J. Tevelev, G. Urz{\'u}a, \textit{Flipping surfaces}. J. Algebraic Geom. \textbf{26} (2017), no. 2, 279--345.

\bibitem{HJS-2018} B. Han, J. Jeon, D. Shin, \textit{Invatiants of deformations of quotient surface singularities}, preprint.

\bibitem{Jankins-Neumann-1983} M. Jankins, W. D. Neumann, \textit{Lectures on Seifert manifolds}. Brandeis Lecture Notes \textbf{2}. Brandeis University, 1983.

\bibitem{Kollar-Mori-1992} J. Koll{\'a}r, S. Mori, \textit{Classification of three-dimensional flips}. J. Amer. Math. Soc. \textbf{5} (1992),  no. 3, 533--703.

\bibitem{Kollar-Mori-1998} J. Koll{\'a}r, S. Mori, \textit{Birational geometry of algebraic varieties}. Cambridge Tracts in Mathematics, \textbf{134}. Cambridge University Press, Cambridge, 1998.

\bibitem{Kollar-Shepherd-Barron-1988} J. Koll{\'a}r, N. I. Shepherd-Barron, \textit{Threefolds and deformations of surface singularities}. Invent. Math. \textbf{91} (1988), no. 2, 299--338.

\bibitem{Lee-Park-K^2=2} Y. Lee, J. Park. \textit{A simply connected surface of general type with $p_g=0$ and $K^2=2$}. Invent. Math. \textbf{170} (2007), 483--505.

\bibitem{Lisca-2008} P. Lisca, \textit{On symplectic fillings of lens spaces}. Trans. Amer. Math. Soc. \textbf{360} (2008),  no. 2, 765--799.

\bibitem{Looijenga-1984} E. Looijenga. \textit{Isolated singular points on complete intersections}. London Mathematical Society Lecture Note Series \textbf{77}. Cambridge University Press, Cambridge, 1984.

\bibitem{McDuff-1990} D. McDuff, \textit{The structure of rational and ruled symplectic 4-manifolds}. J. Amer. Math. Soc. \textbf{3} (1990), no. 3, 679--712.

\bibitem{Mori-1988} S. Mori, \textit{Flip theorem and the existence of minimal models for 3-folds}. J. Amer. Math. Soc. \textbf{1} (1988),  no. 1, 117--253.

\bibitem{Mori-2002} S. Mori, \textit{On semistable extremal neighborhoods}. Higher dimensional birational geometry (Kyoto 1997), Adv. Stud. Pure Math. 35, Math. Soc. Japan, Tokyo, 157--184 (2002).

\bibitem{Nemethi-2005} A. N{\'e}methi, \textit{On the Ozsv{\'a}th-Szab{\'o} invariant of negative definite plumbed 3-manifolds}. Geom. Topol. \textbf{9} (2005), 991--1042.

\bibitem{Nemethi-PPampu-2010} A. N{\'e}methi, P. Popescu-Pampu, \textit{On the Milnor fibres of cyclic quotient surface singularities}. Proc. Lond. Math. Soc. (3) \textbf{101} (2010),  no. 2, 554--588.

\bibitem{Ohta-Ono-2003} H. Ohta, K. Ono, \textit{Symplectic fillings of the link of simple elliptic singularities}. J. Reine Angew. Math. \textbf{565} (2003), 183--205.

\bibitem{Ohta-Ono-2005} H. Ohta, K. Ono, \textit{Simple singularities and symplectic fillings}. J. Differential Geom. \textbf{69} (2005),  no. 1, 1--42.

\bibitem{Ohta-Ono-2005-II} H. Ohta, K. Ono, \textit{Symplectic 4-manifolds containing singular rational curves with $(2,3)$-cusp}. Singularit{\'e}s Franco-Japonaises, 233--241, S{\'e}min. Congr., \textbf{10}, Soc. Math. France, Paris, 2005.

\bibitem{Orlik-1972} P. Orlik, \textit{Seifert manifolds}. Lecture Notes in Mathematics, Vol. \textbf{291}. Springer-Verlag, 1972.

\bibitem{Orlik-Wagreich-1971} P. Orlik, P. Wagreich, \textit{Isolated singularities of algebraic surfaces with $\mathbb{C}^{\ast}$ action}. Ann. of Math. (2) \textbf{93} (1971), 205--228.

\bibitem{HPark-Stipsicz-2014} H. Park, A. I. Stipsicz, \textit{Smoothings of singularities and symplectic surgery}. J. Symplectic Geom. \textbf{12}  (2014),  no. 3, 585--597.

\bibitem{JPark-1977} J. Park, \textit{Seiberg-Witten invariants of generalised rational blow-downs}. Bull. Austral. Math. Soc. \textbf{56} (1997),  no. 3, 363--384.

\bibitem{Pinkham-1977} H. Pinkham, \textit{Normal surface singularities with $C^*$ action}. Math. Ann. \textbf{227} (1977), no. 2, 183--193.

\bibitem{Pinkham-1978} H. Pinkham, \textit{Deformations of normal surface singularities with $C^*$ action}. Math. Ann. \textbf{232} (1978), no. 1, 65--84.

\bibitem{Riemenschneider-1974} O. Riemenschneider, \textit{Deformationen von Quotientensingularit{\"a}ten (nach zyklischen Gruppen)}. Math. Ann. \textbf{209} (1974), 211--248.

\bibitem{Riemenschneider-1981} O. Riemenschneider, \textit{Zweidimensionale Quotientensingularitäten: Gleichungen und Syzygien}. Arch. Math. (Basel)  \textbf{37} (1981), no. 5, 406--417.

\bibitem{Rubinstein-1979} J. H. Rubinstein, \textit{On 3-manifolds that have finite fundamental group and contain Klein bottles}. Trans. Amer. Math. Soc. \textbf{251} (1979), 129--137.

\bibitem{Stevens-1991} J. Stevens, \textit{On the versal deformation of cyclic quotient surface singularities}. Singularity theory and its applications, Part I (Coventry, 1988/1989), 302--319, Lecture Notes in Math. \textbf{1462}, Springer, Berlin, 1991.

\bibitem{Stevens-1993} J. Stevens, \textit{Partial resolutions of quotient surface singularities}. Manuscripta Math. \textbf{79} (1993),  no. 1, 7--11.

\bibitem{Stevens-2015} J. Stevens, \textit{A personal communication}.

\bibitem{Stipzicz-Szabo-Wahl-2008} A. I. Stipsicz, Z. Szab{\'o}, J. Wahl, \textit{Rational blowdowns and smoothings of surface singularities}. J. Topol. \textbf{1} (2008),  no. 2, 477--517.

\bibitem{Urzua-2013} G. Urz{\'u}a, \textit{Identifying neighbors of stable surfaces}. Ann. Sc. Norm. Super. Pisa Cl. Sci. (5) \textbf{16} (2016), no. 4, 1093--1122.

\bibitem{Urzua-2013b} G. Urz{\'u}a, \textit{$\mathbb{Q}$-Gorenstein smoothings of surfaces and degenerations of curves}. Rend. Semin. Mat. Univ. Padova \textbf{136} (2016), 111--136.

\bibitem{Wahl-1976} J. Wahl, \textit{Equisingular deformations of normal surface singularities. I}. Ann. of Math. (2) \textbf{104} (1976), no. 2, 325--356.

\bibitem{Wahl-1981} J. Wahl, \textit{Smoothing of normal surface singularities}. Topology \textbf{20} (1981), 219--246.

\bibitem{Waldhausen-1968} F. Waldhausen, \textit{On irreducible $3$-manifolds which are sufficiently large}. Ann. of Math. (2) \textbf{87} (1968), 56--88.
\end{thebibliography}
\end{document}